\documentclass{article}

\usepackage[utf8]{inputenc}
\usepackage{amsmath, amssymb, amsthm, bm, mathrsfs, nicefrac, gensymb}
\usepackage[top=1.5cm, bottom=1.5cm, left=2cm, right=2cm]{geometry}
\usepackage{csquotes, authblk, cases, enumitem}
\usepackage{graphicx, float, caption, subcaption, adjustbox}
\usepackage[hidelinks]{hyperref}
\usepackage{cleveref}
\usepackage{xparse}
\usepackage[numbers,sort]{natbib}
\usepackage[section]{algorithm}
\usepackage{algpseudocode}
\usepackage{array, tabularx, xcolor, colortbl}
\usepackage{color, xcolor}
\usepackage[para]{footmisc}
\usepackage[normalem]{ulem}
\counterwithin{figure}{section}
\counterwithin{table}{section}
\numberwithin{equation}{section}
\definecolor{darkred}{rgb}{0.8,0,0}
\DeclareMathOperator{\diag}{\operatorname{diag}}

\newcommand{\dd}{\mathrm{d}}

\NewDocumentCommand{\figref}{m g}{\IfNoValueTF {#2} {\Cref{#1}}{\Cref{#1}.#2}}
\newcommand{\joinR}{\hspace{-.1em}}
\newcommand{\RomanI}{\textrm{I}}
\newcommand{\RomanII}{\mbox{\RomanI\joinR\RomanI}}
\newenvironment{frcseries}{\fontfamily{frc}\selectfont}{}
\theoremstyle{definition}
\newtheorem{definition}{Definition}[section]
\newtheorem{remark}[definition]{Remark}
\newtheorem{example}[definition]{Example}

\newtheorem{lemma}[definition]{Lemma}

\title{Stabilization techniques for immersogeometric analysis of plate and shell problems in explicit dynamics}
\author[1]{Giuliano Guarino \thanks{giuliano.guarino@epfl.ch}}
\author[1]{Yannis Voet \thanks{yannis.voet@epfl.ch}}
\author[1]{Pablo Antolin \thanks{pablo.antolin@epfl.ch}}
\author[1]{Annalisa Buffa \thanks{annalisa.buffa@epfl.ch}}
\affil[1]{\small MNS, Institute of Mathematics, École polytechnique fédérale de Lausanne, Station 8, CH-1015 Lausanne, Switzerland}
\date{\today}

\begin{document}

\maketitle

\begin{abstract}
Finite element plate and shell formulations are ubiquitous in structural analysis for modeling all kinds of slender structures, both for static and dynamic analyses. The latter are particularly challenging as the high order nature of the underlying partial differential equations and the slenderness of the structures all impose a stringent constraint on the critical time step in explicit dynamics. Unfortunately, badly cut elements in immersed finite element discretizations further aggravate the issue. While lumping the mass matrix often increases the critical time step, it might also trigger spurious oscillations in the approximate solution thereby compromising the numerical solution. In this article, we extend our previous work in \cite{voet2025stabilization} to allow stable immersogeometric analysis of plate and shell problems with lumped mass matrices. This technique is based on polynomial extensions and restores a level of accuracy comparable to boundary-fitted discretizations.

\end{abstract}

\section{Introduction}
Plate and shell structural elements are practically encountered in every modern structure, ranging from buildings, bridges, ships to aircraft fuselage. Engineering design commonly relies on advanced finite element software built from classical plate and shell mathematical models that were established more than a century ago. The intrinsically small thickness characterizing plates and shells often allows simplifying general 3D elasticity equations by assuming the displacement and/or stress fields satisfy special properties. Therefore, complex three-dimensional bodies are essentially described through their mid-surface as two-dimensional ones. Contrary to standard elasticity, the models derived either feature different primary unknowns (e.g., displacements and rotations) or high order derivatives that naturally impose greater regularity on the solution field. Unfortunately, building smooth approximation spaces from classical $C^0$ finite elements is rather cumbersome and despite many attempts, none were fully satisfactory \cite{hughes2012finite}. Isogeometric analysis, on the contrary, naturally fulfills the smoothness requirement and is ideally suited for high order formulations.

Isogeometric analysis (IGA) is a spline-based finite element method, which was originally designed with the intention of bridging the gap between design and analysis by using a single model for both \cite{hughes2005isogeometric,cottrell2009isogeometric}. The method came along with many other advantages. Most notably, spline functions, which are the primary technology in Computer-Aided-Design (CAD), may completely eliminate the geometry discretization error that characterizes standard finite element discretizations. Moreover, smooth spline functions have better approximation properties than classical $C^0$ Lagrange functions \cite{bazilevs2006isogeometric,bressan2019approximation,sande2020explicit}, which translates into greater accuracy per degree of freedom. Finally, there exist positive, compactly supported and stable bases, ideal for numerical computations. Therefore, isogeometric analysis is an outstanding candidate for high order finite element discretizations and held its promises, as testified by its numerous successful applications \cite{cottrell2006isogeometric,borden2014higher,morganti2015patient}, particularly for plate and shell problems \cite{kiendl2009isogeometric,benson2010isogeometric}.

However, for transient processes as in structural dynamics, a time discretization must complement the spatial one. In principle, isogeometric analysis could be used for both and is actively being investigated as an instance of space-time discretizations \cite{fraschini2024unconditionally}. In practice though, practitioners still mostly resort to standard explicit time integration techniques based on finite differences. One of the reasons for favoring explicit methods over implicit ones stems from the physical constraint on the step size. Indeed, highly oscillatory (nonlinear) processes already inherently limit the step size, independently of the stability constraint. However, other less oscillatory processes in the realm of linear structural dynamics could benefit from greater step sizes and the stability constraint might then get in the way. This stability constraint, commonly referred to as the Courant–Friedrichs–Lewy (CFL) condition, is inversely proportional to the largest discrete frequency of the system ($\omega_n$) and states that
\begin{equation}
\label{eq: CFL_condition}
    \Delta t \leq \Delta t_c = \frac{C}{\omega_n}
\end{equation}
where $\Delta t_c$ is the critical time step and $C$ is a constant that depends on the actual method. Thus, inaccurate high-frequency \emph{optical branches} characterizing standard finite element discretizations often impose a stringent constraint on the step size. Although maximally smooth isogeometric analysis removes most of the optical branches, some \emph{outlier} frequencies persist \cite{cottrell2006isogeometric} and grow larger as the mesh gets finer and the degree increases \cite{gallistl2017stability}. Dampening the optical branches and removing outliers is a central topic in finite element and isogeometric analysis. While algebraic \emph{mass scaling} techniques have been developed for standard finite element methods \cite{olovsson2005selective,tkachuk2014local,voet2025theoretical}, isogeometric analysis (and tensorized finite element discretizations more generally) enjoy significantly more structure, which often forms the cornerstone of tailored outlier removal strategies. In this context, common techniques include spline reparameterizations \cite{cottrell2006isogeometric}, weak \cite{deng2021boundary} or strong \cite{hiemstra2021removal,manni2022application} imposition of additional boundary conditions (also related to $n$-width optimal spline spaces) and eigenvalue deflation techniques \cite{voet2025mass}. Most of these methods are both applicable to second and fourth order problems.

Another great reason for employing explicit methods comes from \emph{mass lumping}. Indeed, explicit time integration of undamped dynamical systems leads to solving a linear system with the mass matrix at each time step. Instead of solving them ``exactly'', practitioners commonly substitute the (consistent) mass matrix with an easily invertible (preferably diagonal) matrix, a process that became known as mass lumping. Its most successful instance is undoubtedly for spectral element discretizations that consist in choosing as finite element nodes the Gauss-Lobatto quadrature nodes. Integrating the mass matrix with the same quadrature rule then naturally produces accurate diagonal lumped mass matrices, a method nowadays commonly known as nodal quadrature \cite{fried1975finite,cohen1994higher}. In this context, it may also coincide with the row-sum \cite{hughes2012finite} and diagonal scaling (or HRZ) techniques \cite{hinton1976note}, two other well-known algebraic lumping strategies \cite{duczek2019mass}. Moreover, mass lumping is often praised for increasing the critical time step, a property formally proved for the row-sum and extended to classes of (block) banded matrices in isogeometric analysis \cite{voet2023mathematical,voet2025mass}. Unfortunately, the nodal quadrature method is built around interpolatory bases and tensor product elements. Choosing a suitable lumping technique for more general bases and elements is far less obvious: while the row-sum technique may produce indefinite lumped mass matrices, the diagonal scaling method is often rather inaccurate \cite{duczek2019mass,eisentrager2024eigenvalue}. Although the positivity of the B-spline basis in isogeometric analysis is often seen as an advantage for guaranteeing positive definite lumped mass matrices, the row-sum technique in this context is still only second order accurate, independently of the spline order \cite{cottrell2006isogeometric}. 

Intensive research is currently underway to improve its accuracy and several paths are being explored. Most notably, several authors have suggested applying the row-sum technique in a Petrov-Galerkin framework by choosing (approximate) dual functions as test functions. Although the idea, whose origins date back to Cottrell \cite{cottrell2009isogeometric}, initially did not attract much attention, it was taken up again recently in \cite{anitescu2019isogeometric} with promising results. Since then, there has been a surge of interest in approximate dual functions \cite{nguyen2023towards, hiemstra2025higher, held2024efficient, nguyen2023higher}. Unfortunately, such functions come with all sorts of difficulties. Firstly, (approximate) duality only holds in the parametric domain. Although the rescaling suggested in \cite{nguyen2023towards} ensures this property also holds in the physical domain, it complicates the assembly of the stiffness matrix. Secondly, as for any discretization involving dual functions, the treatment of essential boundary conditions requires ad hoc techniques \cite{hiemstra2025higher}. Finally, the method has only been tested in rather simple academic settings and its extension to trimmed geometries is not in sight.

In another line of research, other authors \cite{li2022significance,li2024interpolatory} have built interpolatory spline bases and combined them with the row-sum technique in an attempt to mimic the nodal quadrature method. The method is seamlessly applicable to wave type equations \cite{li2022significance} as well as plate models \cite{li2024interpolatory} but although it improves the convergence rate, it remains sub-optimal. Moreover, several critical issues remain unattended. In particular, interpolatory spline functions are globally supported, leading to dense system matrices unsuited for large scale simulations. Additionally, quadrature rules with positive weights are essential for constructing positive definite lumped mass matrices. Yet, the Greville quadrature, as proposed in \cite{li2022significance,li2024interpolatory}, does not fulfill this condition and as a matter of fact, negative weights are encountered for non-uniform knot spans or spline spaces of intermediate regularity \cite{zou2021galerkin}. This last issue completely undermines the applicability of the method and has been overlooked by the authors.

Due to the aforementioned shortcomings, the standard row-sum technique often remains the method of choice for complex industrial applications, particularly for trimmed isogeometric geometries \cite{leidinger2020explicit}. Trimming is ubiquitous in industrial CAD models and consists in discarding or merging portions of the domain delimited by spline curves or surfaces \cite{marussig2018review, de2023stability}. Although it grants virtually unlimited flexibility in geometric design, small trimmed elements often cause great difficulties in analysis, including instabilities \cite{buffa2020minimal}, ill-conditioning \cite{de2017condition} and arbitrarily small step sizes (for consistent mass formulations) \cite{stoter2023critical,bioli2025theoretical}. The latter prohibits explicit time integration altogether unless some form of stabilization is employed \cite{stoter2023critical} or alternative time integration schemes are sought \cite{fassbender2024implicit}. Mass lumping is sometimes a remarkably simple workaround, which has been widely promoted in the engineering community \cite{messmer2021isogeometric,messmer2022efficient}. Indeed, the critical time step  in this case may no longer depend on how badly trimmed elements are, provided the discretization is sufficiently smooth \cite{leidinger2020explicit}. This surprising finding was later confirmed in several independent studies \cite{messmer2021isogeometric,coradello2021accurate,messmer2022efficient,radtke2024analysis} and recently unraveled theoretically \cite{stoter2023critical,bioli2025theoretical}. Nevertheless, there is another side to the story: while the largest discrete eigenvalues constraining the critical time step may remain bounded, the smallest ones instead converge to zero and bring in spurious high frequency modes in the low-frequency spectrum \cite{radtke2024analysis,coradello2021accurate}. As recently shown in \cite{voet2025stabilization}, if these modes are unluckily activated, they may completely ruin the solution by triggering spurious high-frequency oscillations. As shown later in this contribution, those issues are not confined to wave type problems and may also arise for plate and shell models. Fortunately, a solution for the former may also be a solution for the latter and we will confirm it in this work. Tentative mitigation strategies were already listed in \cite{radtke2024analysis}, including incomplete lumping strategies. However, although they might restore some accuracy, conveniently inverting the partially lumped mass matrix is less obvious, particularly in dimension $d \geq 2$, when trimmed elements may constitute a significant portion of the boundary. Instead, in \cite{voet2025stabilization}, the authors stabilized the approximation space by extending polynomial segments of large elements into small neighboring elements. This technique both removed trimming-related outlier frequencies from the consistent mass and prevented spurious modes from surfacing in the low-frequency spectrum when it is lumped. After stabilization, the accuracy and CFL condition were in all aspects similar to boundary-fitted discretizations. Moreover, the stabilization method may be combined with advanced mass lumping techniques \cite{voet2023mathematical,voet2025mass} for further improving the accuracy of the row-sum. In this work, we extend the stabilization technique to plate and shell problems to produce similar results.

The outline for the rest of our article is as follows: After reviewing in \Cref{se: shell_theory} some elementary principles from differential geometry and shell theory, the variational problem is discretized in \Cref{se: discretization} with immersed isogeometric methods. The stabilization technique is then recalled in \Cref{se: stabilization} and numerical experiments follow in \Cref{se: numerical_experiments}. Although plate and shell models are prone to spurious oscillations for lumped mass formulations, they are cured with the same stabilization technique as for standard wave equations. Finally, \Cref{se: conclusion} summarizes our findings and concludes our article.

% mass lumping on trimmed geometries causes a staggering loss of of accuracy in the low-frequency spectrum recently attributed to small cut elements.

% However, for more general bases, choosing a suitable lumping technique is less obvious. 

% The nodal quadrature method, built around interpolatory bases and tensor product elements, may simply not apply. 

% Nowadays, mass lumping is standard practice and in the engineering literature, the mass matrix is often assumed lumped unless stated otherwise. 

% Their underlying smoothness provides many advantages in applications, particularly for high order plate and shell formulations that assume smooth solutions.

% Building smooth approximation spaces is a hurdle for classical finite element spaces.

\section{Shell theory}
\label{se: shell_theory}
In this section, we first recall some fundamental notions of differential geometry at the core of shell theories before formulating the problem in its weak (or variational) form. Throughout this section, we adopt some common conventions of shell theory whereby Greek indices take values $1,2$ and Latin indices takes values $1,2,3$. Additionally, we employ in this section (and this section only) Einstein's convention where summations are carried out over repeated indices. Several key references have guided the construction of this section, including \cite{ciarlet2005introduction,ciarlet2022mathematical,chapelle2010finite,bischoff2004models}, which the reader may consult for further details.

\subsection{Shell geometry}
The deformation of shells is always described with respect to their mid-surface, defined by a parameterized surface; i.e. a smooth injective map (or chart) $\bm{\mathcal{F}} \in C^2(S; \ \mathbb{R}^3$), which maps a point $\bm{\xi} = (\xi_1,\xi_2)$ in the parametric domain $S \subset \mathbb{R}^2$ to a point $\bm{x}=\bm{\mathcal{F}}(\bm{\xi})$ in the physical space $\mathbb{R}^3$. The parametric domain $S$ is assumed open, bounded and connected with a Lipschitz boundary $\partial S$. The mid-surface is then defined as $\widehat{S}=\bm{\mathcal{F}}(S)$, as shown in \Cref{fig: param_surface}.

\begin{figure}[H]	\centering
    \includegraphics[width=0.7\textwidth]{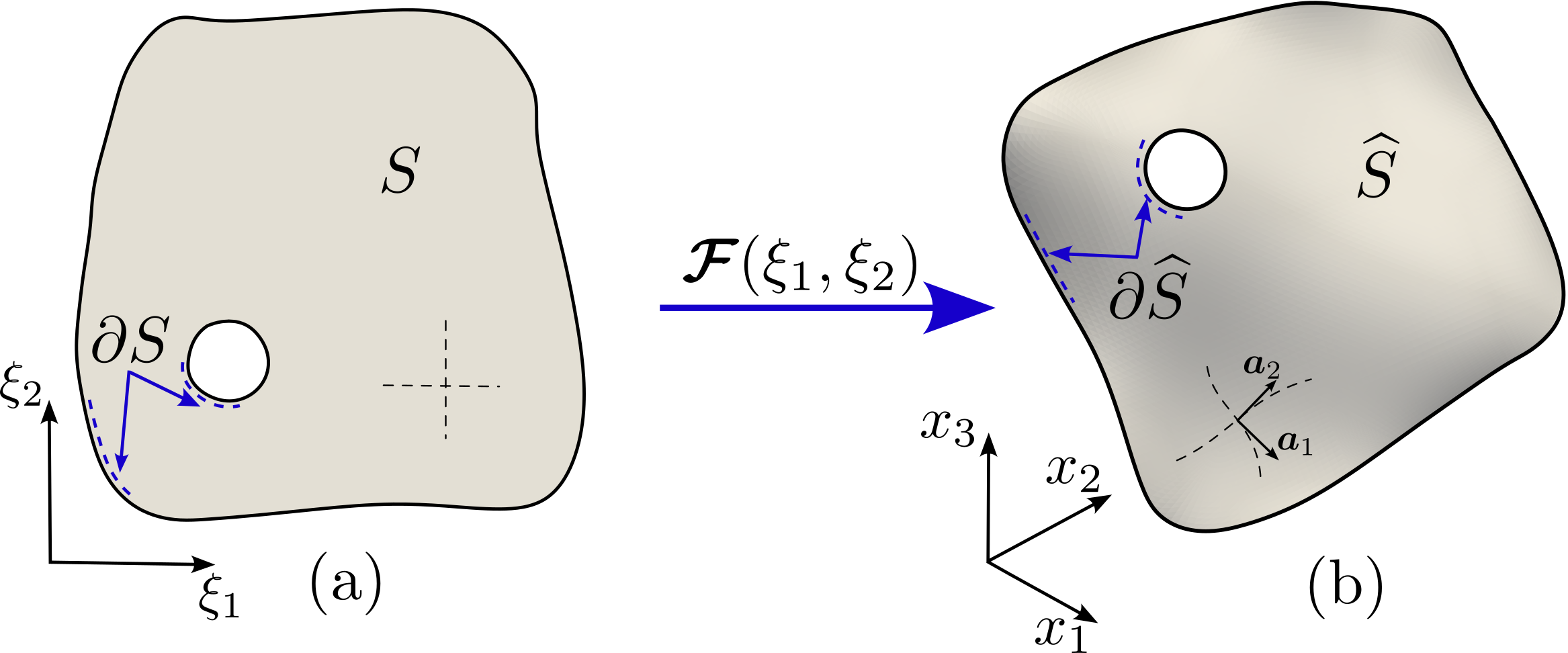}
    \caption{Parameterized surface.}
    \label{fig: param_surface}
\end{figure}

For complex shell structures, several maps may be necessary for defining the mid-surface but for the sake of the presentation, we only consider a single one. To each point $\bm{x}=\bm{\mathcal{F}}(\xi_1,\xi_2)$ on the mid-surface is associated a local basis (called the covariant basis) whose vectors are defined as 
\begin{equation}
    \bm{a}_\alpha = \bm{\mathcal{F}}_{,\alpha}
\end{equation}
where $\bm{\mathcal{F}}_{,\alpha}$ denotes differentiation of $\bm{\mathcal{F}}$ with respect to $\xi_{\alpha}$. The covariant basis vectors $\bm{a}_\alpha$ are tangent to the mid-surface and span the tangent plane (see \figref{fig: param_surface}{b}). A unit normal vector (or director) is computed from the covariant basis as
\begin{equation}
    \bm{a}_3 = \frac{\bm{a}_1 \times \bm{a}_2}{\|\bm{a}_1 \times \bm{a}_2\|}
\end{equation}
where $\times$ denotes the cross product and $\|\bullet\|$ the standard Euclidean norm. One of the most prominent objects in differential geometry is the metric tensor (or first fundamental form), which is nothing but a standard Gram matrix for the Euclidean inner product (denoted $\cdot$) on the tangent space. Its components in the covariant basis are $a_{\alpha\beta}=\bm{a}_\alpha\cdot\bm{a}_\beta$ and we let $a = \det(\nabla \bm{\mathcal{F}}^T \nabla \bm{\mathcal{F}})$ denote its determinant.

Shell structures may also experience deformations relative to its mid-surface. For this reason, one often introduces another more convenient basis, called the contravariant basis. Its basis vectors $\bm{a}^\alpha$ are dual to the covariant basis; i.e. they satisfy
\begin{equation*}
    \bm{a}^\alpha \cdot \bm{a}_\beta = \delta_{\alpha\beta}
\end{equation*}
where $\delta_{\alpha\beta}$ is the Kronecker delta symbol ($\delta_{\alpha\beta}=1$ if $\alpha=\beta$ and $0$ otherwise). The components of the metric tensor in the covariant and contravariant bases are related by
\begin{equation}
\label{metric_tensor_relation}
    [a^{\alpha\beta}] = [a_{\alpha\beta}]^{-1}
\end{equation}
and allow computing the vectors of the local contravariant basis as
\begin{equation}
    \bm{a}^\alpha= a^{\alpha\beta}\bm{a}_\beta.
\end{equation}
Shell theory distinguishes thick and thin shells based on their curvature and thickness. The components of the curvature tensor (or second fundamental form) are defined as
\begin{equation*}
    b_{\alpha\beta} = \bm{a}_{\alpha,\beta} \cdot \bm{a}_3.
\end{equation*}
The first and second fundamental forms induce the symmetric bilinear forms
\begin{equation*}
    \RomanI(\bm{u},\bm{v})= a_{\alpha\beta}u^{\alpha}v^{\beta} \quad \text{and} \quad \RomanII(\bm{u},\bm{v})= b_{\alpha\beta}u^{\alpha}v^{\beta}
\end{equation*}
where $\bm{u}= u^\alpha \bm{a}_\alpha$ and $\bm{v}= v^\beta \bm{a}_\beta$ are vectors in the tangent plane and $u^\alpha$ and $v^\beta$ are their contravariant components. Note that the components in the covariant basis are actually called ``contravariant'', while the components in the contravariant basis are called ``covariant''. From the first and second fundamental forms, we define the curvature along the cross section normal to $\bm{u}$ by the Rayleigh quotient
\begin{equation*}
    \kappa(\bm{u})=\frac{\RomanII(\bm{u},\bm{u})}{\RomanI(\bm{u},\bm{u})}.
\end{equation*}
Since $\RomanI$ and $\RomanII$ are symmetric and $\RomanI$ is positive definite, 
\begin{equation*}
    \kappa_{\min} \leq \kappa(\bm{u}) \leq \kappa_{\max}
\end{equation*}
where $\kappa_{\min}$ and $\kappa_{\max}$ are called the principal curvatures. The curvature may take positive or negative values but in shell theory, one is mostly interested in its largest magnitude or equivalently, in the smallest radius of curvature
\begin{equation*}
    R_{\min} = \frac{1}{\max |\kappa(\bm{u})|}.
\end{equation*}
The radius of curvature is one of the key quantities for defining the slenderness of shell structures. The other important geometrical quantity is the shell's thickness $\tau$ that now shapes it as a fully three-dimensional body. Denoting $T = \left(-\frac{\tau}{2}, \ \frac{\tau}{2}\right)$, we define the 3D parametric domain
\begin{equation*}
    \Omega = S \times T.
\end{equation*}
In principle, the shell's thickness may vary from one point to another (i.e., $\tau=\tau(\xi_1,\xi_2)$) but for simplicity we do not consider such cases here. The reference (or initial) configuration of the full three-dimensional body of the shell is then defined as $\widehat{\Omega}=\bm{\mathcal{G}}(\Omega)$, where $\bm{\mathcal{G}} \colon \Omega \to \mathbb{R}^3$ is the 3D chart
\begin{equation*}
    \bm{\mathcal{G}}(\xi_1,\xi_2,\xi_3) = \bm{\mathcal{F}}(\xi_1,\xi_2) + \xi_3 \bm{a}_3(\xi_1,\xi_2),
\end{equation*}
also sometimes called the canonical extension of $\bm{\mathcal{F}}$. Similarly to the surface, the covariant basis vectors for the volume $\bm{g}_i = \bm{\mathcal{G}}_{,i}$ are given by
\begin{subequations}
\label{eq: vol_covariant_basis}
\begin{align}
    \bm{g}_\alpha &= \bm{a}_\alpha + \xi_3 \bm{a}_{3,\alpha}, \\
    \bm{g}_3 &= \bm{a}_3,
\end{align}
\end{subequations}
and the corresponding contravariant basis vectors $\bm{g}^i$ and determinant $g=\det(\nabla \bm{\mathcal{G}}^T \nabla \bm{\mathcal{G}})$ are defined analogously as before. Given some characteristic length $L$ for the in-plane dimension of $\Omega$, the ratios
\begin{equation}
\label{eq: slenderness_ratio}
\eta=\frac{R_{\min}}{\tau} \quad \text{and} \quad \zeta=\frac{L}{\tau}
\end{equation}
are called \emph{slenderness ratios} and play an important role in (heuristically) classifying shell or plate theories. The former is commonly used for shells while the latter is customary for plates. To ensure all quantities are well-defined, one must assume that $\eta > 0.5$ \cite[][Fig. 2.6]{chapelle2010finite}. Depending on how large $\eta$ or $\zeta$ is, various assumptions on the strain and/or stress fields may apply and underpin different theories. Since plates are actually a particular case of shells, we will focus on the latter, which is much more general. Among the most common and well-studied theories are the Kirchhoff-Love and Reissner-Mindlin theories. The Kirchhoff-Love theory assumes that the vectors normal to the mid-surface remain normal after deformation and therefore neglects transverse shear deformations. This assumption applies to thin shells (e.g., $\eta \geq 20$ \cite{kiendl2009isogeometric}) and relates the deformation within the shell to the deformation of the mid-surface only. Although the model only features displacement unknowns, it also involves high order derivatives and imposes greater regularity requirements. As a matter of fact, finite element discretizations based on the Kirchhoff-Love theory require $C^1$ approximation spaces, a condition that is not easily fulfilled with classical (Lagrange) finite elements. This was historically one of the main reasons for adopting alternative shell theories in finite element codes \cite{kiendl2009isogeometric}. 

One very popular alternative is the Reissner-Mindlin theory, where transverse shear deformations are properly accounted for and is applicable to moderately thick shells (e.g., $\eta < 20$ \cite{kiendl2009isogeometric}). However, the model requires introducing additional unknowns (rotations) for relating the displacement of the mid-surface to the displacement through the thickness. Those additional unknowns avoid introducing high order derivatives, which has tangible consequences: contrary to the Kirchhoff-Love theory, finite element discretizations based on the Reissner-Mindlin theory only require classical $C^0$ approximation spaces and are therefore more widespread \cite{kiendl2009isogeometric}. For this reason, we adopt the Reissner-Mindlin theory in the rest of the article, while keeping in mind that the methods presented herein virtually apply to any shell formulation, including Kirchhoff-Love theory.

\subsection{Shell kinematics}
The variational formulation of shell problems originally stems from simplifications of the variational formulation of three-dimensional elasticity, albeit stated in curvilinear coordinates. Curvilinear coordinates are very natural in this setting and allow separating the in-plane contributions from the out-of-plane (or transverse) ones, which is essential for reducing the three-dimensional problem to a simpler two-dimensional one along the mid-surface. We merely state below the result of this transformation and refer interested readers to \cite[][Chapter 1]{ciarlet2022mathematical} for its derivation. In curvilinear coordinates, the three-dimensional variational form reads: find $\mathbf{u}(t) \in V = \{\mathbf{v} \in [H^1]^3 \colon \mathbf{v}|_{\partial \Omega_D} = 0\}$ such that
\begin{equation}
\label{eq: 3D_elasticity_curvilinear}
    \int_{\Omega} \rho \ddot{\mathbf{u}} \cdot \mathbf{v} \sqrt{g} \ \dd \Omega + \int_{\Omega} \bm{\epsilon}(\mathbf{u}) \colon \bm{C} \colon \bm{\epsilon}(\mathbf{v}) \sqrt{g} \ \dd \Omega = \int_{\Omega} \mathbf{f} \cdot \mathbf{v} \sqrt{g} \ \dd \Omega + \int_{\partial \Omega_N} \mathbf{h} \cdot \mathbf{v} \sqrt{g} \ \dd \partial \Omega \quad \forall \mathbf{v} \in V
\end{equation}
where the Dirichlet and Neumann boundaries ($\partial \Omega_D$ and $\partial \Omega_N$, respectively) form a disjoint partition of the boundary ($\partial \Omega = \partial \Omega_D \cup \partial \Omega_N$ and $\partial \Omega_D \cap \partial \Omega_N = \emptyset$) and we have assumed homogeneous Dirichlet boundary conditions for simplicity. In \eqref{eq: 3D_elasticity_curvilinear}, $\rho$ is the density, $\bm{\epsilon}$ is the linearized strain tensor (defined later), $\mathbf{f}$ and $\mathbf{h}$ are the volume force and ``surface traction'', respectively, and $\mathbf{v}$ is a test function. We must stress that the ``surface traction'' $\mathbf{h}$ differs from the expression one is accustomed to in standard elasticity and accounts for additional terms stemming from the change of variables. Its expression is derived in \cite[][Chapter 1]{ciarlet2022mathematical} but is not needed in the sequel. Finally, $\bm{C}$ is the fourth order elasticity tensor, whose contravariant components for a linear elastic homogeneous and isotropic material are given by
\begin{equation*}
    C^{ijkl} = \lambda g^{ij}g^{kl} + \mu(g^{ik}g^{jl} + g^{il}g^{jk}),
\end{equation*}
where $\lambda$ and $\mu$ are the Lamé constants. In this contribution, we limit ourselves to isotropic materials that cover the majority of industrial applications. For the constitutive relationship of laminates, which are also relevant in the context of shell structures, the interested reader may consult \cite{guarino2024interior,guarino2024immersed}.

Shell theories assume that the displacement field takes a specific form, which is another cause for diverging theories. Nevertheless, all shell theories share one common goal: when combined with ad hoc approximations, they reduce the problem to a two-dimensional one along the mid-surface. Our work focuses on classical shell theories that assume a unit vector normal to the mid-surface remains straight and unstretched after deformation. This assumption translates into expressing the displacement field $\mathbf{u}$ as
\begin{equation}
\label{eq: disp_field}
    \mathbf{u}(\xi_1,\xi_2,\xi_3,t)=\bm{u}(\xi_1,\xi_2,t)+\xi_3 \bm{\theta}(\xi_1,\xi_2,t)
\end{equation}
where $\bm{u}(\xi_1,\xi_2,t)$ and $\bm{\theta}(\xi_1,\xi_2,t)$ denote the displacement and rotation, respectively, at a point $\bm{x}=\bm{\mathcal{F}}(\xi_1,\xi_2)$ of the mid-surface at time $t$. We have deliberately chosen the same letter for the displacement field in the shell and along the mid-surface but different font styles. This convention is systematically employed in analogous expressions.

The Reissner-Mindlin kinematical assumption \cite[][p. 95]{chapelle2010finite} states that $\bm{a}_3 \cdot \bm{\theta}=0$, which implies that $\bm{\theta}$ lies in the tangent plane (i.e., it does not have any component along $\bm{a}_3$). Contrary to the Kirchhoff-Love theory, where the rotations are related to the displacements (derivatives), the Reissner-Mindlin theory features both displacements and rotations as primary unknowns. The displacements pertain to the mid-surface while the rotations relate the displacement of the mid-surface to the displacement through the thickness. Using the covariant coordinates $\theta_\alpha$ of the rotation vector is advantageous in this context since it reduces the number of unknowns. However, the lack of assumptions on the mid-surface's displacement field does not suggest a particularly attractive basis and we simply stick to the Cartesian basis; i.e., $\bm{u} = u_i \bm{e}_i$ and $\bm{\theta} = \theta_\alpha \bm{a}^\alpha$. For a given vector field $\bm{v}$ and a fixed basis $\mathcal{B}$, we will denote $[\bm{v}]_{\mathcal{B}}$ the component vector of $\bm{v}$ in the basis $\mathcal{B}$. Hence, denoting $\mathcal{E}=\{\bm{e}_1,\; \bm{e}_2,\; \bm{e}_3\}$ the Cartesian basis and $\mathcal{A}=\{\bm{a}^1,\; \bm{a}^2\}$ the surface's contravariant basis,
\begin{equation*}
    [\bm{u}]_{\mathcal{E}} = (u_1, u_2, u_3) \quad \text{and} \quad [\bm{\theta}]_{\mathcal{A}} = (\theta_1,\theta_2).
\end{equation*}

\begin{remark}
Multiple other forms of the displacement field have been suggested in the literature. One rather natural generalization of \eqref{eq: disp_field} consists in adding high order terms in $\xi_3$ to better model the stress distribution through the thickness. The interested reader may consult \cite{carrera2003theories, patton2021efficient, guarino2021high} for an example of high order shell theories. There exist numerous other shell theories, which arise from combining various geometrical and mechanical assumptions.
\end{remark}

As we will see, the special form of the displacement field \eqref{eq: disp_field} allows to reduce the general three-dimensional elasticity problem \eqref{eq: 3D_elasticity_curvilinear} to a simpler two-dimensional one. For understanding how, we must simplify the expression of the linearized strain tensor, whose covariant components are given by
\begin{equation}
\label{eq: covariant_comp}
    \epsilon_{ij} = \frac{1}{2}(\bm{g}_i \cdot \mathbf{u}_{,j} + \bm{g}_j \cdot \mathbf{u}_{,i}).
\end{equation}
The derivation of the foregoing expression (rarely detailed in most articles) is deferred to \Cref{app: covariant_comp_strain}. Given the derivatives of the displacement field
\begin{subequations}
\label{eq: disp_derivatives}
\begin{align}
    \mathbf{u}_{,\alpha} &= \bm{u}_{,\alpha} + \xi_3 \bm{\theta}_{,\alpha}, \\
    \mathbf{u}_{,3} &= \bm{\theta},
\end{align}
\end{subequations}
and the covariant basis vectors \eqref{eq: vol_covariant_basis}, the covariant components of the strain tensor \eqref{eq: covariant_comp} reduce to
\begin{subequations}
\label{eq: linearized_strains}
\begin{align}
    \epsilon_{\alpha\beta} &= \varepsilon_{\alpha \beta} + \xi_3 \kappa_{\alpha \beta} + \xi_3^2 \chi_{\alpha\beta}, \label{eq: membrane_strain}\\
    \epsilon_{3\alpha} &= \frac{\gamma_{\alpha}}{2}, \label{eq: bending_strain}\\
    \epsilon_{33} &= 0, \label{eq: shear_strain}
\end{align}
\end{subequations}
where
\begin{subequations}
\label{eq: strains_def}
\begin{align}
    2\varepsilon_{\alpha\beta} &= \bm{a}_\alpha\cdot\bm{u}_{,\beta}+\bm{a}_\beta\cdot\bm{u}_{,\alpha},\\
    2\kappa_{\alpha\beta} &= 
    (\theta_{\alpha,\beta}+\theta_{\beta,\alpha}) - (\bm{a}_{\alpha,\beta}\cdot\bm{\theta} + \bm{a}_{\beta,\alpha}\cdot\bm{\theta}) + (\bm{a}_{3,\alpha}\cdot\bm{u}_{,\beta} + \bm{a}_{3,\beta}\cdot\bm{u}_{,\alpha}),\\
    2\chi_{\alpha\beta} &= \bm{a}_{3,\alpha}\cdot \bm{\theta}_{,\beta} + \bm{a}_{3,\beta} \cdot \bm{\theta}_{,\alpha}, \\
    \gamma_{\alpha} &= \bm{a}_3\cdot\bm{u}_{,\alpha} + \theta_\alpha.
\end{align}
\end{subequations}
These expressions (called generalized strains) merely follow from substituting \eqref{eq: vol_covariant_basis} and \eqref{eq: disp_derivatives} into \eqref{eq: covariant_comp} and exploiting the fact that 
\begin{align*}
    0=(\bm{a}_3 \cdot \bm{\theta})_{,\alpha} &= \bm{a}_{3,\alpha} \cdot \bm{\theta} + \bm{a}_3 \cdot \bm{\theta}_{,\alpha}, \\
    \theta_{\alpha,\beta} = (\bm{a}_{\alpha} \cdot \bm{\theta})_{\beta} &= \bm{a}_{\alpha,\beta} \cdot \bm{\theta} + \bm{a}_{\alpha} \cdot \bm{\theta}_{,\beta},
\end{align*}
to simplify and rewrite the relations. In particular, combining the special form of the displacement field with the kinematic assumption already enforces the place strain assumption ($\epsilon_{33}=0$). The conflicting plane stress assumption $(\sigma_{33}=0)$ is later enforced in the constitutive equations (see \Cref{app: volume_to_surface}). The quadratic dependency on $\xi_3$ in \eqref{eq: membrane_strain} also later justifies neglecting $\chi_{\alpha\beta}$ for thin shells. The remaining quantities $\varepsilon_{\alpha\beta}$, $\kappa_{\alpha\beta}$ and $\gamma_{\alpha}$ are called membrane, bending and shear strains, respectively. Their actual expression is not very important in the context of this work; one must only note that they are independent of the transverse variable $\xi_3$.

\subsection{Two-dimensional variational problem}
The two-dimensional variational problem along the mid-surface stems from its full three-dimensional counterpart \eqref{eq: 3D_elasticity_curvilinear}, after substituting the expressions for the displacement and strain fields \eqref{eq: disp_field} and \eqref{eq: linearized_strains} and combining them with several geometrical approximations. The final result is reported in this section while its derivation is postponed to \Cref{app: volume_to_surface}. We have not found a satisfactory reference presenting a sufficiently concise and comprehensive treatment of the Reissner-Mindlin model. For completeness, we have summarized and regrouped the main steps in the appendix and included pointers to relevant literature for a more detailed discussion.

Denoting $\bm{\mathsf{u}}=(\bm{u},\bm{\theta})$, $\bm{\mathsf{v}}=(\bm{v},\bm{\phi})$ and $\mathsf{E} = \mathcal{E} \times \{\bm{0}\} \cup \{\bm{0}\} \times \mathcal{A}$ the basis, we seek $\bm{\mathsf{u}}(t) \in \mathsf{V}$, where
\begin{equation*}
    \mathsf{V} = \{\bm{\mathsf{v}} \colon [\bm{\mathsf{v}}]_{\mathsf{E}} \in [H^1]^3 \times [H^1]^2,\, \bm{\mathsf{v}}|_{\partial S_D} = \bm{0}\}
\end{equation*}
such that
\begin{equation}
\label{eq: variational_problem}
    b(\bm{\mathsf{u}}, \bm{\mathsf{v}})+a(\bm{\mathsf{u}}, \bm{\mathsf{v}}) = F(\bm{\mathsf{v}}) \qquad \forall \bm{\mathsf{v}} \in \mathsf{V}
\end{equation}
where
\begin{subequations}
\begin{align}
    a(\bm{\mathsf{u}}, \bm{\mathsf{v}}) &= \int_{S} (\bm{N}(\bm{\mathsf{u}}) \colon \bm{\varepsilon}(\bm{\mathsf{v}}) +  \bm{M}(\bm{\mathsf{u}}) \colon \bm{\kappa}(\bm{\mathsf{v}}) + \bm{Q}(\bm{\mathsf{u}}) \cdot \bm{\gamma}(\bm{\mathsf{v}})) \sqrt{a} \dd S, \label{eq: bilinear_form_a} \\
    b(\bm{\mathsf{u}}, \bm{\mathsf{v}}) &= \int_{S}(\ddot{\bm{\mathsf{u}}} \cdot W \bm{\mathsf{v}}) \sqrt{a} \dd S, \label{eq: bilinear_b} \\
    F(\bm{\mathsf{v}}) &= \int_{S} \bm{\mathsf{f}} \cdot \bm{\mathsf{v}} \sqrt{a} \dd S + \int_{\partial S_N} \bm{\mathsf{h}} \cdot \bm{\mathsf{v}} \sqrt{a} \dd \partial S, \label{eq: linear_form_F}
\end{align}
\end{subequations}
where the Dirichlet and Neumann boundaries ($\partial S_D$ and $\partial S_N$, respectively) form a disjoint partition of the shell's boundary ($\partial S = \partial S_D \cup \partial S_N$ and $\partial S_D \cap \partial S_N = \emptyset$). In \eqref{eq: bilinear_b}, $W=\diag(\rho_u I_3,\; \rho_\theta I_3)$ with $\rho_u=\tau\rho$ and  $\rho_\theta=\frac{\tau^3}{12}\rho$ and in \eqref{eq: linear_form_F}
\begin{equation*}
    \bm{\mathsf{f}} = (\bm{f}, \bm{m}) \quad \text{and} \quad \bm{\mathsf{h}} = (\bm{h}, \bm{n}) 
\end{equation*}
are the shell's ``external forces'' and ``boundary tractions'', whose components are defined in \Cref{app: volume_to_surface}. The tensors $\bm{N}$, $\bm{M}$ and $\bm{Q}$ appearing in \eqref{eq: bilinear_form_a}, called the generalized membrane, bending and shear stresses, respectively, are defined as
\begin{align}
    N^{\alpha\beta} = E_u^{\alpha\beta\gamma\delta}\varepsilon_{\gamma\delta}, \quad M^{\alpha\beta} = E_\theta^{\alpha\beta\gamma\delta}\kappa_{\gamma\delta} \quad \text{and} \quad Q^{\alpha}      &= \mu\alpha_s\tau a^{\alpha\beta }\gamma_\beta,
\end{align}
where $\alpha_s=\nicefrac{5}{6}$ is the shear correction factor and the fourth order tensors $E_u^{\alpha\beta\gamma\delta}=\tau E^{\alpha\beta\gamma\delta}$ and $E_\theta^{\alpha\beta\gamma\delta}=\frac{\tau^3}{12} E^{\alpha\beta\gamma\delta}$ are related to the shell elasticity tensor, whose contravariant components are
\begin{equation}
    E^{\alpha\beta\gamma\delta}=\frac{2 \lambda \mu}{\lambda+2\mu}a^{\alpha\beta} a^{\gamma\delta} + \mu(a^{\alpha\gamma} a^{\beta\delta}+a^{\alpha\delta} a^{\beta\gamma}).
\end{equation}
Verifying the well-posedness of this reduced variational problem falls outside the scope of this contribution. The interested reader may nevertheless consult \cite{ciarlet2022mathematical} for useful hints on the matter.

% where the operator $\colon$ denotes the contraction operation between tensors. 

% Shell theories differ in their treatment of the displacements along the thickness of the shell. Briefly summarize various theories. This work mostly focuses on the Reissner-Mindlin theory due to its less restrictive assumptions but the methods presented herein also apply to Kirchhoff-Love shells.

\section{Isogeometric discretization and mass lumping}
\label{se: discretization}
In finite element analysis, the components of an approximate solution to the variational problem \eqref{eq: variational_problem} are sought in a finite dimensional subspace. In isogeometric analysis, this subspace is a spline space; i.e., a space of piecewise polynomials $\mathcal{S}_{h}$, parameterized by a mesh size $h$, spline order $p$ and continuity $0 \leq k \leq p-1$. In principle, different components of the solution may be approximated in different spaces and the discretization parameters of a given space  may depend on the parametric direction but for simplifying the notation, we assume that the spaces are identical for each component and that the discretization parameters are uniform. Extending the analysis to mixed approximation spaces is straightforward, albeit more cumbersome. Furthermore, the solution of time-dependent PDEs is often approximated in two stages, by first discretizing in space and then in time. The approximate solution is called semi-discrete when it is only discretized in space while it is called fully discrete when it is discretized in both space and time. Although space-time discretization techniques (simultaneously discretizing both in space and in time) are gaining momentum, the former is still the most popular technique and the one employed in this work. Hence, denoting $m=\dim(\mathcal{S}_{h})$ the dimension of the spline space and $\bm{\xi}=(\xi_1,\xi_2)$ the two-dimensional curvilinear coordinates, for each time $t$, $\bm{\mathsf{u}}(t) \approx \bm{\mathsf{u}}_h(t) = (\bm{u}_h(t),\bm{\theta}_h(t))$ where $\bm{\mathsf{u}}_h(t) \in \mathsf{V}_h = \{\bm{\mathsf{v}}_h \colon [\bm{\mathsf{v}}_h]_{\mathsf{E}} \in [\mathcal{S}_{h}]^3 \times [\mathcal{S}_{h}]^2,\, \bm{\mathsf{v}}_h|_{\partial S_D} = \bm{0}\}$ such that
\begin{equation*}
    \bm{u}_h(\bm{\xi},t) = \sum_{j=1}^m \bm{u}_j(t) B_j(\bm{\xi}) \quad \text{and} \quad \bm{\theta}_h(\bm{\xi},t) = Q\sum_{j=1}^m \bm{\theta}_j(t) B_j(\bm{\xi})
\end{equation*}
where $B_j(\bm{\xi})$ are the B-spline basis functions, $\bm{u}_j(t) \in \mathbb{R}^3$ and $\bm{\theta}_j(t) \in \mathbb{R}^2$ are unknown coefficient vectors and $Q=[\bm{a}^1,\; \bm{a}^2] \in \mathbb{R}^{3 \times 2}$ is the matrix whose columns are the (surface) contravariant basis vectors. This matrix accounts for the fact that $\bm{\theta}_h$ is expressed in the contravariant basis. For trimmed isogeometric shells, the B-spline basis is typically constructed over an entire fictitious (non-trimmed) parametric domain $S_0 \supseteq S$ before retaining the functions whose support intersects the parametric domain $S$, defining the so-called active basis. This procedure is exemplarily illustrated in \Cref{fig: fictitious_domain} for the parameterized surface of \Cref{fig: param_surface}. The construction of the B-spline basis over the fictitious domain $S_0$ is fairly standard and the reader may consult \cite{hughes2005isogeometric,cottrell2009isogeometric} for the details.

\begin{figure}[H]	
    \centering
    \includegraphics[width=0.7\textwidth]{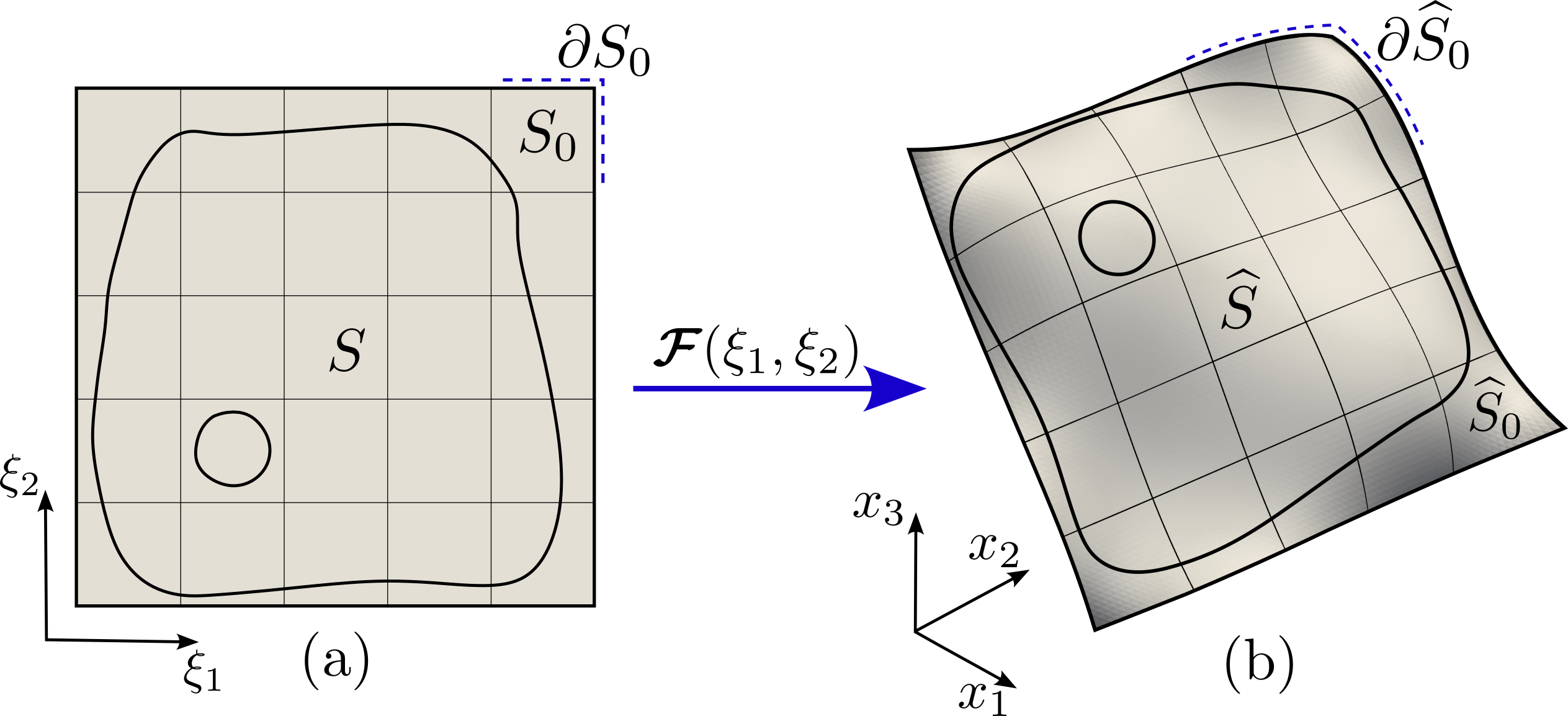}
    \caption{Fictitious parametric domain $S_0$ embedding $S$ and their images.}
    \label{fig: fictitious_domain}
\end{figure}

% Although this construction is convenient from an implementation point of view, it also causes great difficulties in the analysis by incorporating functions whose support may be arbitrarily small.

A standard Galerkin approximation of the variational problem now leads to solving the following second order system of ordinary differential equations (ODEs) \citep{hughes2012finite,quarteroni2009numerical}
\begin{align}
\label{eq: semi_discrete_pb}
\begin{split}
M\ddot{\bm{\textsf{u}}}(t) + K\bm{\textsf{u}}(t) &= \bm{\textsf{f}}(t) \qquad \text{for } t \in [t_0,\; t_1], \\
\bm{\textsf{u}}(t_0) &= \bm{\textsf{u}}_0,\\
\dot{\bm{\textsf{u}}}(t_0) &= \bm{\textsf{v}}_0.
\end{split}
\end{align}
where $\bm{\textsf{u}}(t)$ is the unknown coefficient vector containing all the displacement $\bm{u}_j(t)$ and rotation $\bm{\theta}_j(t)$ degrees of freedom, sorted componentwise. Hence, the mass matrix $M$ and stiffness matrix $K$ are $5 \times 5$ block matrices. While the stiffness matrix $K$ is generally fully block-coupled and its expression is rather complicated (and not needed here), the mass matrix $M$ is only partially coupled and is of much greater interest in the context of this work. Indeed, since $W$ is diagonal in \eqref{eq: bilinear_b}, $M$ is actually block diagonal. Its structure can be conveniently described through the \emph{direct sum} of matrices. Given square matrices $A_i$, their direct sum is the block diagonal matrix
\begin{equation*}
    \bigoplus_{i=1}^q A_i = \diag(A_1,\dots,A_q).
\end{equation*}
With this definition, the consistent mass matrix is expressed as $M = M_u \oplus M_\theta$, where
\begin{equation*}
    M_u = \bigoplus_{i=1}^3 \mathrm{M}_u \quad \text{and} \quad 
    M_\theta =
    \begin{pmatrix}
        \mathrm{M}^{11}_\theta & \mathrm{M}^{12}_\theta \\
        \mathrm{M}^{21}_\theta & \mathrm{M}^{22}_\theta
    \end{pmatrix}
\end{equation*}
are the mass matrices related to the displacement and rotation degrees of freedom, respectively. Their blocks $\mathrm{M}_u$ and $\mathrm{M}^{\alpha\beta}_\theta$ (denoted with straight letters) are defined as
\begin{equation}
\label{eq: Mu_M_theta}
    (\mathrm{M}_u)_{ij} = \int_S \rho_u B_i B_j \sqrt{a} \dd S, \quad \text{and} \quad (\mathrm{M}^{\alpha\beta}_\theta)_{ij} = \int_S \rho_\theta B_i B_j a^{\alpha\beta} \sqrt{a} \dd S, \qquad i,j=1,\dots,m.
\end{equation}
Clearly, $M$ (and its diagonal blocks $M_u$ and $M_\theta$) are symmetric positive definite. Those properties still hold in more general cases with non-diagonal matrices $W$, although $M$ would then have some block-coupling between displacement and rotation degrees of freedom, just as the stiffness matrix. However, although the B-spline basis functions are nonnegative, the mass matrix (and more specifically the off-diagonal blocks of $M_\theta$) may feature negative entries, due to the coefficients $a^{\alpha\beta}$ in \eqref{eq: Mu_M_theta}. This property sharply contracts with classical elasticity or plate models and might jeopardize standard mass lumping techniques. Indeed, when explicit time-stepping schemes are employed for approximately solving \eqref{eq: semi_discrete_pb}, the consistent mass $M$ is oftentimes substituted with an ad hoc (diagonal) lumped mass matrix. Various lumping techniques are known but a very popular choice is the row-sum technique \cite{hughes2012finite}. For a given $M \in \mathbb{R}^{n \times n}$, this technique is defined algebraically as
\begin{equation}
\label{eq: row_sum}
\mathcal{L}(M)=\diag(d_1,\dots,d_n)
\end{equation}
where $d_i=\sum_{j=1}^n m_{ij}$ for $i=1,\dots,n$. Practitioners directly substitute in \eqref{eq: semi_discrete_pb} the consistent mass $M$ by its lumped mass approximation $\mathcal{L}(M)$, which in our context becomes
\begin{equation*}
    \mathcal{L}(M) = \mathcal{L}(M_u) \oplus \mathcal{L}(M_\theta).
\end{equation*}
However, since $M_\theta$ is not nonnegative, the lumped mass matrix $\mathcal{L}(M_\theta)$ might end up being indefinite or even singular, both of which preclude simulations in explicit dynamics. In the next lemma, we present a sufficient condition guaranteeing positive definite lumped mass matrices. This condition depends on the angle between the covariant basis vectors, denoted $\theta(\bm{a}_1,\bm{a}_2)$.

\begin{lemma}
\label{lem: PD_cond}
If
\begin{equation}
\label{eq: sufficient_cond}
    \frac{\min_i \|\bm{a}_i\|_2}{\max_i \|\bm{a}_i\|_2} > \cos \theta(\bm{a}_1,\bm{a}_2)
\end{equation}
then $\mathcal{L}(M)$ is positive definite.
\end{lemma}
\begin{proof}
First note that
\begin{equation*}
    \mathcal{L}(M) = \left(\bigoplus_{i=1}^3 \mathcal{L}(\mathrm{M}_u)\right) \oplus \left(\bigoplus_{\alpha=1}^2 \mathrm{L}^\alpha\right),
\end{equation*}
where $\mathrm{L}^\alpha = \sum_{\beta=1}^2 \mathcal{L}(\mathrm{M}^{\alpha\beta}_\theta)$. Since $\mathrm{M}_u$ is nonnegative, our only concern is ensuring that $\mathrm{L}^\alpha$ is positive definite. Yet,
\begin{equation}
\label{eq: L^alpha_ii}
    (\mathrm{L}^\alpha)_{ii} = \sum_{\beta=1}^2 \sum_{j=1}^m (\mathrm{M}^{\alpha\beta}_\theta)_{ij} = \int_S \rho_\theta \left(\sum_{j=1}^m B_i B_j\right) \left(\sum_{\beta=1}^2 a^{\alpha\beta}\right) \sqrt{a} \dd S.
\end{equation}
Condition \eqref{eq: sufficient_cond} guarantees that $\sum_{\beta=1}^2 a^{\alpha\beta} > 0$ and therefore $(\mathrm{L}^\alpha)_{ii}>0$. Indeed,
\begin{equation*}
    \frac{\min_i \|\bm{a}_i\|_2}{\max_i \|\bm{a}_i\|_2} > \cos \theta(\bm{a}_1,\bm{a}_2), \iff
    \begin{cases}
        \frac{\|\bm{a}_1\|_2}{\|\bm{a}_2\|_2} > \cos \theta(\bm{a}_1,\bm{a}_2), \\
        \frac{\|\bm{a}_2\|_2}{\|\bm{a}_1\|_2} > \cos \theta(\bm{a}_1,\bm{a}_2),
    \end{cases}
    \iff
    \begin{cases}
        a_{11} = \bm{a}_1 \cdot \bm{a}_1 > \bm{a}_2 \cdot \bm{a}_1 = a_{21}, \\
        a_{22} = \bm{a}_2 \cdot \bm{a}_2 > \bm{a}_1 \cdot \bm{a}_2 = a_{12}.
    \end{cases}
\end{equation*}
Thus, recalling \eqref{metric_tensor_relation},
\begin{align*}
    a^{11} + a^{12} = a^{-1}(a_{22}-a_{12})>0, \\
    a^{21} + a^{22} = a^{-1}(a_{11}-a_{21})>0,
\end{align*}
since $a$ (the determinant of the surface metric tensor) is positive. Consequently, $\mathcal{L}(M)$ is positive definite.
\end{proof}

\begin{remark}
Although sufficient, the condition stated in \Cref{lem: PD_cond} is not necessary. Indeed, the integral in \eqref{eq: L^alpha_ii} may be positive even though the integrand becomes negative pointwise. Anyway, the condition is not very restrictive and essentially means that the norm of the covariant basis vectors should be balanced and the angle between them should not be too small. In the sequel, we will assume this condition is always satisfied.  
\end{remark}

Although more advanced mass lumping strategies have been developed for isogeometric analysis \cite{anitescu2019isogeometric,nguyen2023towards,li2022significance}, most of them are only applicable to boundary-fitted discretizations. To the best of our knowledge, the only methods supporting single-patch, multi-patch as well as trimmed geometries, are based on algebraic generalizations of the row-sum technique that effectively enlarge its bandwidth \cite{voet2023mathematical,voet2025mass}. However, the resulting lumped mass matrices are no longer diagonal and we will stick to the classical row-sum technique for simplicity. Apart from facilitating the solution of linear systems, mass lumping is also praised for increasing the critical time step of explicit time integration schemes, which is inversely proportional to the largest discrete frequency of the system (see \eqref{eq: CFL_condition}). The constraint on the step size is often referred to as the Courant–Friedrichs–Lewy (CFL) condition.

For shell problems, mass lumping is often combined with a form of rotational mass scaling that artificially increases the inertia of rotational degrees of freedom \cite{benson2010isogeometric,voet2025theoretical}. Although it increases the critical time step, it might also heavily degrade the accuracy of the solution. Thus, we have not incorporated any form of mass scaling in this work.

Mass lumping has always played a central role in explicit dynamics and became even more prominent after Leidinger \cite{leidinger2020explicit} showed that the largest generalized eigenvalues for the row-sum lumped mass remained bounded regardless of how badly cut the elements were, provided the spline space was sufficiently smooth. This finding has since then been thoroughly investigated, both numerically \cite{messmer2021isogeometric,messmer2022efficient,coradello2021accurate,radtke2024analysis} and more theoretically \cite{stoter2023critical,bioli2025theoretical} and is by now relatively well understood. It turns out this ``blessing of smoothness'' is more of a double-edged sword. Indeed, while the largest eigenvalues may no longer depend on the cut elements, the smallest ones do and converge to zero at a rate that depends on the polynomial degree and cut configuration \cite{bioli2025theoretical}. These eigenvalues bring in spurious oscillatory modes in the low-frequency spectrum \cite{coradello2021accurate,radtke2024analysis} and if those modes are unluckily activated, they may cause spurious oscillations in the solution itself. Their disastrous effect is evident from the recent counter-examples reported in \cite{voet2025stabilization} for 1D and 2D wave problems. Thus, mass lumping is not a miracle and merely trades one problem for another. Once spurious eigenvalues and modes are lying in the low-frequency spectrum, filtering them out is not so obvious. Instead, the most convenient way of eliminating them is to prevent them from surfacing in the first place.

Fortunately, spurious high-frequency modes for the consistent mass and spurious low-frequency modes for the lumped mass have exactly the same origin and any solution for the former is also a solution for the latter. Among them are stabilization techniques based on local polynomial extensions \cite{buffa2020minimal,burman2023extension,voet2025stabilization} and extended B-splines \cite{hollig2003finite,hollig2005introduction,marussig2017stable}. In \cite{voet2025stabilization}, the authors used the former for stabilizing the discrete formulation and subsequently formed the row-sum lumped mass from the stabilized consistent mass. This technique removes the large spurious (trimming-related) eigenvalues for the consistent mass and prevents the corresponding modes from surfacing in the low-frequency spectrum when the mass matrix is lumped. Although this solution has only been studied for a simple wave equation, we believe it also applies to plate and shell models, where similar issues might arise. We will demonstrate it in \Cref{se: numerical_experiments} but before doing so, the stabilization technique is recalled in the next section.

\section{Stabilization}
\label{se: stabilization}
In this section, we explain how to build a stable approximation space that is later used for approximating both the displacement and rotation fields. This introduction closely follows the presentation in \cite{buffa2020minimal,burman2023extension,voet2025stabilization}, where a similar construction was employed for the wave equation and the Poisson problem. 

We consider the B-spline basis $\mathcal{B}=\{ B_i \}_{i=1}^{n}$ for a spline space in the trimmed parametric domain $S$. As explained earlier, for trimmed geometries, a standard tensor product basis is first constructed over a fictitious domain $S_0$ containing the parametric domain $S$ and $\mathcal{B}$ then collects all basis functions whose support intersects the parametric domain. Similarly, $\mathcal{T}_h$ collects all active mesh elements. Unfortunately, the support of some of the functions in $\mathcal{B}$ may only intersect the domain over arbitrarily small elements and those functions typically cause great stability and conditioning issues \cite{de2017condition,de2023stability}. The idea presented in \cite{buffa2020minimal} is to locally substitute those functions with the extension of polynomial segments from large neighboring elements. For a fixed tolerance $\gamma \in [0,1]$, an element $T \in \mathcal{T}_h$ is labeled as large (or good) if
\begin{equation}
\label{eq: good_element}
    |T \cap S | \geq \gamma |T|
\end{equation}
and is small (or bad) otherwise. The sets of large and small elements, denoted $\mathcal{T}_h^L$ and $\mathcal{T}_h^S$ respectively, partition $\mathcal{T}_h$ into disjoint sets such that $\mathcal{T}_h = \mathcal{T}_h^L \cup \mathcal{T}_h^S$. Given a small element $T \in \mathcal{T}_h^S$, there always exist a large neighboring element $T' \in \mathcal{T}_h^L$ provided the boundary is sufficiently regular and the mesh sufficiently fine \cite{burman2022cutfem}. This large neighboring element may either be chosen such that it maximizes $|T' \cap S|$ among all neighbors \cite{buffa2020minimal} or minimizes the distance between the centers of mass of $|T \cap S|$ and $|T' \cap S|$ \cite{burman2023extension}. The B-splines on the small element $T$ are then locally substituted with their extension from the good neighbor $T'$, as in \Cref{fig: poly_ext}. In 1D (\Cref{fig: 1D_poly_ext}), this technique coincides with extended B-splines \cite{hollig2005introduction,marussig2017stable}. However, in dimension $d \geq 2$ (\Cref{fig: 2D_poly_ext}), independent extensions may create discontinuities between neighboring elements. Consequently, the approximation space is generally non-conforming and is not a subspace of the original spline space, which precludes forming the system matrices from their non-stabilized counterpart. Instead, the assembly routine of element matrices must be modified, as described in \cite[][Algorithm 5.1]{voet2025stabilization}. Adapting the algorithm presented therein to other operators and right-hand sides is rather straightforward. Alternatively, one may restore conformity by combining the extension with a projection step \cite{burman2023extension}, but this is not necessary from an approximation point of view. Another possible remedy are extended B-splines \cite{hollig2005introduction,marussig2017stable}, although the increased overlap of extended basis functions may slightly reduce the sparsity of system matrices. On the contrary, the local extension technique ensures the sparsity pattern of stabilized system matrices is contained in their non-stabilized counterpart \cite[][Lemma 5.2]{voet2025stabilization}, which also allows the deployment of advanced mass lumping techniques \cite{voet2025mass}, if need be.

In any case, the extension procedure generally does not preserve the positivity of the B-spline basis (see \Cref{fig: 1D_poly_ext}). Consequently, the stabilized consistent mass matrix may feature negative entries and again the row-sum technique no longer guarantees positive definite lumped mass matrices. Fortunately, the magnitude of such entries is typically very small unless large elements are stabilized. Thus, indefinite lumped mass matrices are very unlikely provided $\gamma$ is sufficiently small. In all our experiments, we chose $\gamma=0.1$ and never encountered any issue. More generally, values of $\gamma$ between $0.05$ and $0.15$ are equally good choices. For a more complete overview of polynomial extension techniques and implementation aspects, we refer to \cite{hollig2003finite,burman2023extension,voet2025stabilization} and the references therein. It is important to realize that none of these techniques are restricted to smooth spline discretizations and also apply to classical $C^0$ finite elements.

\begin{figure}[H]
     \centering
     \begin{subfigure}[t]{0.48\textwidth}
    \centering
    \includegraphics[width=\textwidth]{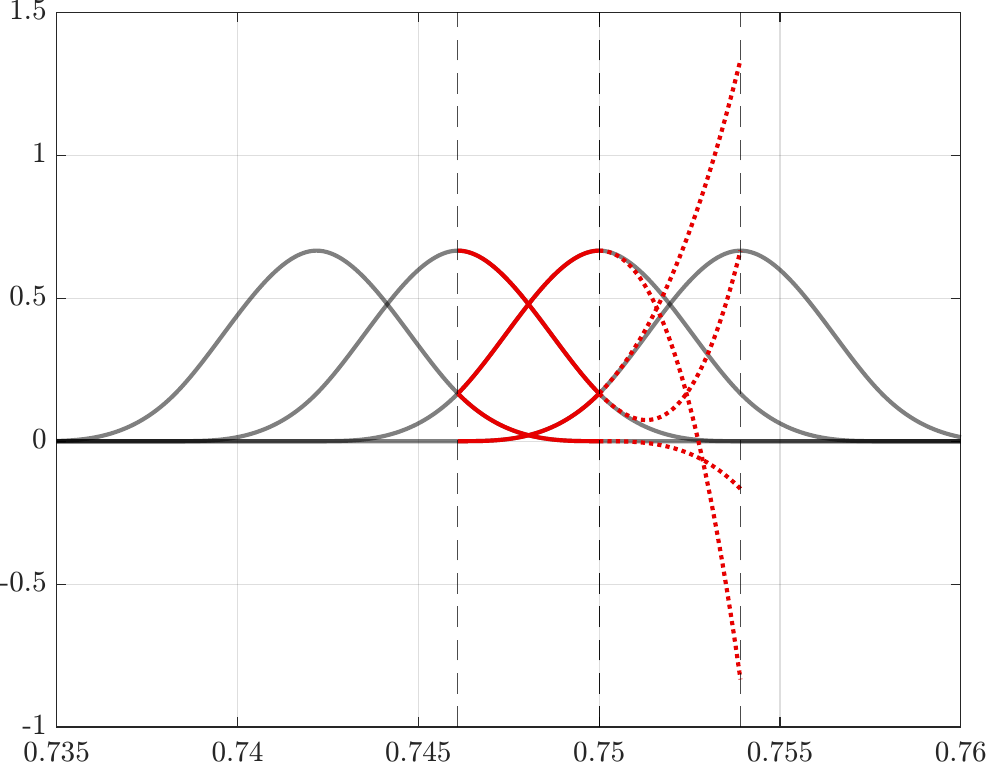}
    \caption{1D cubic $C^2$ splines}
    \label{fig: 1D_poly_ext}
     \end{subfigure}
     \hfill
     \begin{subfigure}[t]{0.48\textwidth}
    \centering
    \includegraphics[width=\textwidth]{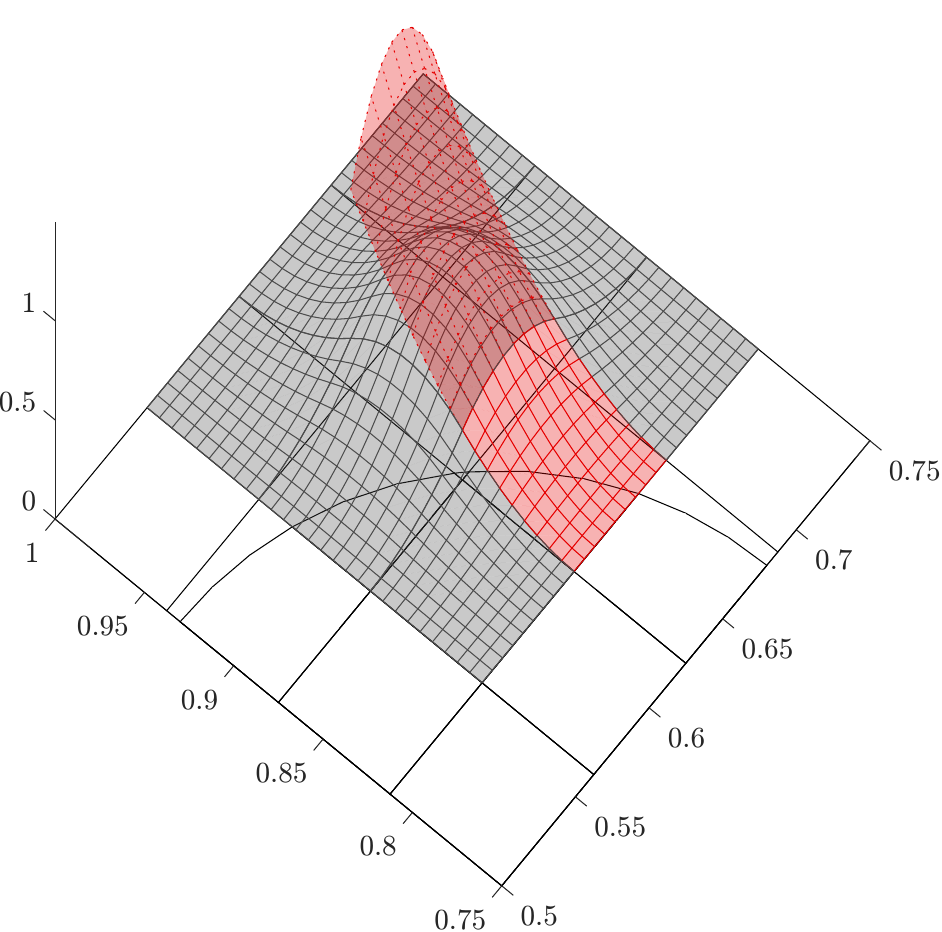}
    \caption{2D quadratic $C^1$ spline}
    \label{fig: 2D_poly_ext}
     \end{subfigure}
     \hfill
    \caption{Polynomial extension technique. The polynomial segments on a large element (in red) are extended into a small neighboring element (dotted red).}
    \label{fig: poly_ext}
\end{figure}

\section{Numerical experiments}
\label{se: numerical_experiments}
In this section, we present a few numerical experiments showing on the one hand the alarming effect mass lumping may have on the numerical solution and on the other hand how the polynomial extension technique might resolve the issue. All our experiments are carried out in GeoPDEs \cite{vazquez2016new}, an open-source MATLAB/Octave software package for isogeometric analysis. Our experiments range from academic examples to more realistic industrial ones, featuring irregularly shaped cut elements. In all our experiments, we chose $\gamma=0.1$ for classifying large and small elements.

\begin{example}[Plate externally trimmed]
\label{ex: plate_externally_trimmed}
Our simplest example is a trimmed square plate of side length $L=1$ m (see \Cref{fig: trimmed_square_plate}). In this case, the shell model of \Cref{se: shell_theory} reduces to a plate one. For illustrating how mass lumping may ruin the simulation, we consider a manufactured solution whose displacement and rotation fields are
\begin{equation*}
    \bm{u}(\bm{x},t) = w(\bm{x})\phi(t) \bm{e}_3, \quad \text{and} \quad \bm{\theta}(\bm{x},t) = - \sum_{i=1}^2 w_{,i}(\bm{x})\phi(t)\bm{a}^i,
\end{equation*}
where the spatial part $w \colon S \to \mathbb{R}$ and temporal part $\phi \colon [0,\; +\infty) \to \mathbb{R}$ are defined as
\begin{equation*}
    w(\bm{x}) = a_0 w_1 w_2 w_n(\bm{x}), \quad \text{and} \quad \phi(t) = \sin(\omega t),
\end{equation*}
with the functions
\begin{align*}
    w_\alpha(\bm{x}) &= e^{\frac{2x_\alpha-1}{b_0}}+e^{\frac{-2x_\alpha-1}{b_0}}, \\
    w_n(\bm{x}) &= \cos\left(4 \pi n_0(x_1^2+x_2^2)\right),
\end{align*}
with parameter values $a_0 = 10$ cm, $b_0=0.3$, $n_0=5$, $\omega = \frac{1}{10L}\sqrt{\frac{E}{\rho}}$, unit Young modulus $E$ and density $\rho$ and Poisson ratio $\nu=0.25$.

Neumann (or type 2) boundary conditions are prescribed over the entire boundary and are computed from the exact solution, just as the source term and initial conditions. Snapshots of the vertical component of the manufactured displacement field $(u_3)$ are shown in \figref{fig: trimmed_square_plate_sol}{e$_1$-e$_3$} for three different times $t_1$-$t_3$ as indicated in \figref{fig: trimmed_square_plate_sol}{a}. The plate is discretized with maximally smooth B-splines of degree $p$ over a background grid of mesh size $h$. The left and right (resp. top and bottom) boundaries are at a distance of $\delta$ from the nearest vertical (resp. horizontal) grid line. For future reference, we denote $\epsilon=\delta/h$ the relative trimming parameter. The numerical solutions are computed with (1) the consistent mass, (2) the lumped mass and (3) the stabilized lumped mass for various spline orders $p$, trimming parameters $\epsilon$ and thicknesses $\tau=\{1,\; 5,\; 10\}$ cm. For the stabilized lumped mass, the approximation space is first stabilized with the polynomial extension technique recalled in \Cref{se: stabilization} before lumping the mass matrix with the standard row-sum technique. Stabilizing the mass matrix prior to mass lumping filters out the spurious (trimming-related) outlier frequencies and modes, thereby preventing them from moving toward the low-frequency part of the spectrum when the mass matrix is lumped. Note that computing the numerical solution for the (non-stabilized) consistent mass with an explicit scheme is practically infeasible, due to the stringent constraint on the step size. Hence, we have used the Central Difference method (an explicit Newmark method) for the stabilized and non-stabilized lumped mass solutions, while relying on an implicit unconditionally stable Newmark method for the consistent mass solution, which was computed with the same step size as the explicit solutions to ensure a comparable error in time.

Snapshots of the solutions are provided in \Cref{fig: trimmed_square_plate_sol} for the consistent (c$_1$-c$_3$), lumped (l$_1$-l$_3$) and stabilized lumped (s$_1$-s$_3$) solutions at three different times $t_1$-$t_3$. As shown in \figref{fig: trimmed_square_plate_sol}{s$_1$-s$_3$}, the stabilization technique drastically improves the accuracy of the lumped mass solution. Indeed, while the solution for the consistent mass (\figref{fig: trimmed_square_plate_sol}{c$_1$-c$_3$}) closely resembles the exact solution (\figref{fig: trimmed_square_plate_sol}{e$_1$-e$_3$}), the solution for the lumped mass is utterly inaccurate and features spurious oscillations nears the corners of the plate (\figref{fig: trimmed_square_plate_sol}{l$_1$-l$_3$}), even when the exact solution vanishes. As a matter of fact, the numerical solution in this case is completely worthless and further deteriorates as the spline order and slenderness increase or as the trimming parameter decreases (see \Cref{fig: trimmed_square_plate_err_cons_lumped}). The error is much more pronounced in the $L^\infty$ norm that measures pointwise inaccuracies (\Cref{fig: trimmed_square_plate_err_cons_lumped_refinement}). Although refining the mesh (i.e., decreasing $h/L$) improves the lumped mass solution (see \Cref{fig: trimmed_square_plate_err_cons_lumped_refinement}), an acceptable accuracy might only be achieved for infeasibly small mesh sizes. Therefore, refining the mesh might not be a viable solution due to computer limitations. Remarkably, stabilizing the space prior to lumping the mass completely eliminates the oscillations. As shown in \Cref{fig: trimmed_square_plate_err_cons_stab_lumped}, the error for the stabilized lumped mass now perfectly overlaps with the error for the consistent mass. We must stress that the improvement is not attributed to an improvement of the lumping technique per se, since both the stabilized and non-stabilized lumped mass solutions are computed with the row-sum technique. However, its accuracy is space (and even basis) dependent, which again hints at modifying the space prior to lumping the mass.

\begin{figure}[H]
    \centering
    \includegraphics[width=0.5\textwidth]{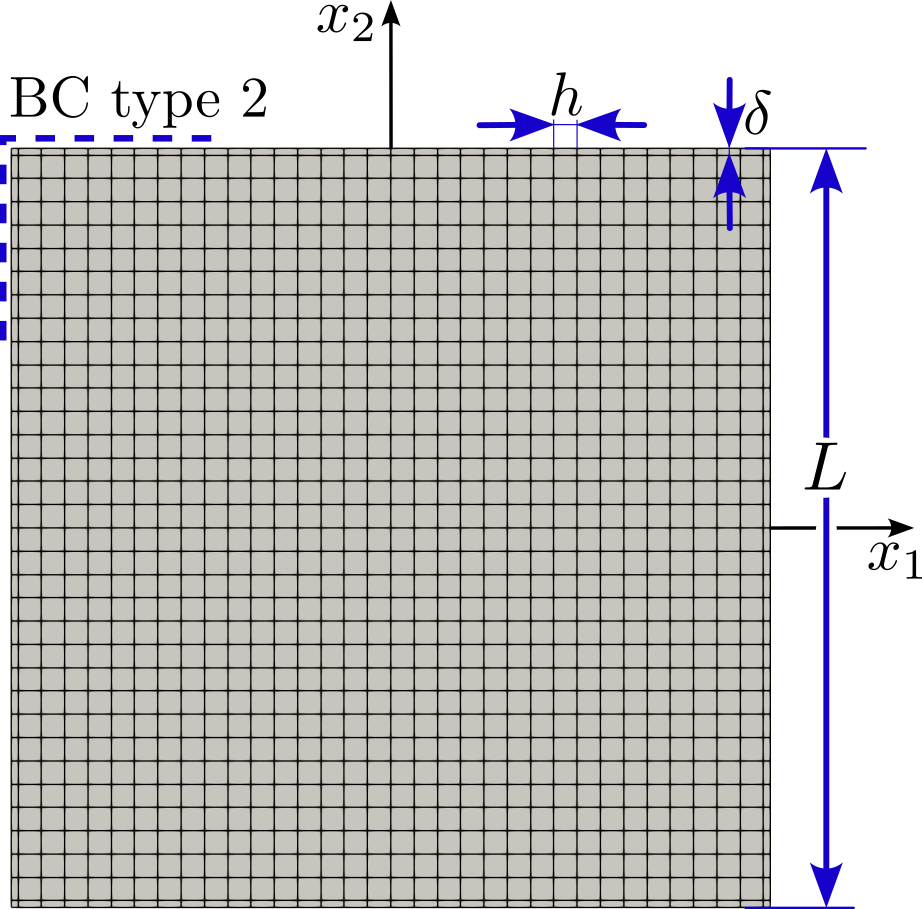}
    \caption{\Cref{ex: plate_externally_trimmed}: Trimmed square plate.}
    \label{fig: trimmed_square_plate}
\end{figure}

\begin{figure}[H]	
    \centering
    \includegraphics[width=0.7\textwidth]{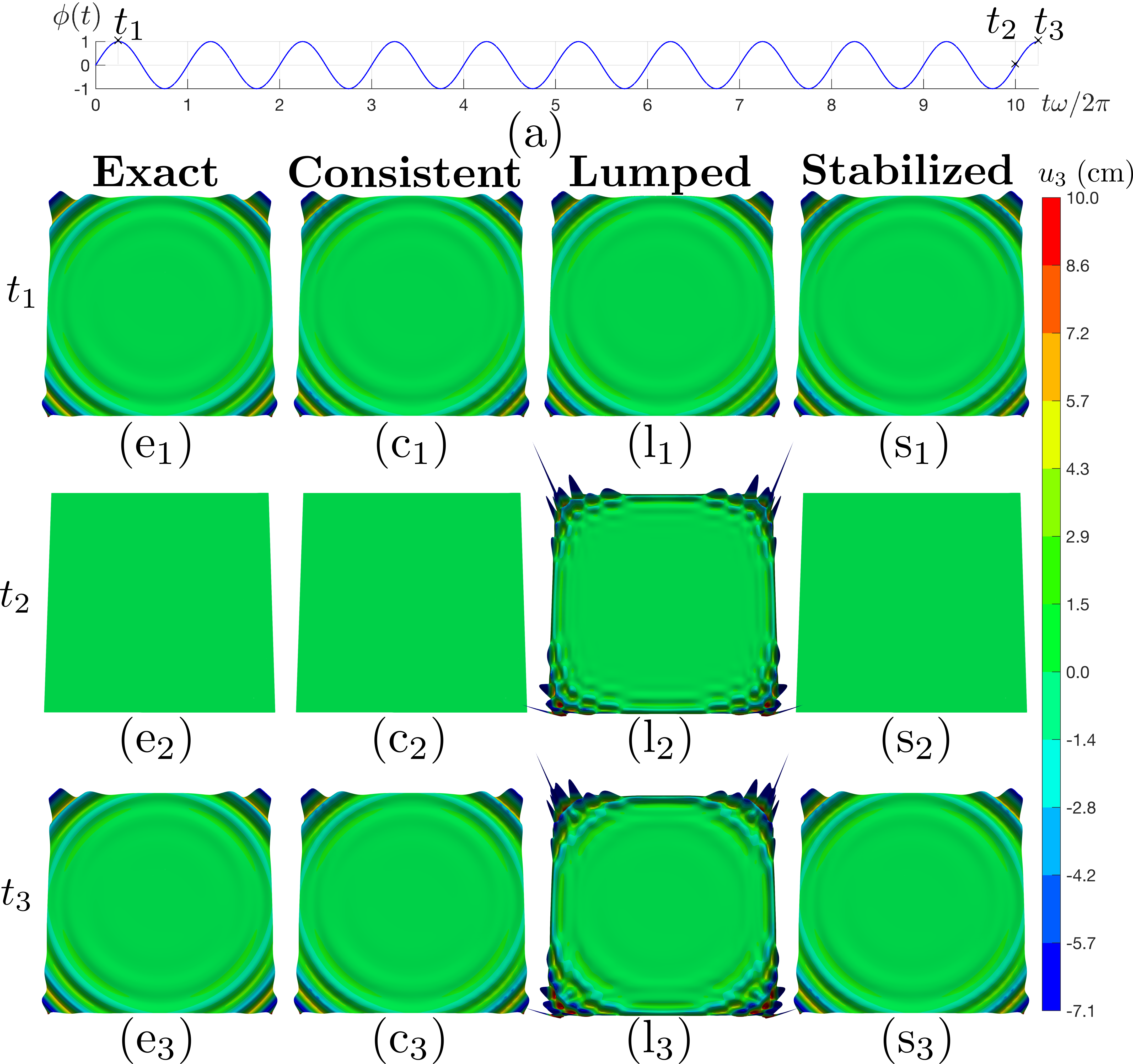}
    \caption{\Cref{ex: plate_externally_trimmed}: Snapshots of the exact and numerical solutions.}
    \label{fig: trimmed_square_plate_sol}
\end{figure}

% The slenderness ratio in this case becomes $L/\tau$.

% \begin{itemize}
%     \item BC type 2: Neumann boundary conditions
%     \item $\epsilon=\delta/h$
%     \item $L=1$ m
%     \item $\omega=0.5\sqrt{\frac{E}{\rho L^2}}$
%     \item $\nu=0.25$
%     \item $\tau=[10, 5, 1]$ cm 
%     \item manufactured displacement field $\varphi(\bm{\xi})\phi(t) \bm{e}_3$
%     \item manufactured rotation field $\bm{\theta}(\bm{\xi},t) = -\varphi_{,\alpha}\bm{\xi})\phi(t)\bm{a}^\alpha$
% \end{itemize}

% The function $\varphi$ is constructed as 

% $\varphi = a_0 \varphi_1 \varphi_2 \varphi_n$

% $\varphi_\alpha = e^{\frac{\tilde{x}_\alpha-1}{b_0}}+e^{\frac{-\tilde{x}_\alpha-1}{b_0}}$

% $\varphi_n = \sin\left(\pi(\tilde{x}_1^2+\tilde{x}_2^2)n_0 +\frac{\pi}{2}\right)$

% $\tilde{x}_\alpha= \frac{2 x_\alpha}{ L}-1$

% The values of the parameters for this test are the following:

% \begin{tabular}{c|c}
% \hline
% \hline
% Parameter & Value \\  
% \hline
% $a_0$ & 10 cm \\  
% $b_0$ & 0.3 \\  
% $n_0$ & 5 \\
% \end{tabular}

\begin{figure}[H]	
    \centering
    \begin{subfigure}[b]{0.27\textwidth}\includegraphics[height=4.0cm]{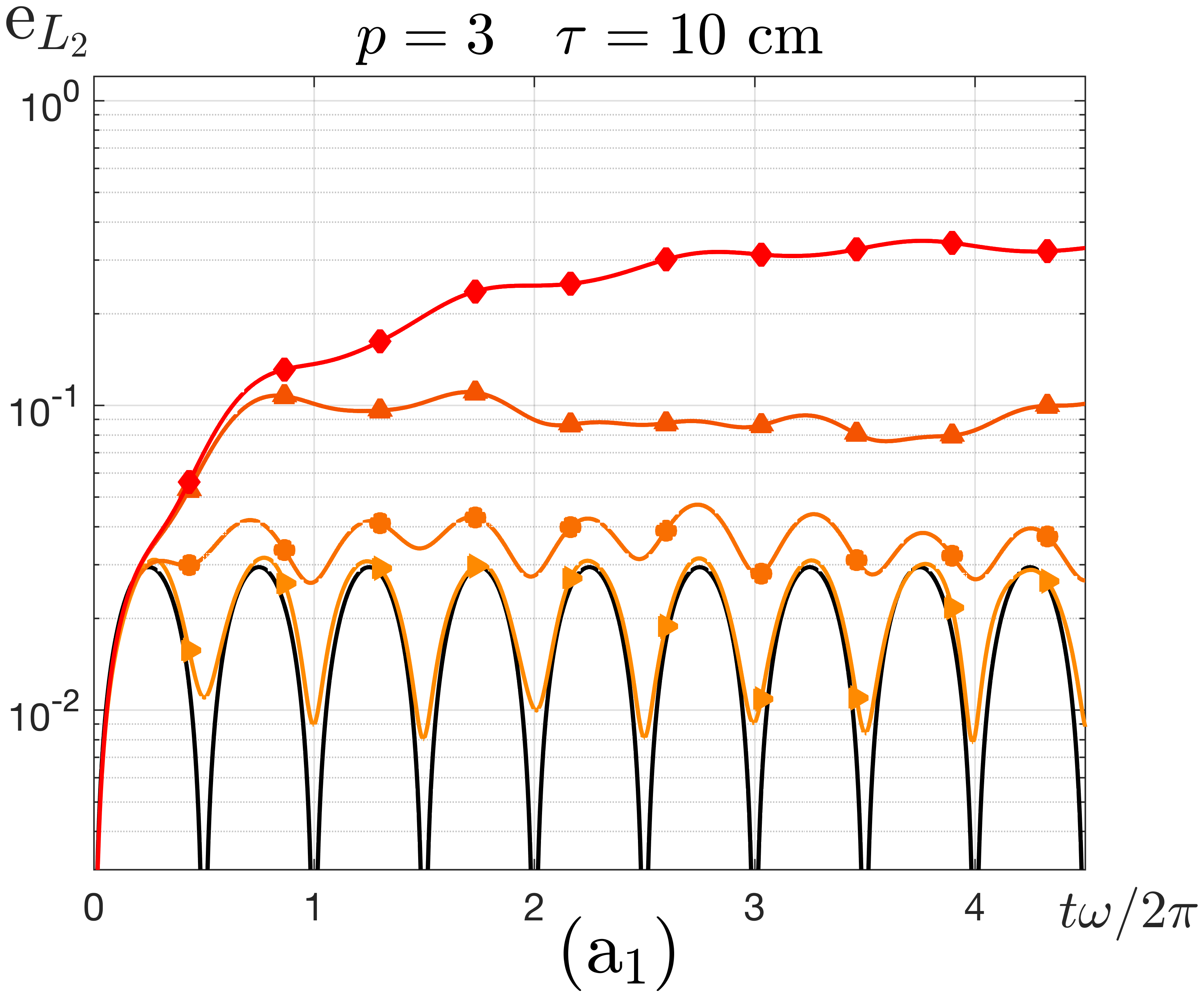}\end{subfigure} 
    \begin{subfigure}[b]{0.27\textwidth}\includegraphics[height=4.0cm]{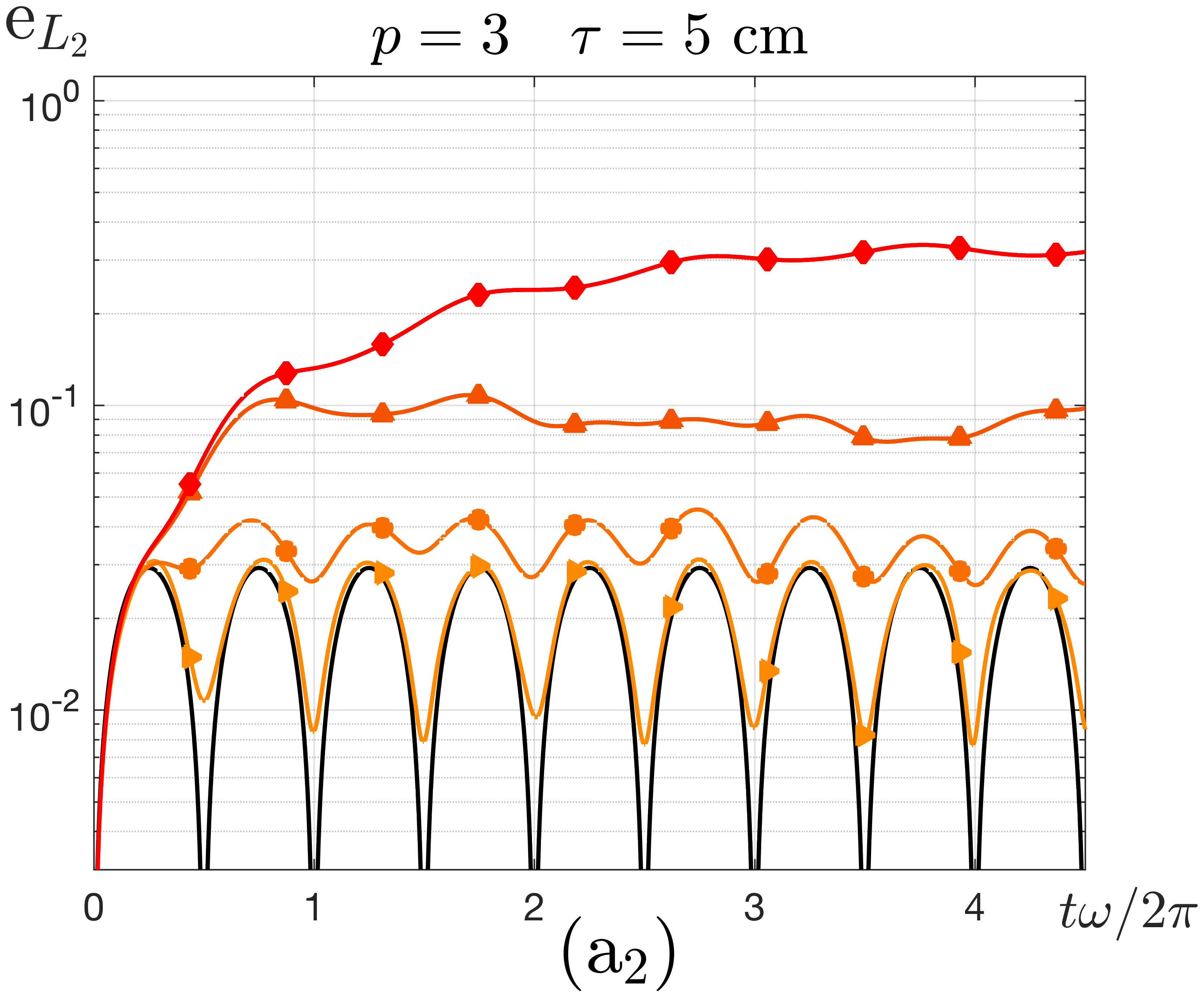}\end{subfigure} 
    \begin{subfigure}[b]{0.44\textwidth}\includegraphics[height=4.0cm]{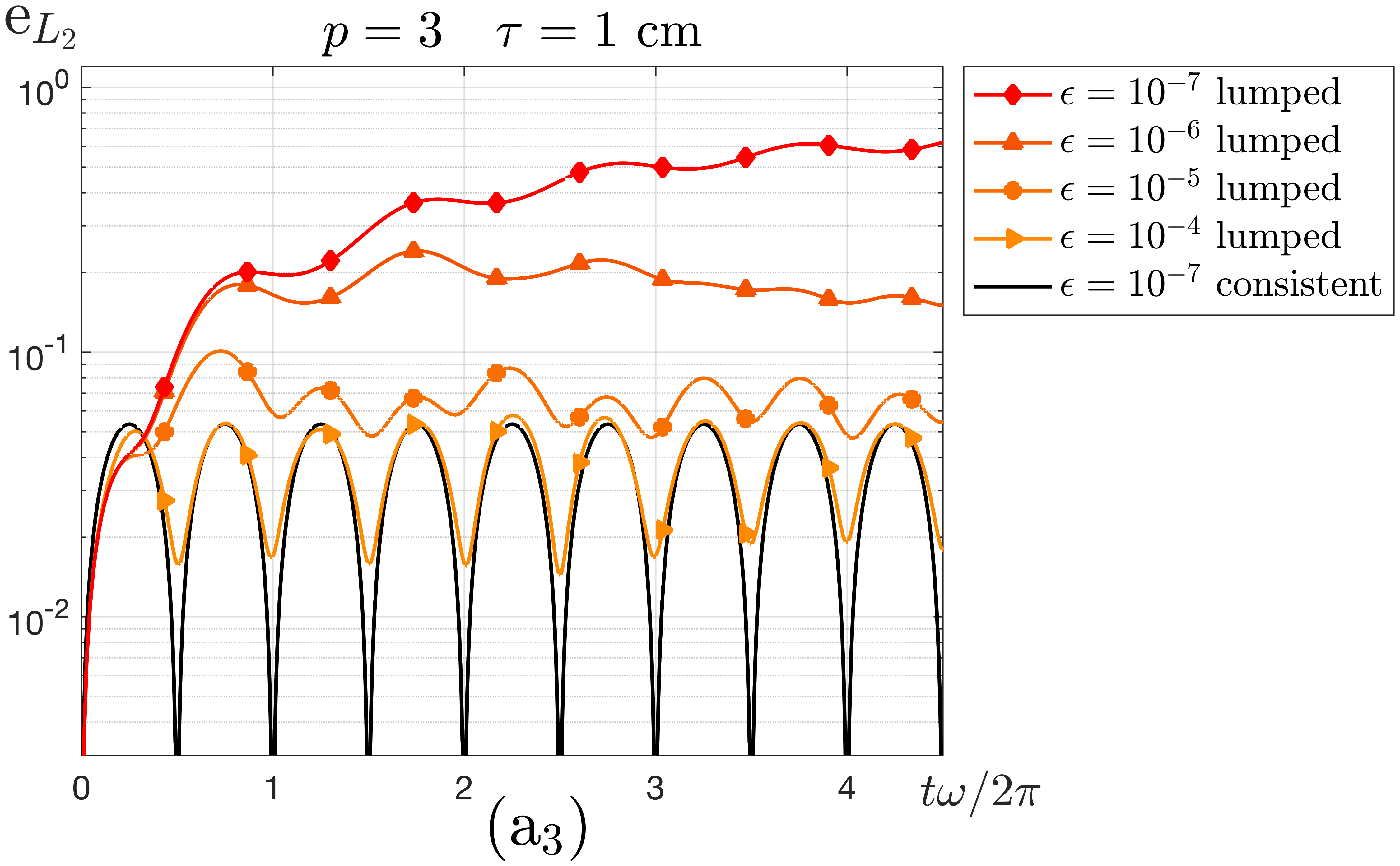}\end{subfigure}\\
    \begin{subfigure}[b]{0.27\textwidth}\includegraphics[height=4.0cm]{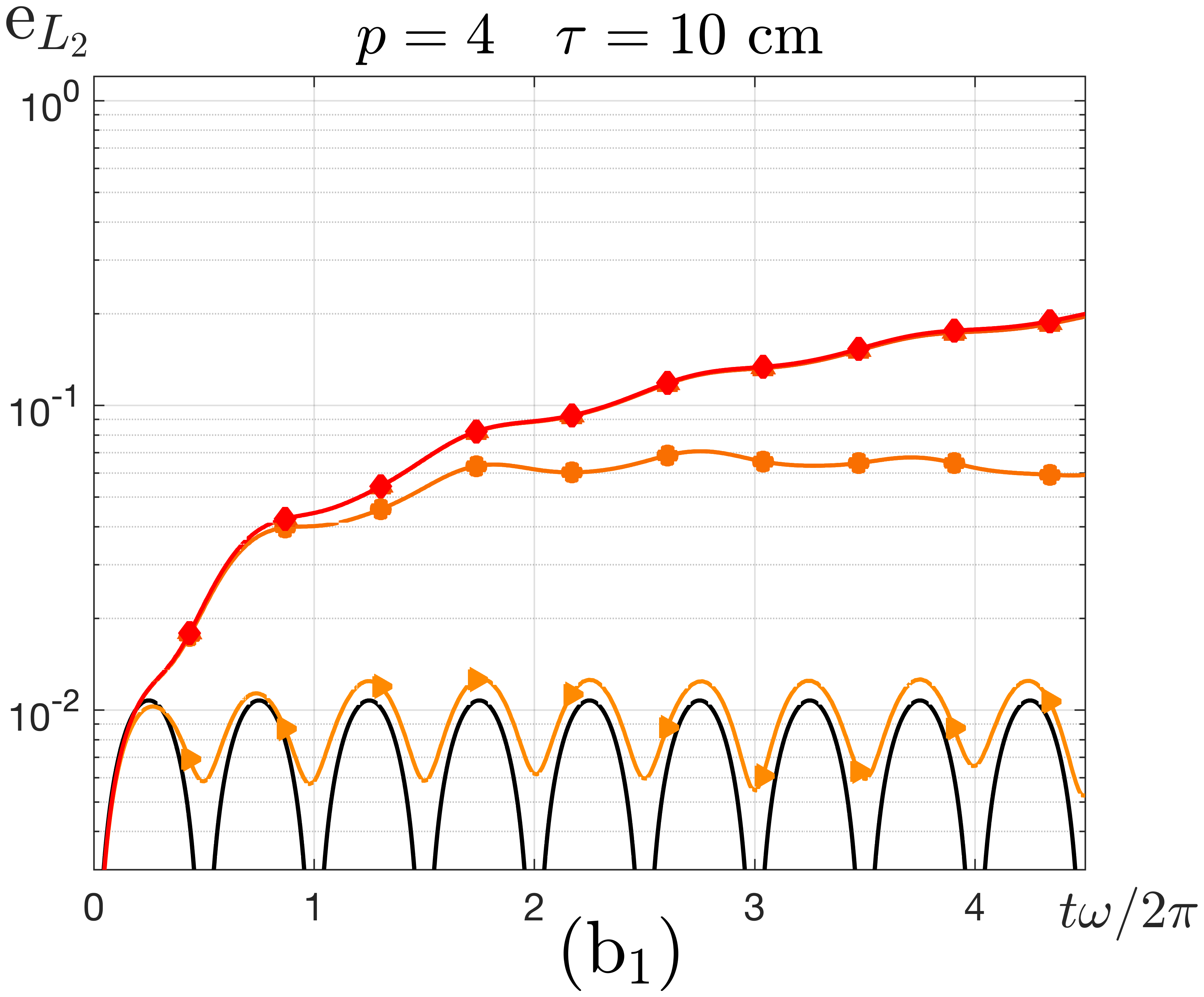}\end{subfigure}
    \begin{subfigure}[b]{0.27\textwidth}\includegraphics[height=4.0cm]{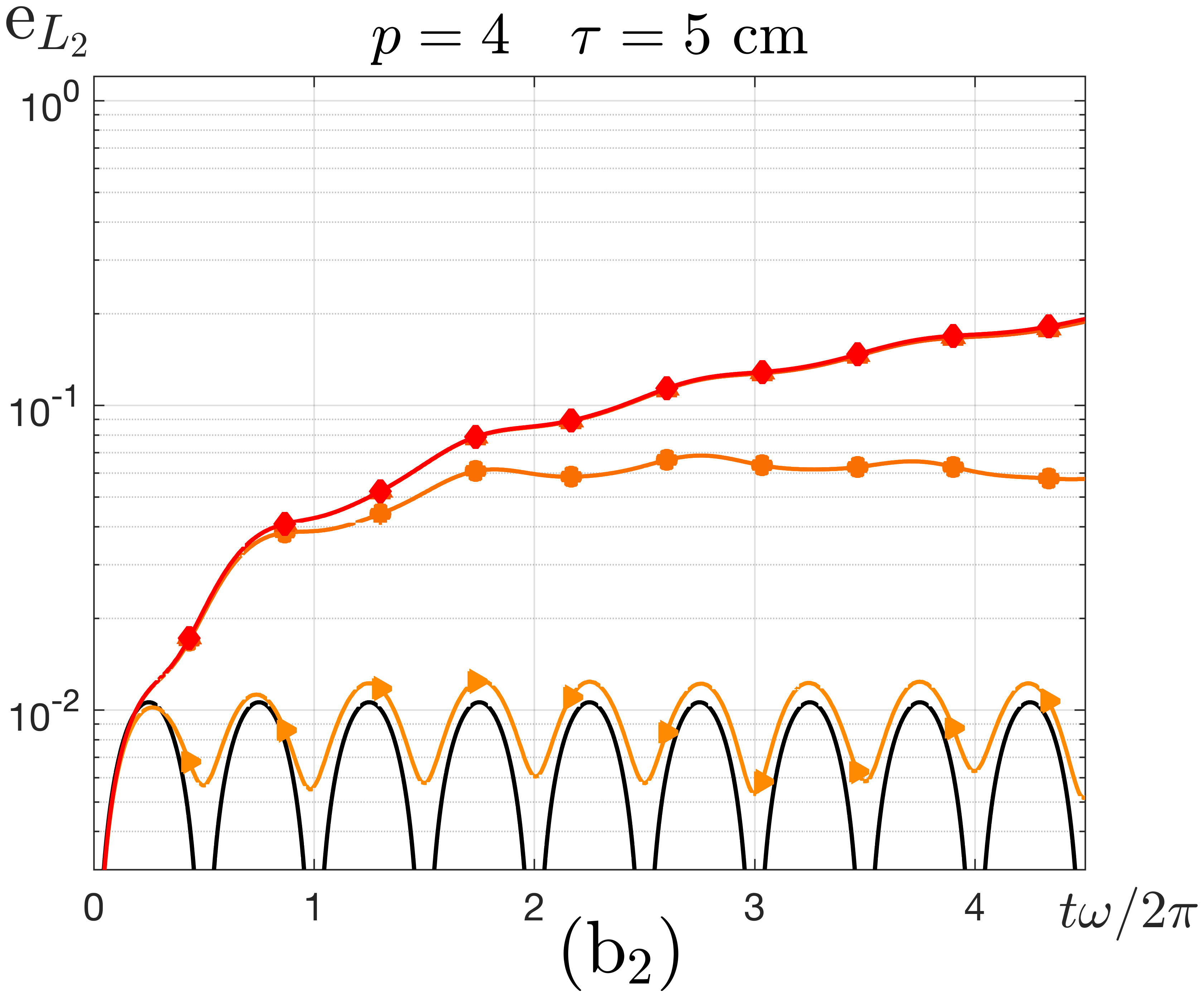}\end{subfigure}
    \begin{subfigure}[b]{0.44\textwidth}\includegraphics[height=4.0cm]{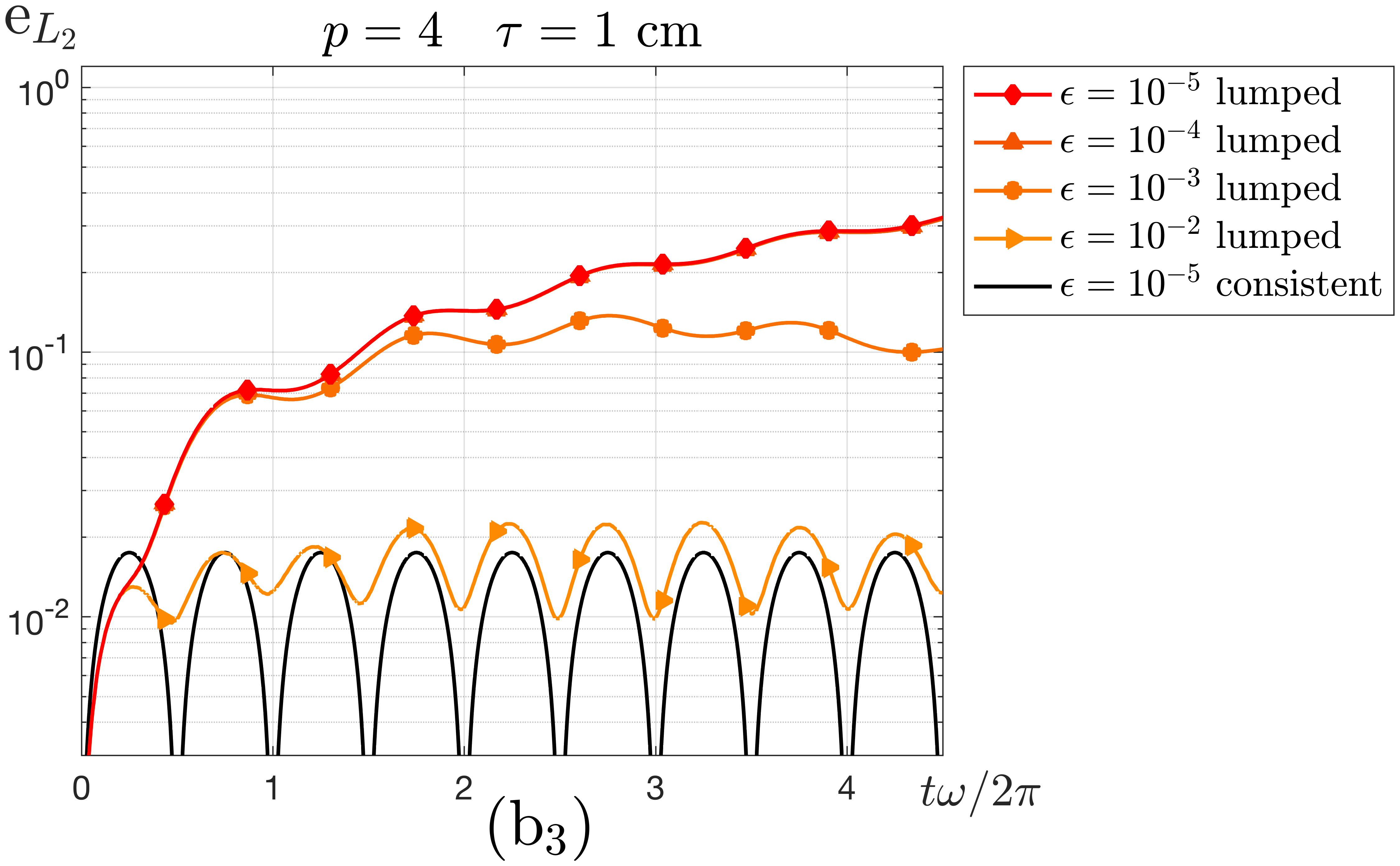}\end{subfigure}
    \caption{\Cref{ex: plate_externally_trimmed}: $L^2$ error for the consistent and lumped mass solutions for different spline orders, thicknesses and trimming parameters.} 
    \label{fig: trimmed_square_plate_err_cons_lumped}
\end{figure}

\begin{figure}[H]	
    \centering 
    \begin{subfigure}[b]{0.27\textwidth}\includegraphics[height=4.0cm]{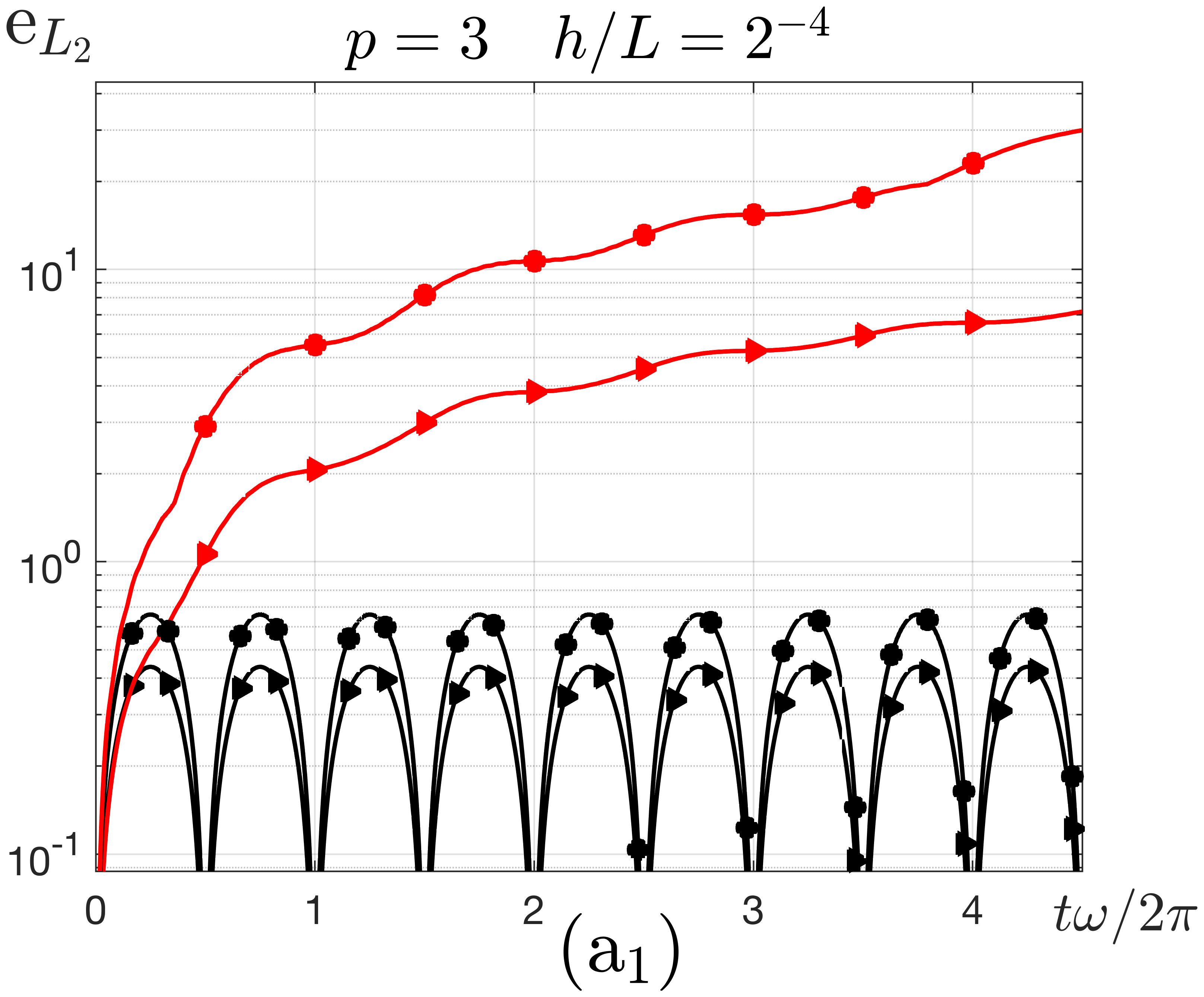}
    \end{subfigure}
    \begin{subfigure}[b]{0.27\textwidth}\includegraphics[height=4.0cm]{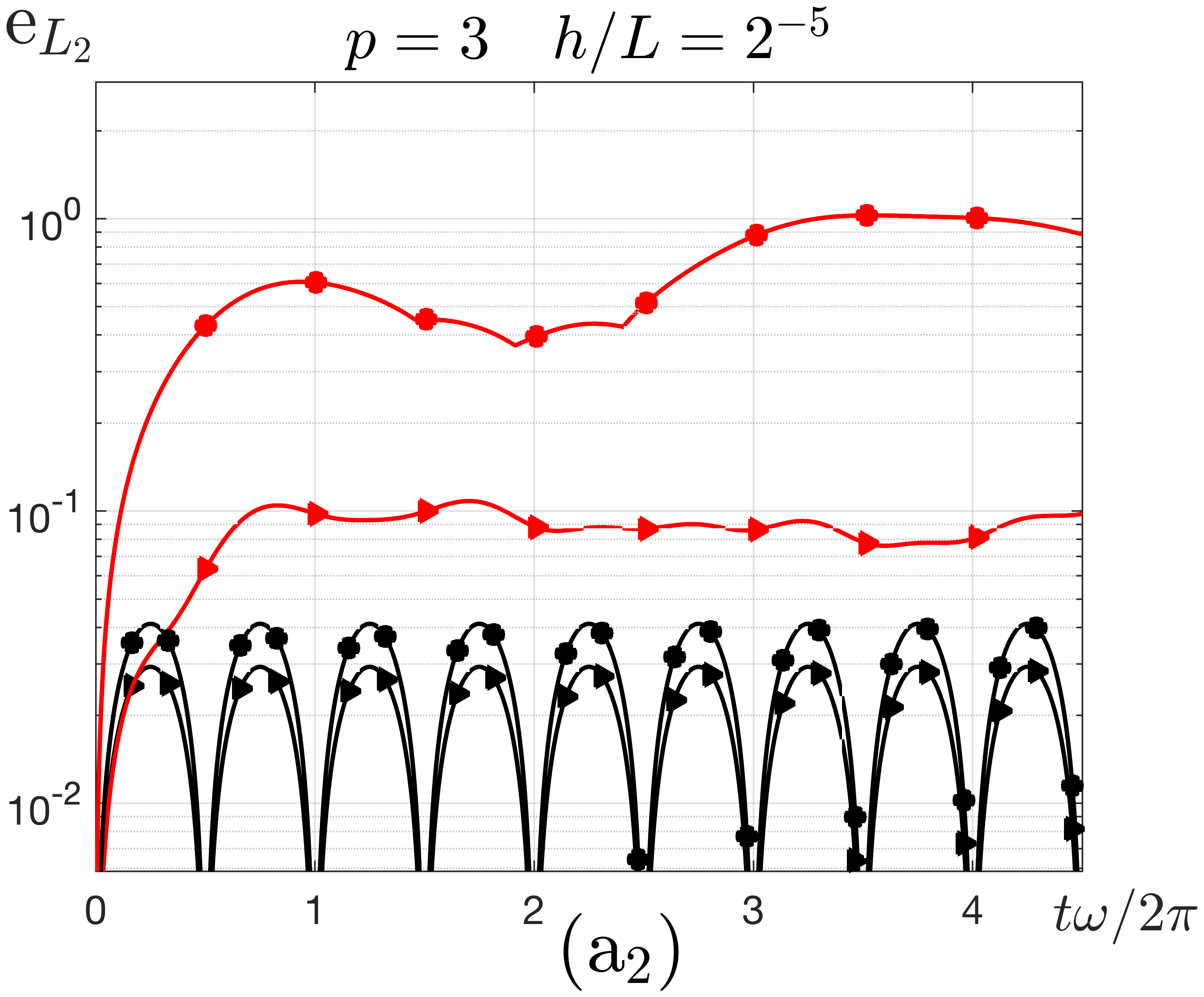}
    \end{subfigure}
    \begin{subfigure}[b]{0.44\textwidth}\includegraphics[height=4.0cm]{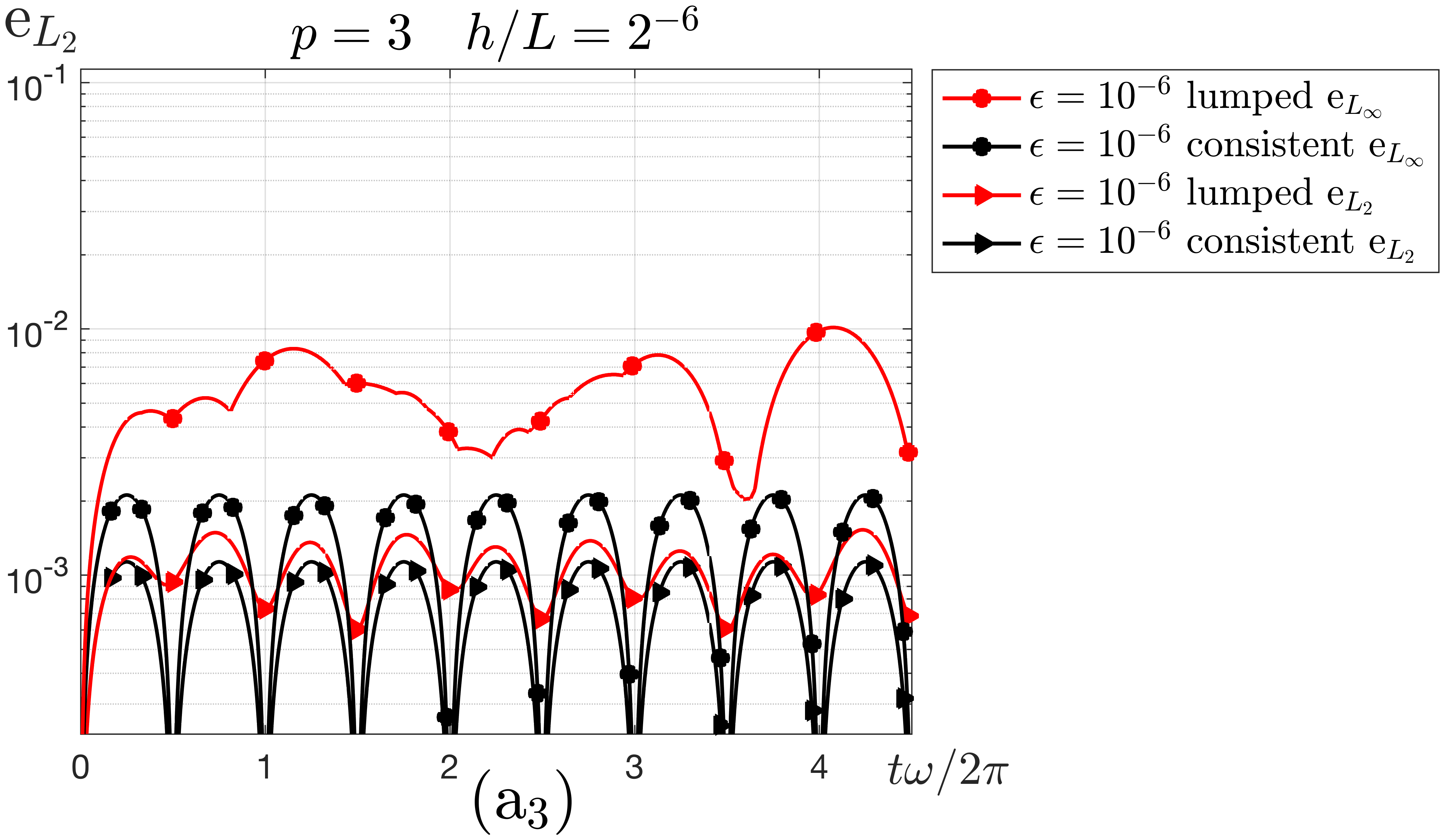}
    \end{subfigure}\\
    \begin{subfigure}[b]{0.27\textwidth}\includegraphics[height=4.0cm]{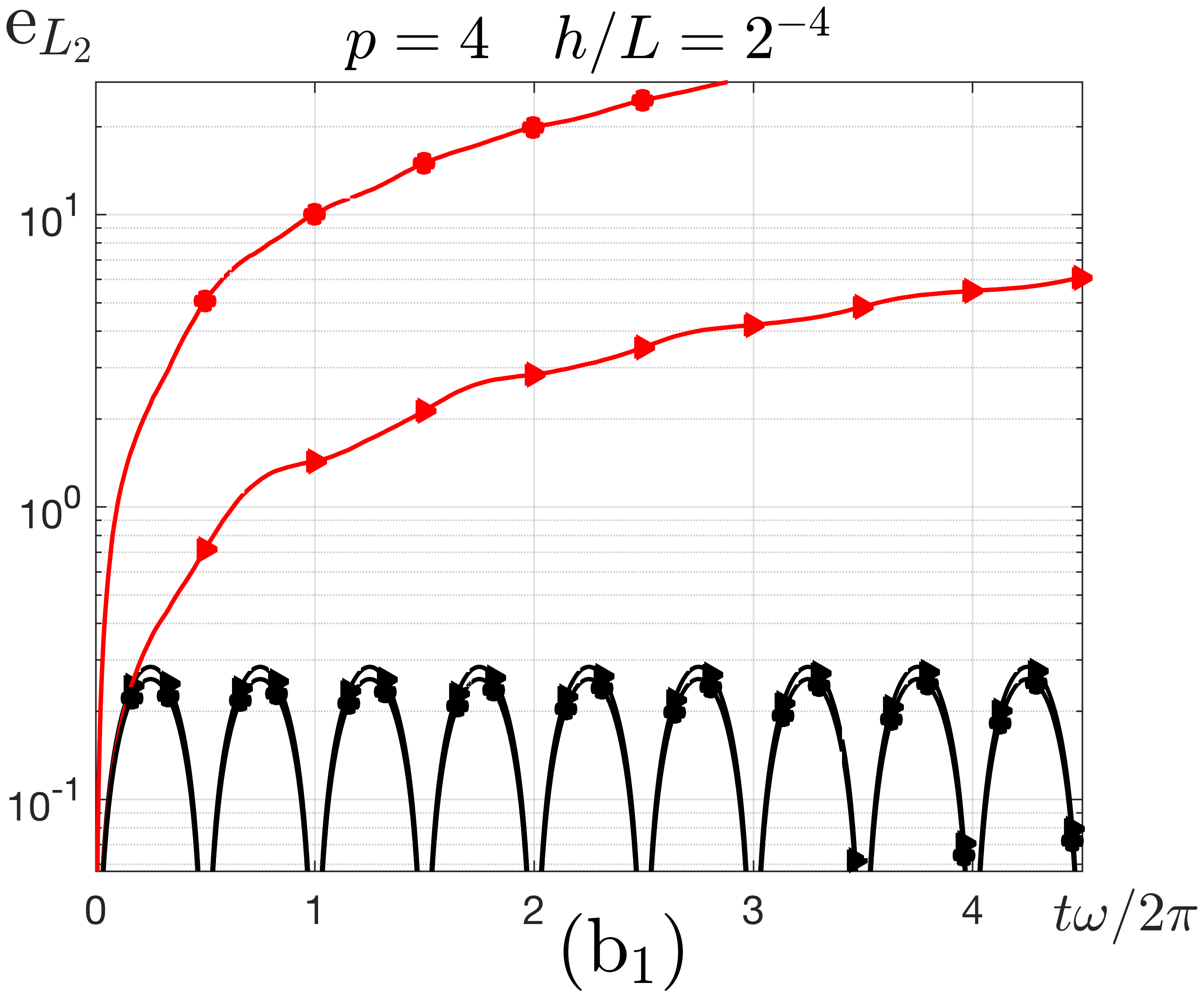}
    \end{subfigure}
    \begin{subfigure}[b]{0.27\textwidth}\includegraphics[height=4.0cm]{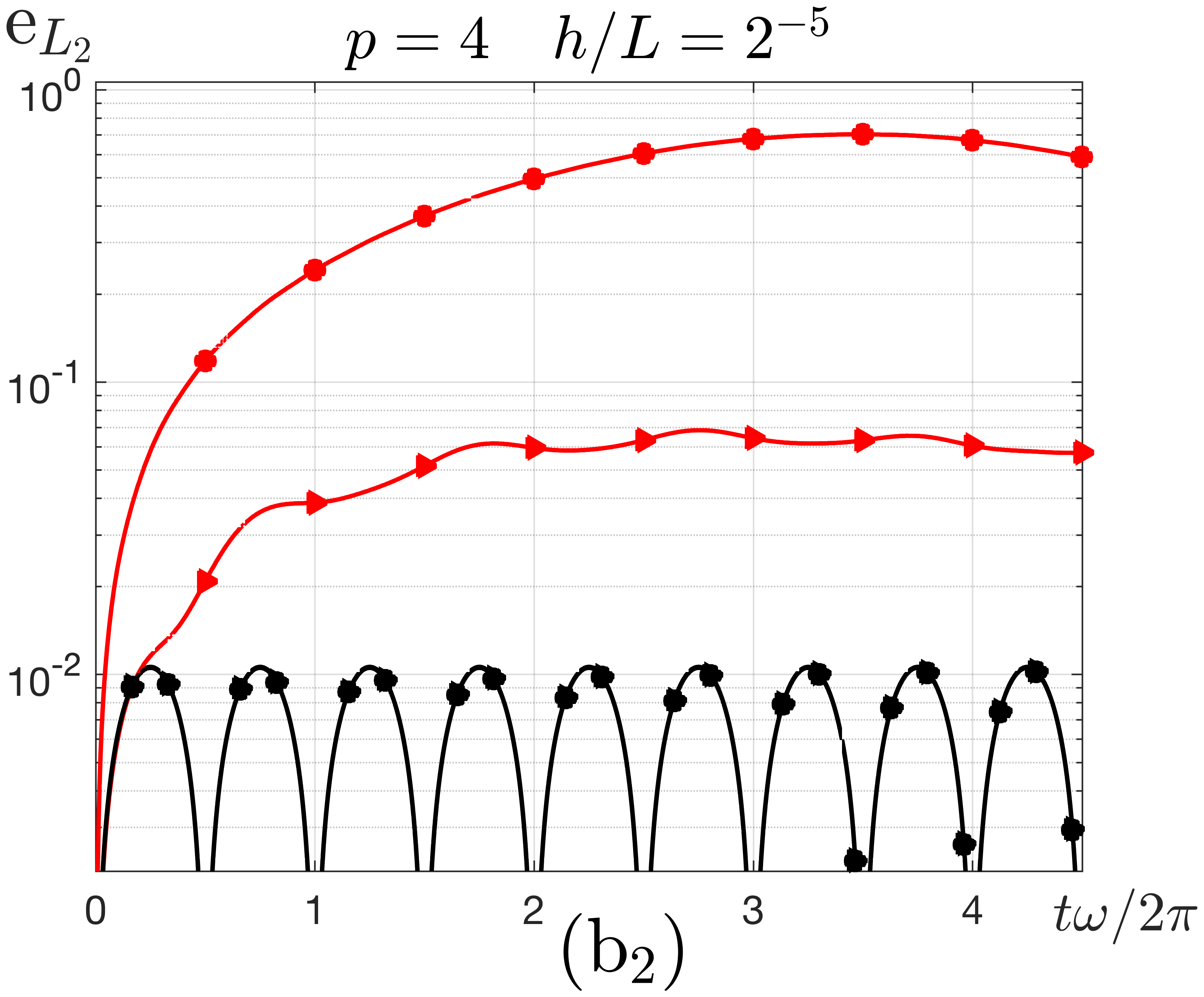}
    \end{subfigure}
    \begin{subfigure}[b]{0.44\textwidth}\includegraphics[height=4.0cm]{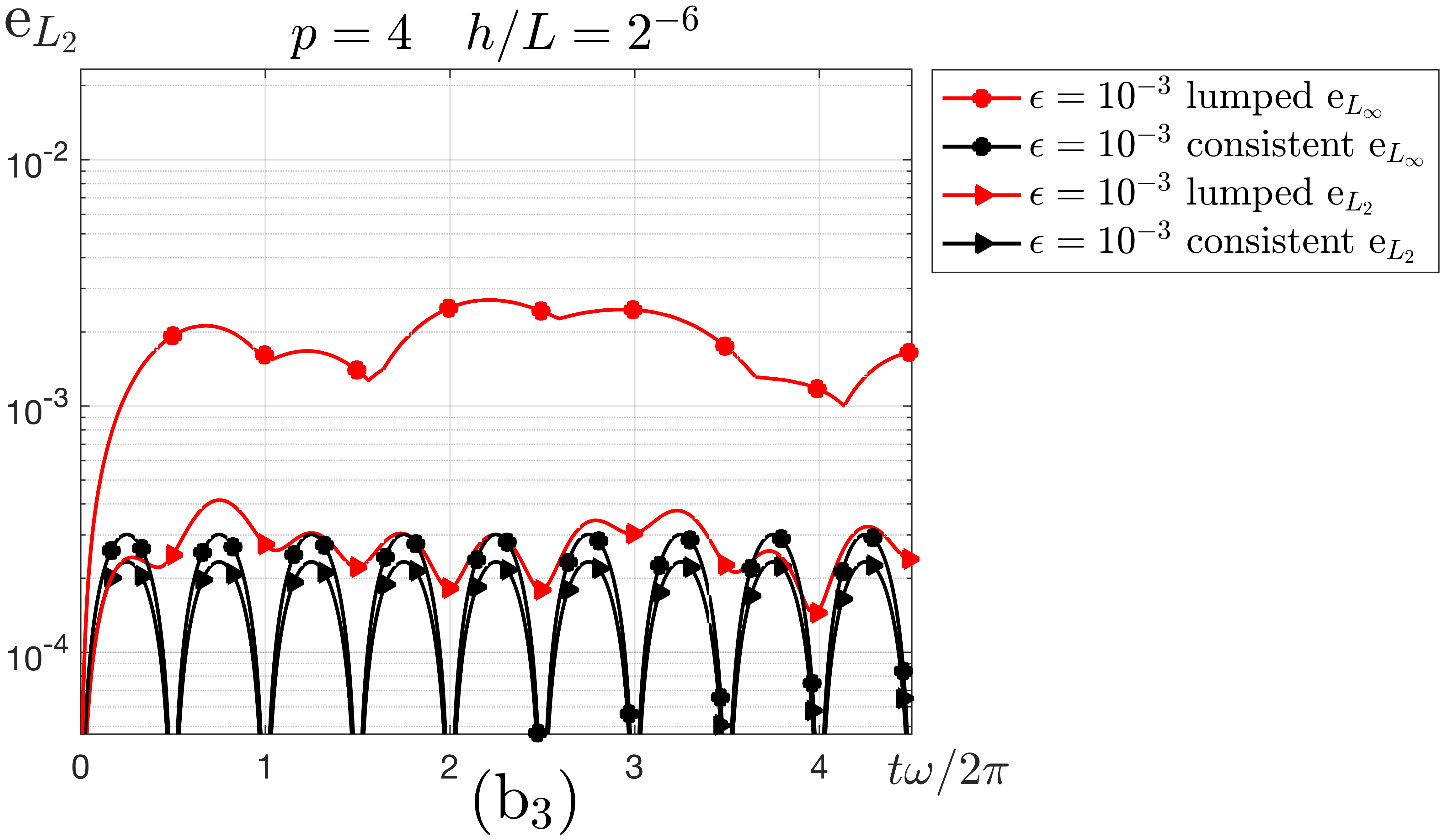}
    \end{subfigure}
    \caption{\Cref{ex: plate_externally_trimmed}: Comparison of the $L^2$ and $L^\infty$ errors for the consistent and lumped mass solutions for different refinement levels, spline orders and trimming parameters.}
    \label{fig: trimmed_square_plate_err_cons_lumped_refinement}
\end{figure}

\begin{figure}[H]	
    \centering
    \begin{subfigure}[b]{0.27\textwidth}\includegraphics[height=4.0cm]{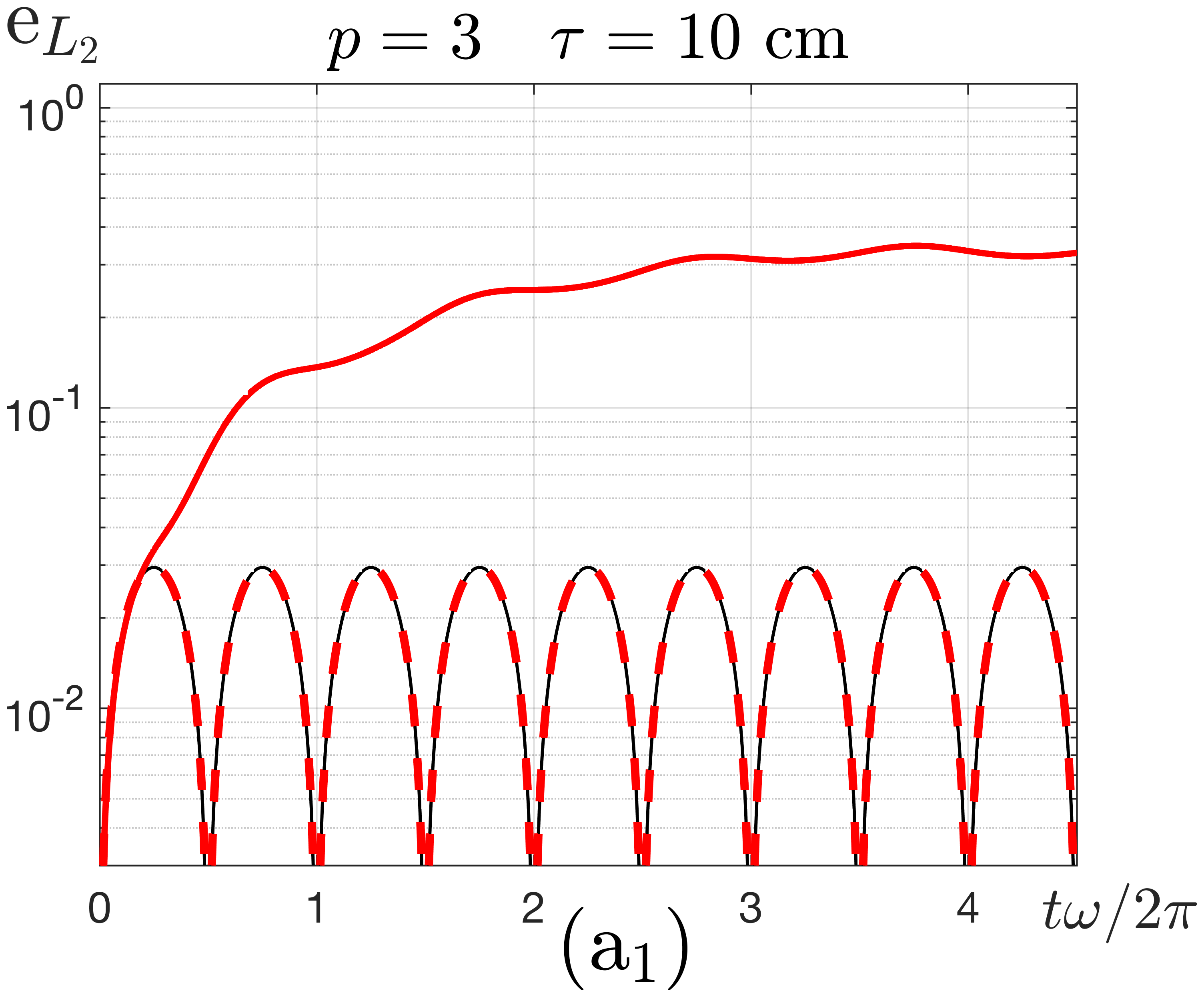}
    \end{subfigure}
    \begin{subfigure}[b]{0.27\textwidth}\includegraphics[height=4.0cm]{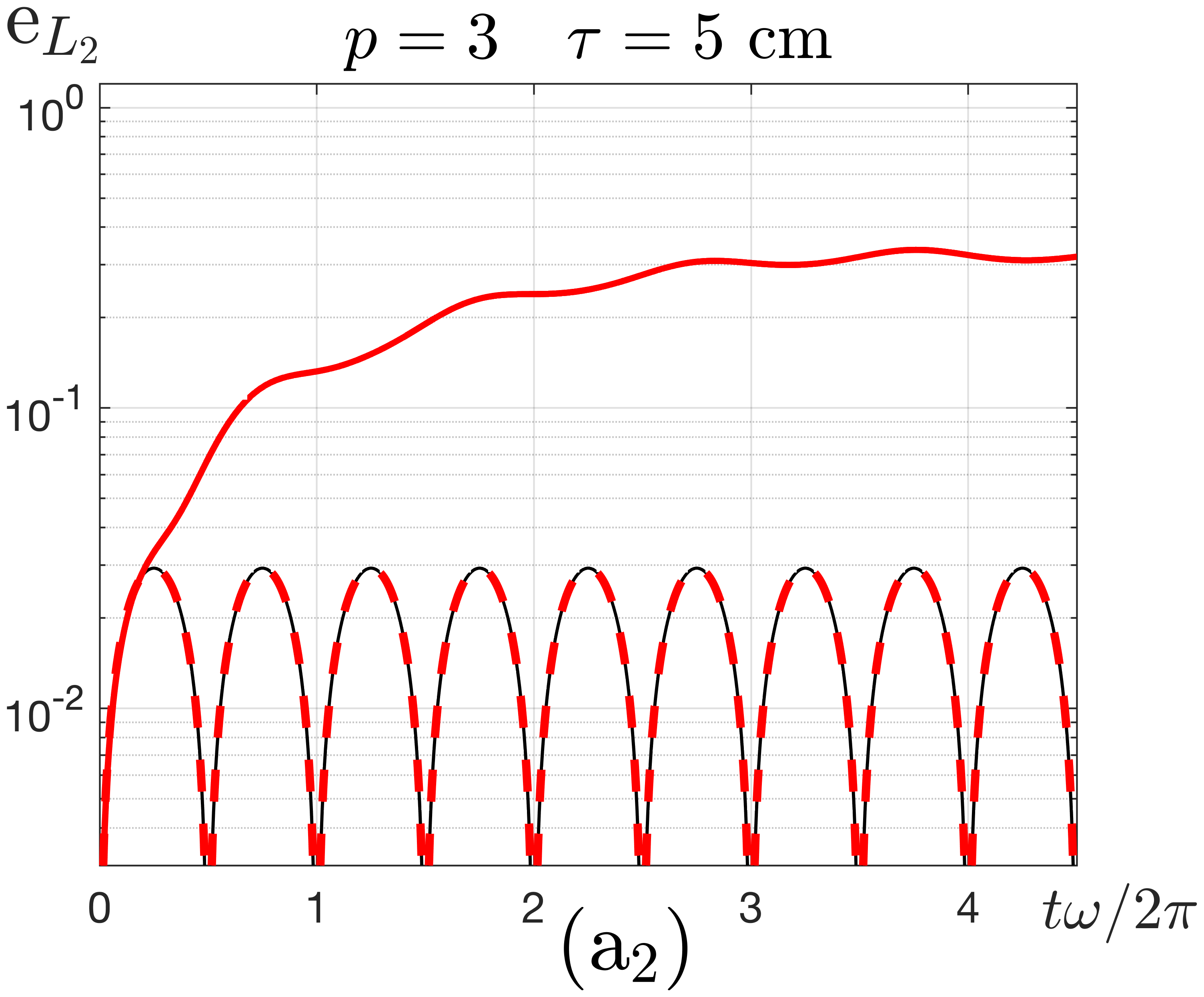}
    \end{subfigure}
    \begin{subfigure}[b]{0.44\textwidth}\includegraphics[height=4.0cm]{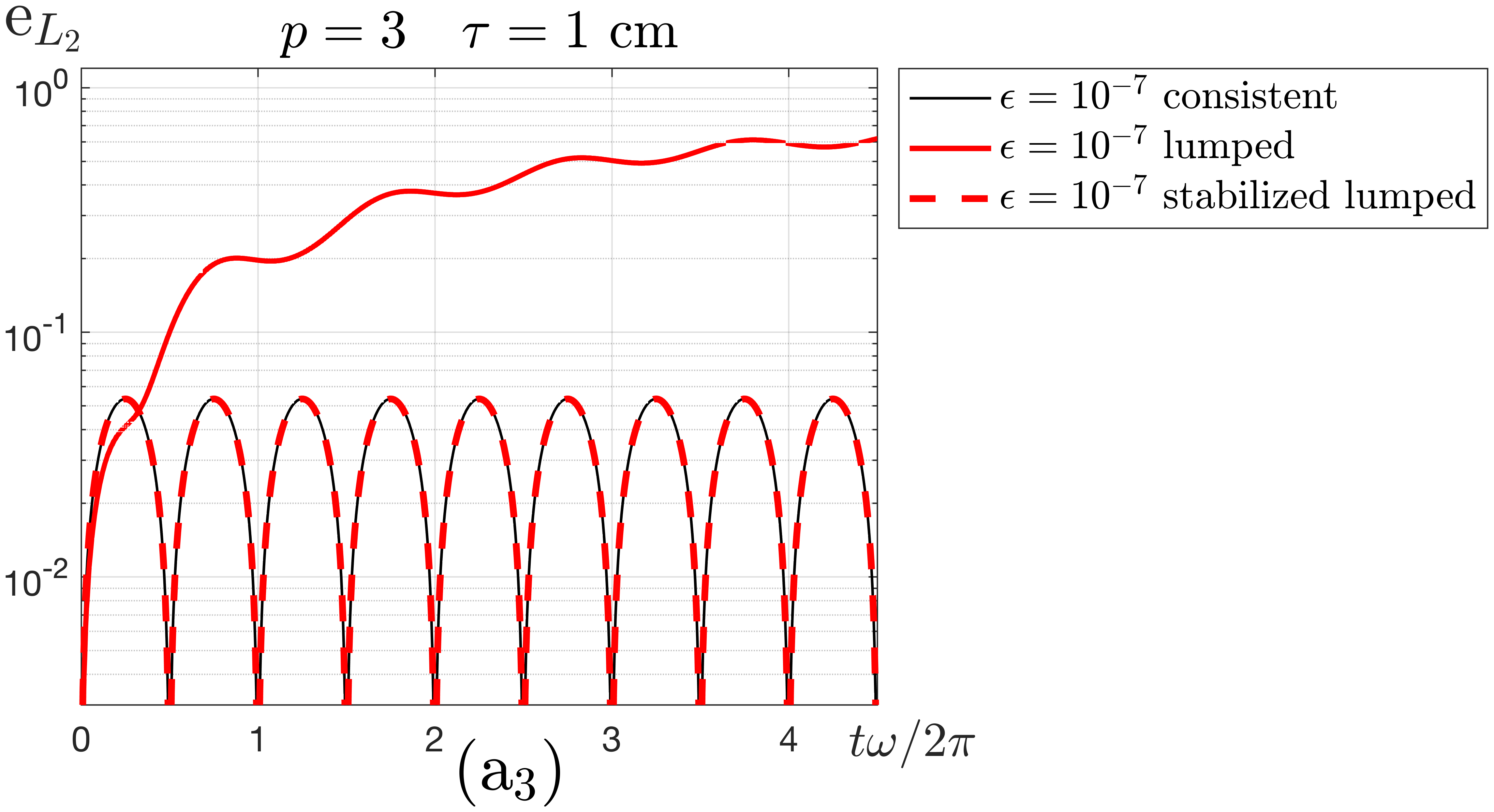}
    \end{subfigure}\\
    \begin{subfigure}[b]{0.27\textwidth}\includegraphics[height=4.0cm]{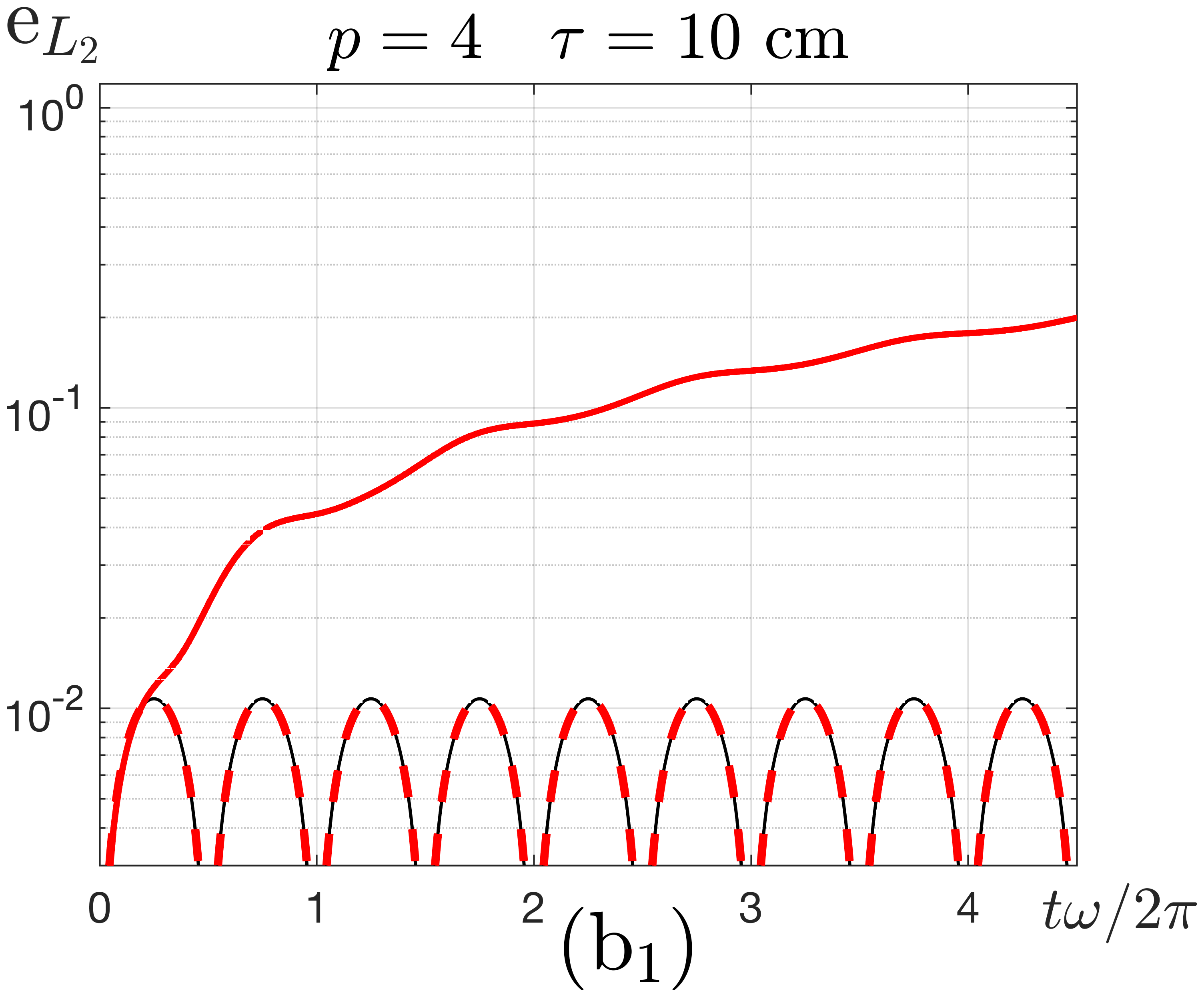}
    \end{subfigure}
    \begin{subfigure}[b]{0.27\textwidth}\includegraphics[height=4.0cm]{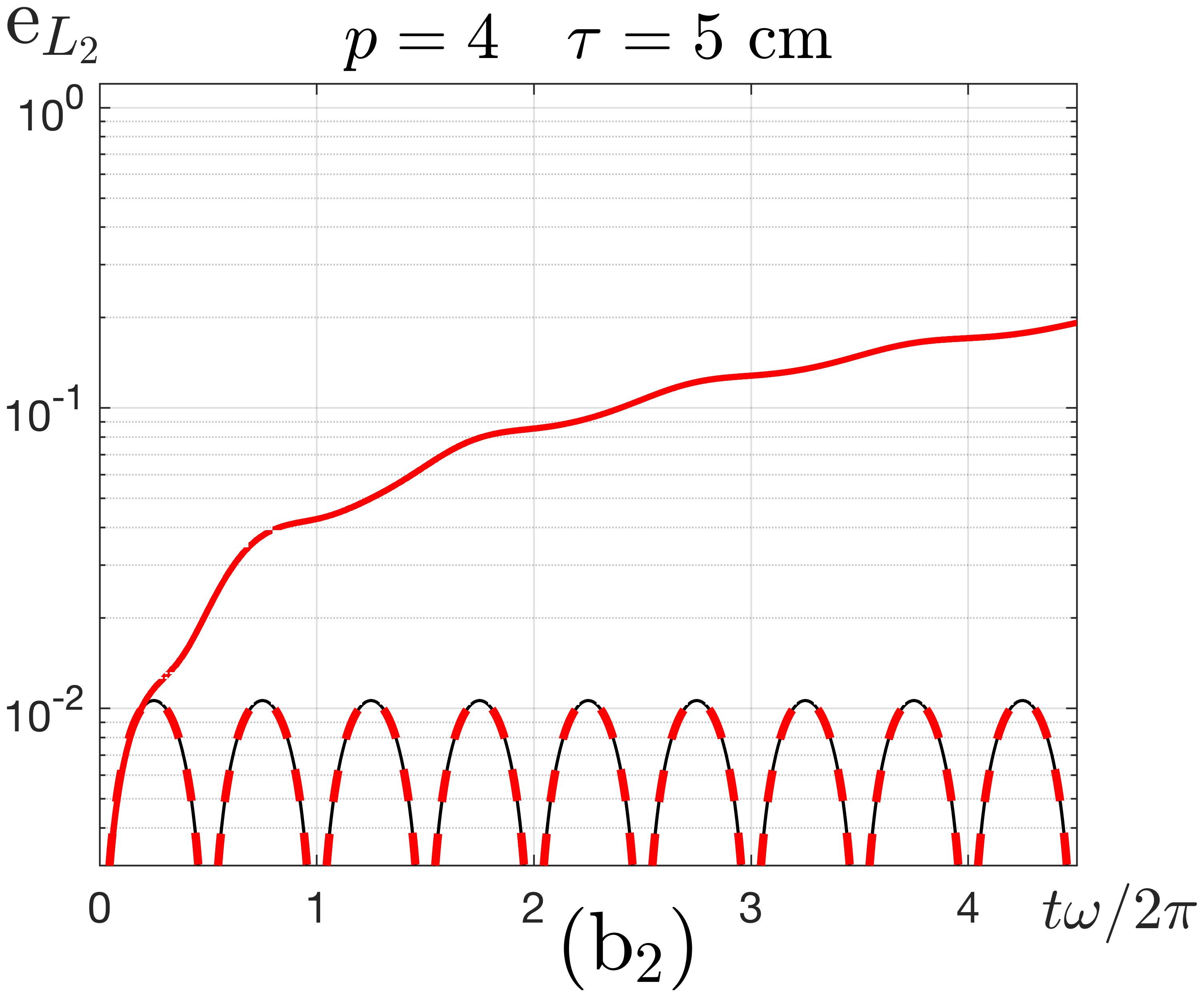}
    \end{subfigure}
    \begin{subfigure}[b]{0.44\textwidth}\includegraphics[height=4.0cm]{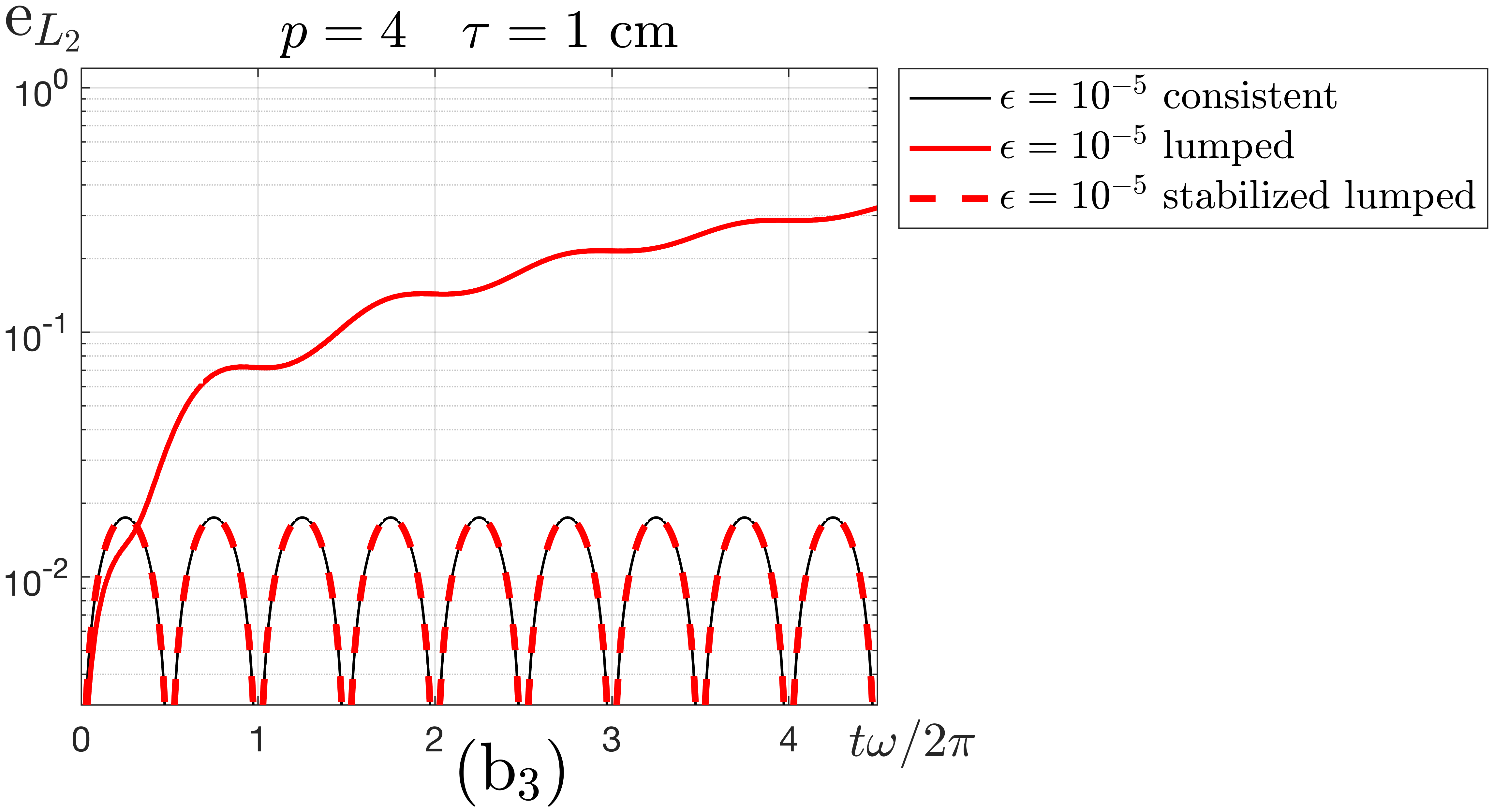}
    \end{subfigure}
    \caption{\Cref{ex: plate_externally_trimmed}: $L^2$ error for the consistent and (stabilized) lumped mass solutions for different spline orders, thicknesses and trimming parameters.}
    \label{fig: trimmed_square_plate_err_cons_stab_lumped}
\end{figure}

Simulations in explicit dynamics are a fine balance of accuracy and efficiency. On the one hand, the consistent mass oftentimes yields the best accuracy one can hope for, but is prohibitively expensive. On the other hand, the lumped mass drastically improves the efficiency if one is willing to give up on some accuracy. Admirably, the stabilized lumped mass retains the efficiency of the (non-stabilized) lumped mass while improving its accuracy. Indeed, as shown in \Cref{fig: trimmed_square_plate_cfl}, the critical time step of neither the lumped mass nor its stabilized counterpart depend on the trimming parameter, contrary to the (non-stabilized) consistent mass. For the sake of completeness, we have also included results for the stabilized consistent mass. While stabilizing the discrete formulation eliminates the deleterious effect trimming has on the spectrum, all other dependencies on the discretization parameters remain unchanged, particularly for the largest eigenvalue indirectly shown in \Cref{fig: trimmed_square_plate_cfl}. Indeed, as for boundary-fitted discretizations, the critical time step still decreases as the spline order increases or as the thickness and the mesh size decrease.

\begin{figure}[H]	
    \centering
    \begin{subfigure}[b]{0.27\textwidth}\includegraphics[height=4.0cm]{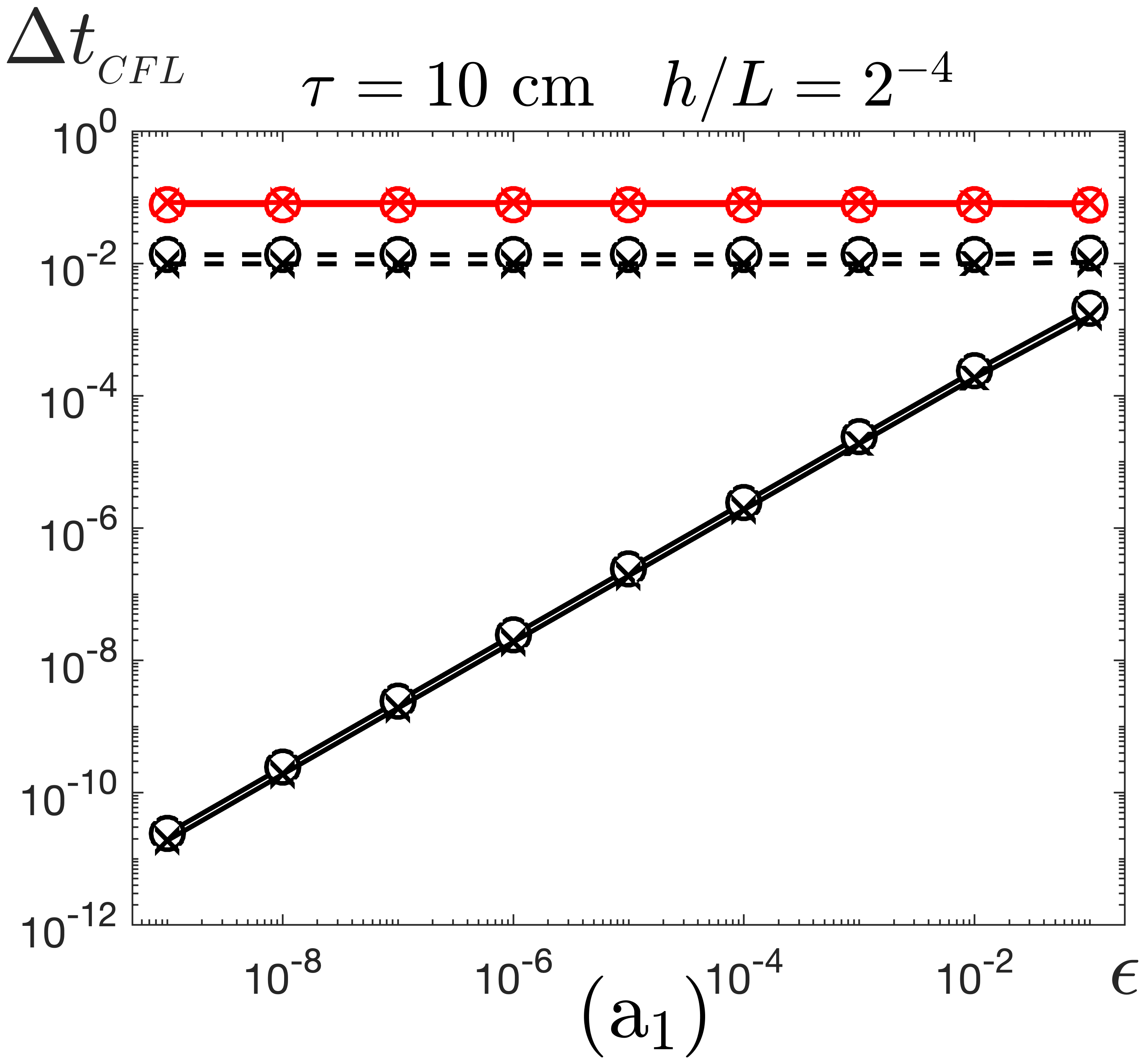}\end{subfigure} 
    \begin{subfigure}[b]{0.27\textwidth}\includegraphics[height=4.0cm]{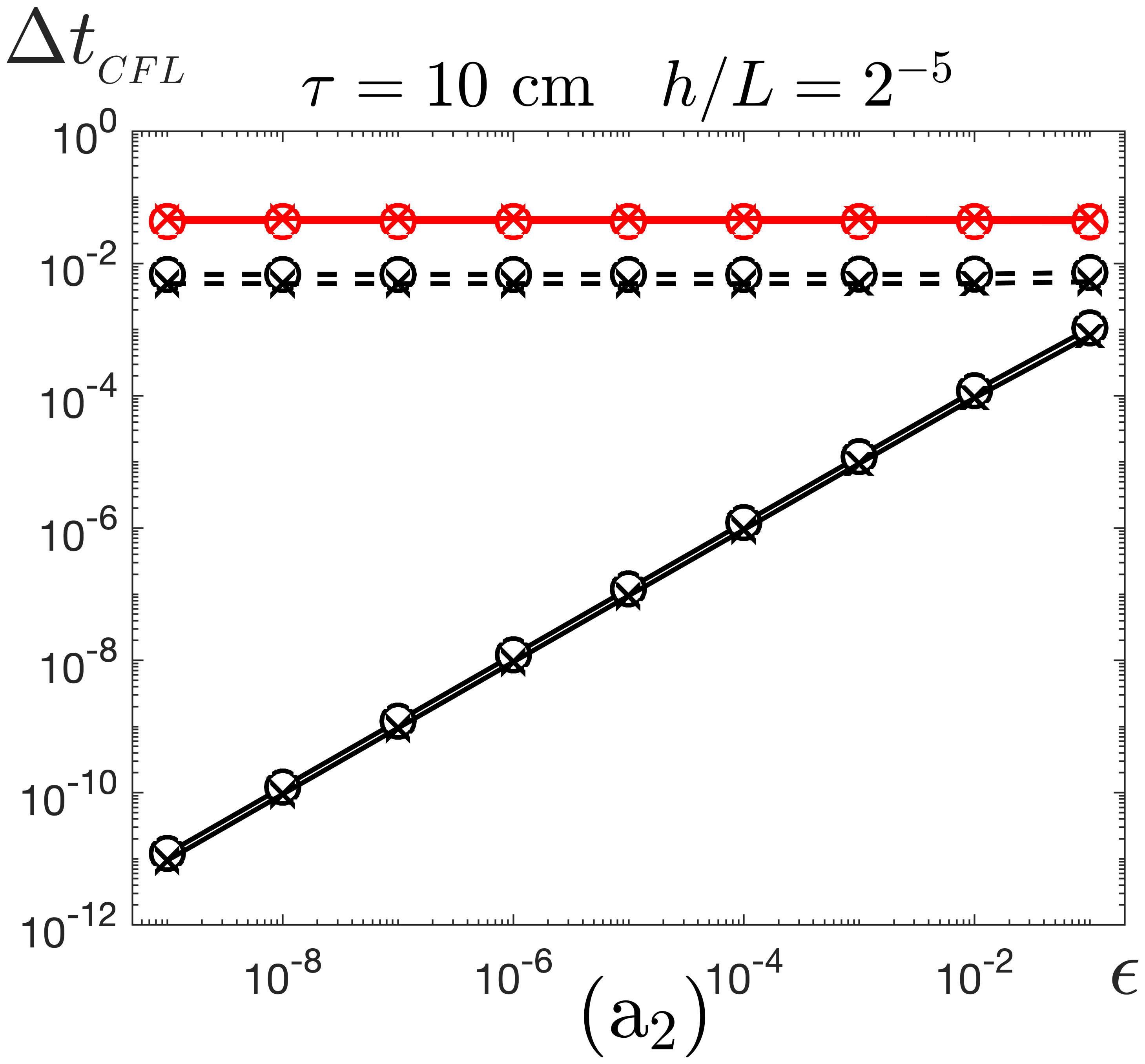}\end{subfigure} 
    \begin{subfigure}[b]{0.44\textwidth}\includegraphics[height=4.0cm]{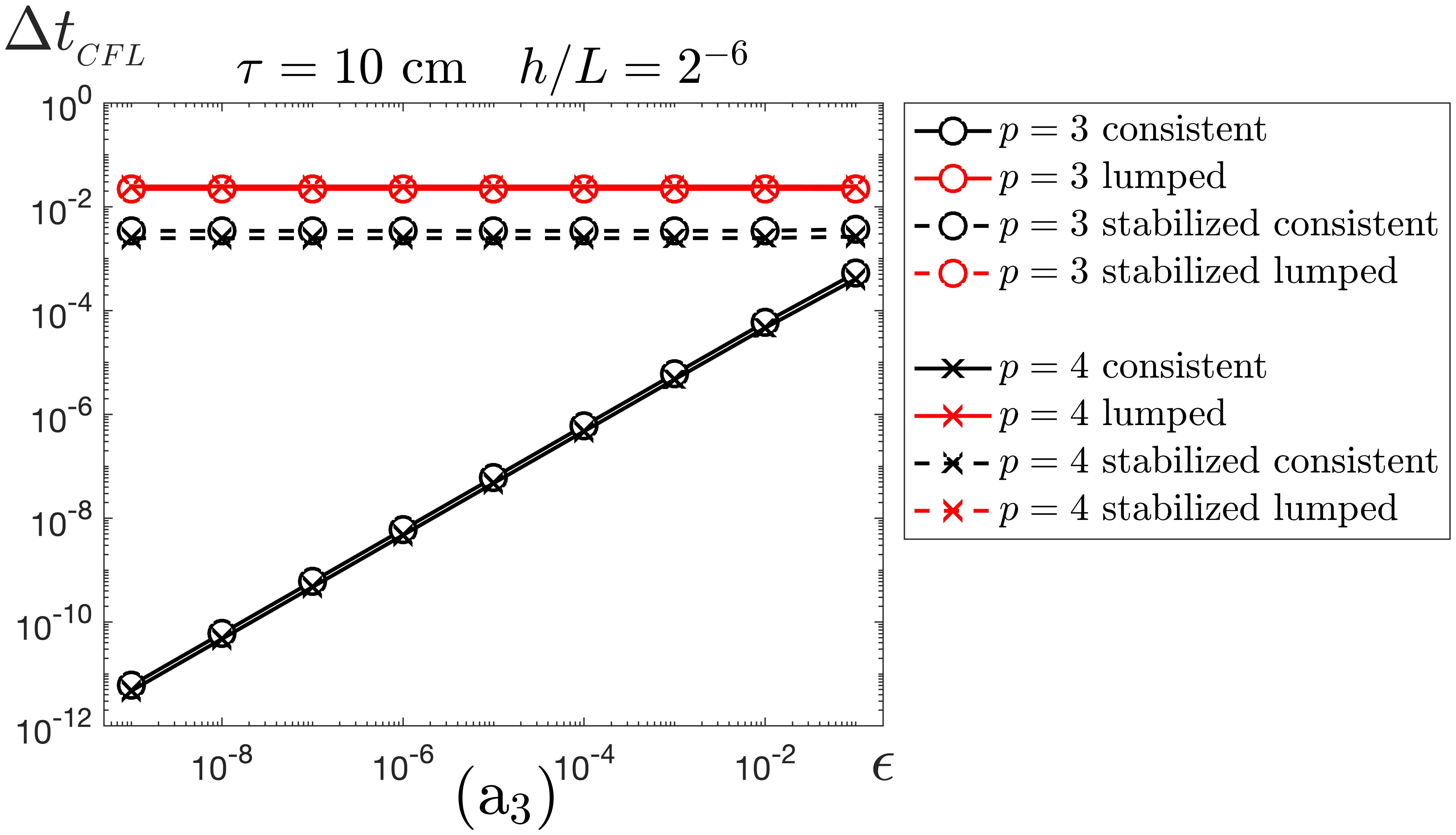}\end{subfigure}\\
    \begin{subfigure}[b]{0.27\textwidth}\includegraphics[height=4.0cm]{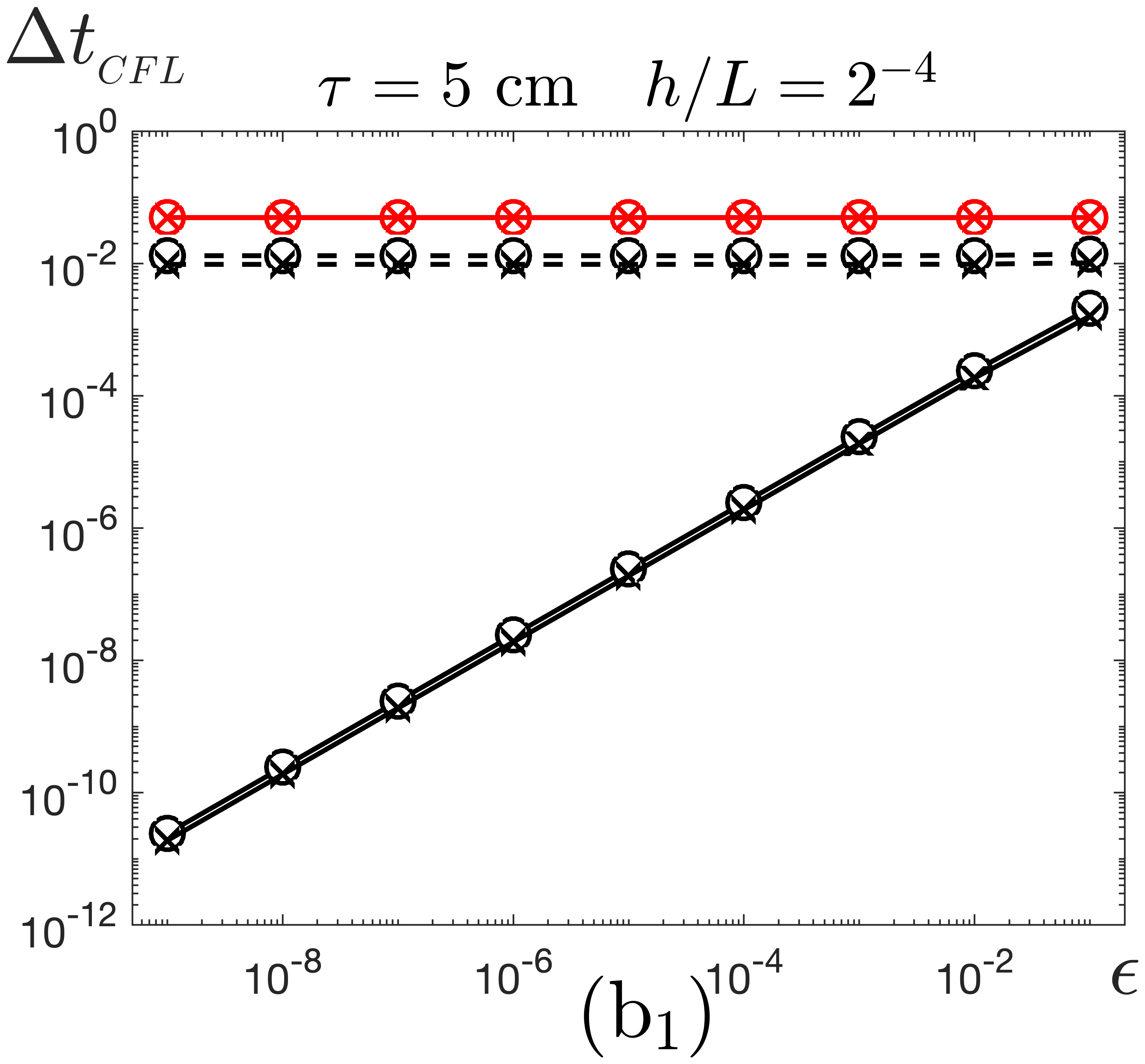}\end{subfigure} 
    \begin{subfigure}[b]{0.27\textwidth}\includegraphics[height=4.0cm]{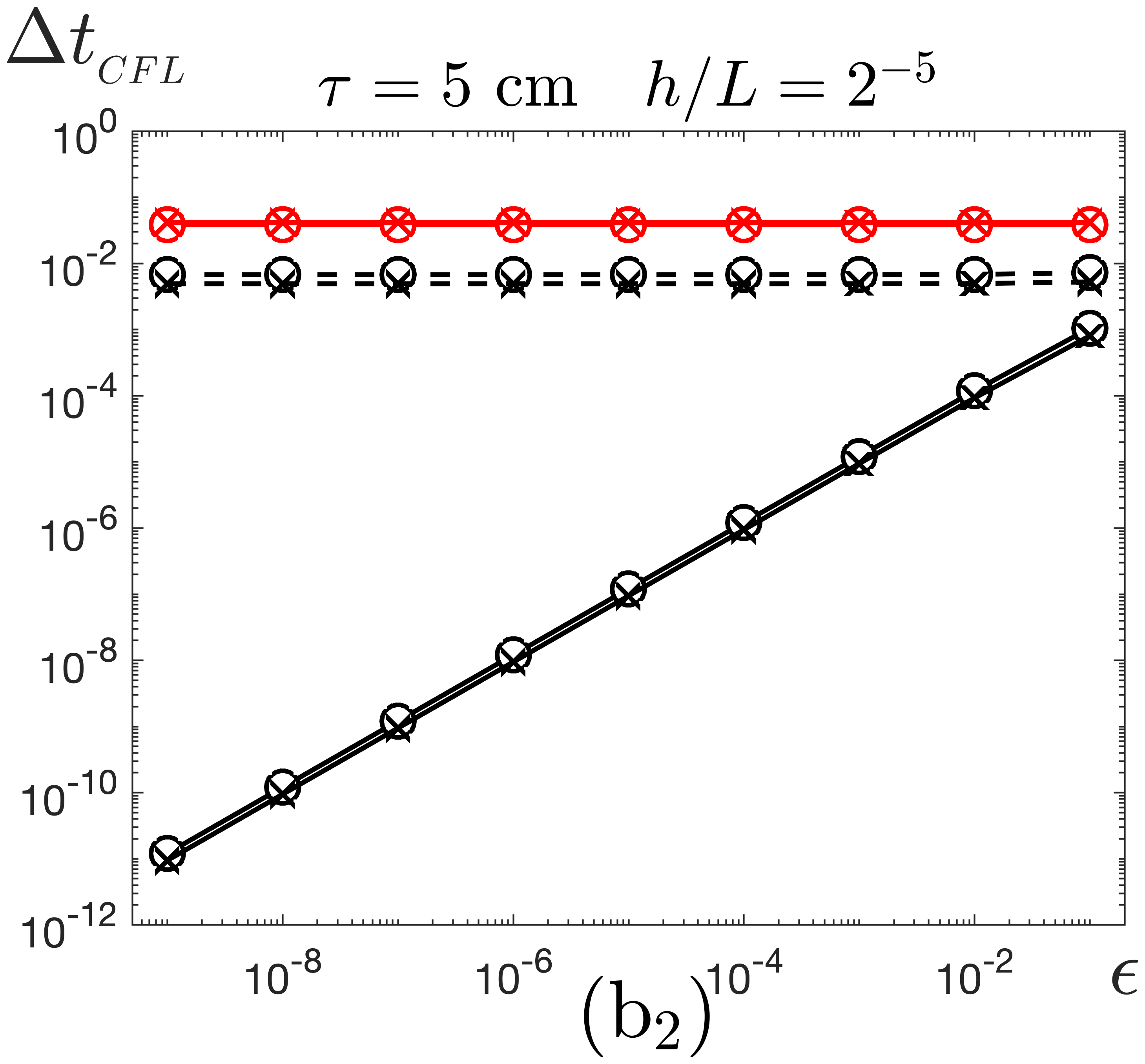}\end{subfigure} 
    \begin{subfigure}[b]{0.44\textwidth}\includegraphics[height=4.0cm]{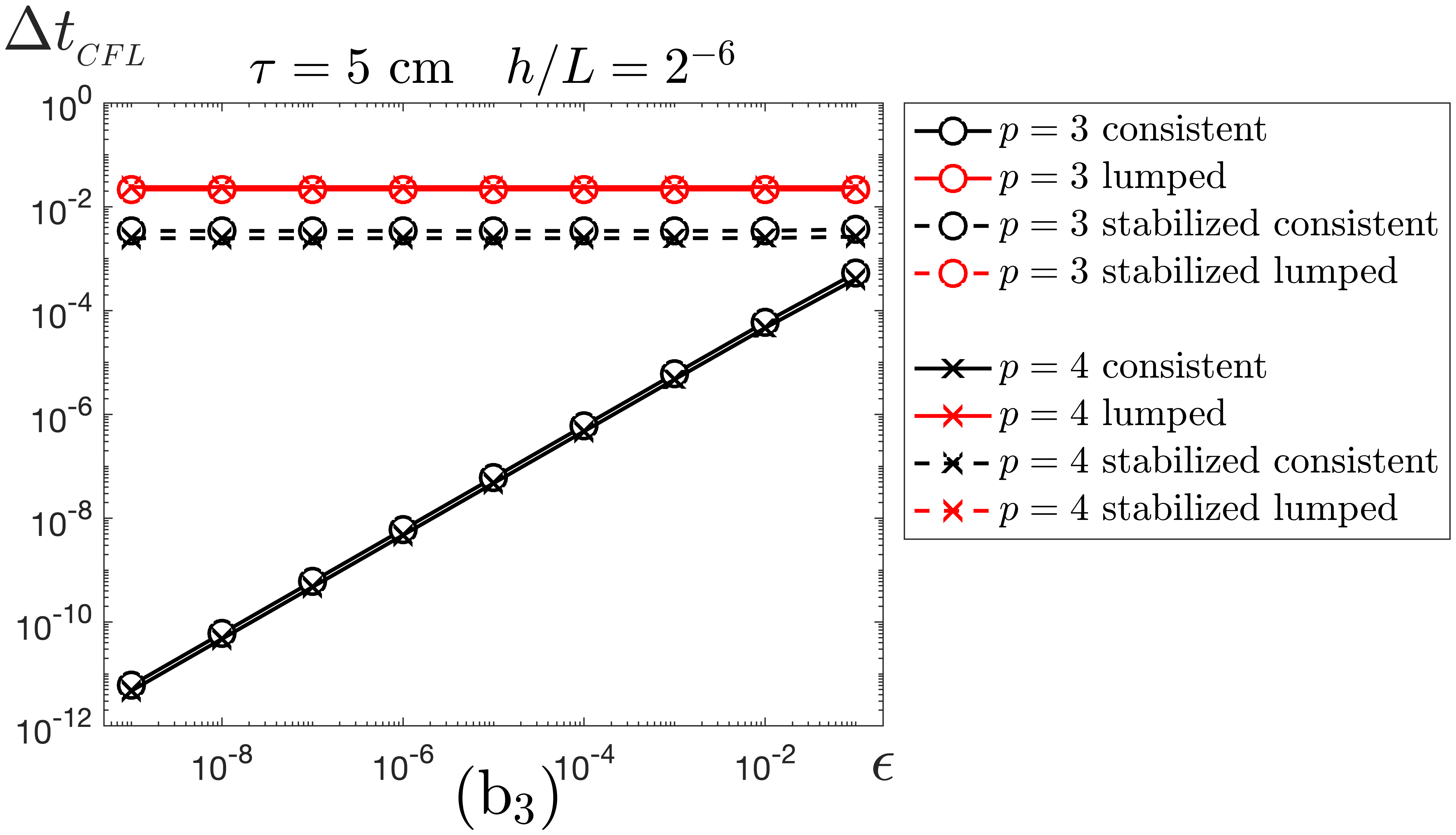}\end{subfigure}\\ 
    \begin{subfigure}[b]{0.27\textwidth}\includegraphics[height=4.0cm]{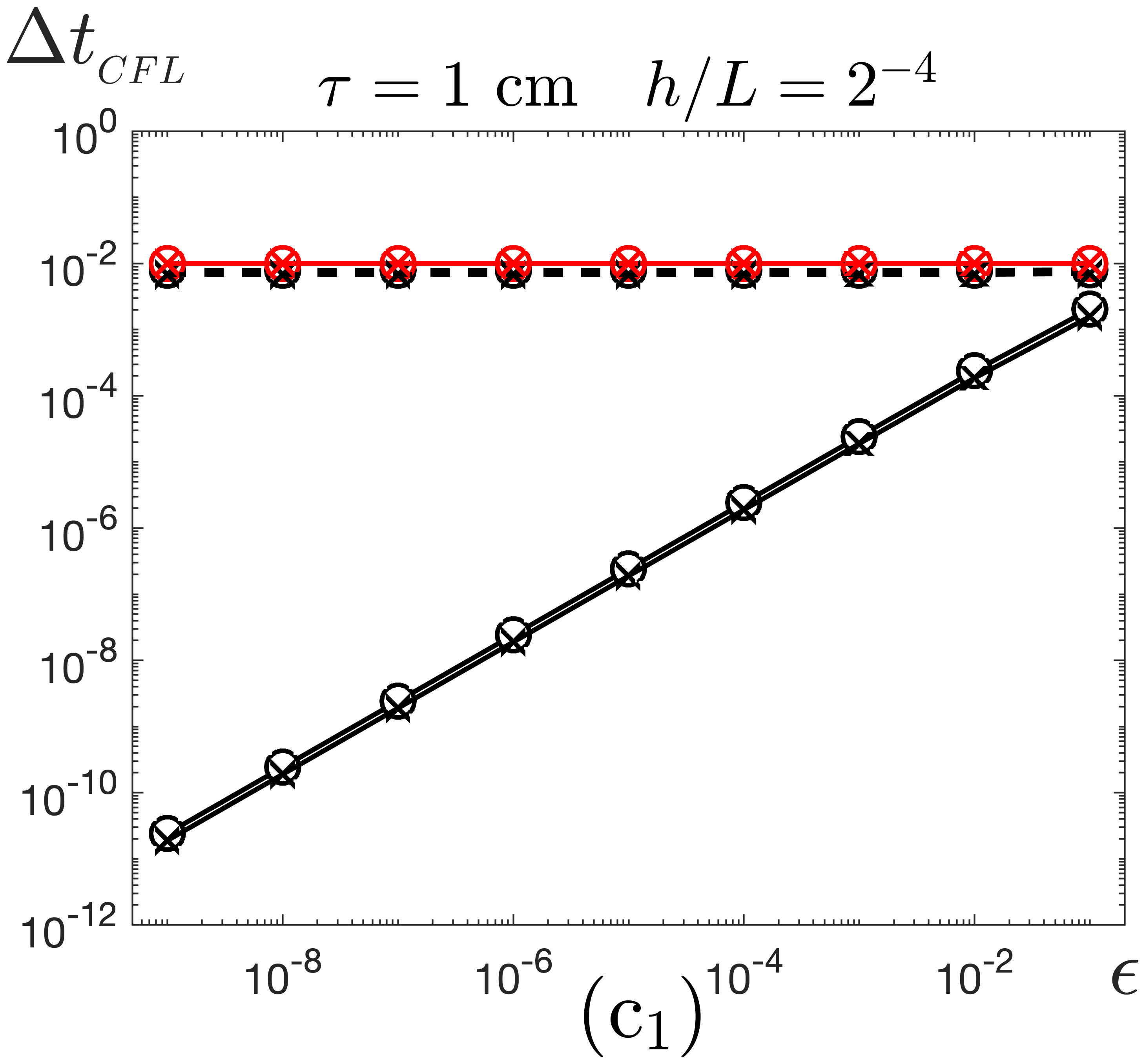}\end{subfigure} 
    \begin{subfigure}[b]{0.27\textwidth}\includegraphics[height=4.0cm]{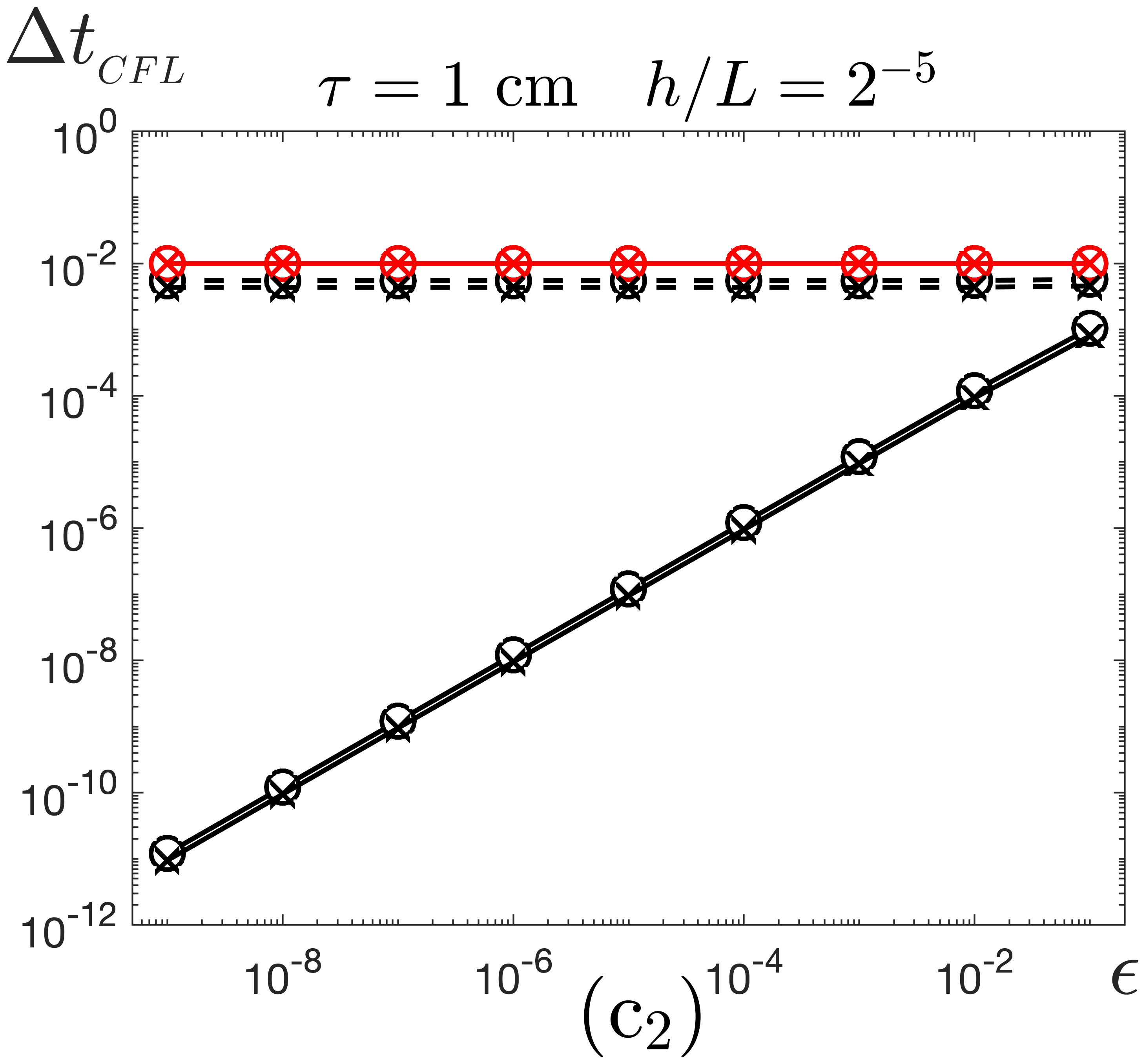}\end{subfigure} 
    \begin{subfigure}[b]{0.44\textwidth}\includegraphics[height=4.0cm]{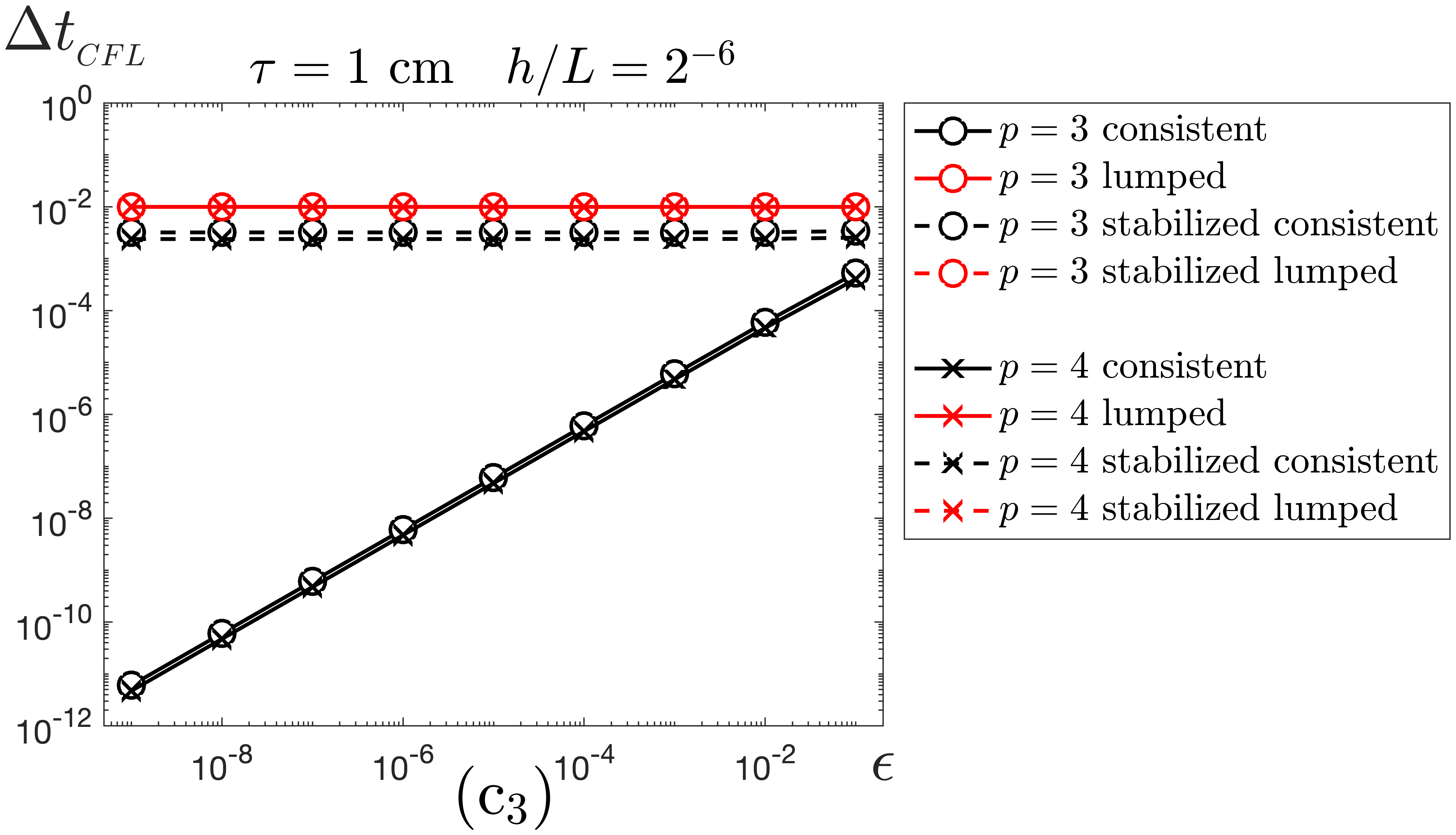}\end{subfigure}
    \caption{\Cref{ex: plate_externally_trimmed}: CFL condition for different refinement levels, spline orders and thicknesses.} 
    \label{fig: trimmed_square_plate_cfl}
\end{figure}
\end{example}

\begin{example}[Rotated plate]
\label{ex: rotated_square}
As we will show, relatively simple examples are not exempt from developing spurious oscillations. A square plate of thickness $\tau=5$ cm initially centered at the origin is trimmed along its bottom and right edges before being rotated $45 \degree$ counterclockwise (see \Cref{fig: rotated_plate}). Neumann (or type 2) boundary conditions are prescribed along its two trimmed edges while Dirichlet (or type 1) boundary conditions are enforced on the remaining two edges. Contrary to the previous example, we prescribe boundary data and initial conditions instead of computing them from a manufactured solution. The shell's ``external forces'' $\bm{\mathsf{f}} = (\bm{f}, \bm{m})$ are
\begin{align*}
    \bm{f}(\bm{x},t) &= f_0\; e^{b_0 x_1}\sin{(\pi n_0 x_1)} \phi(t) \;\bm{e}_3, \\
    \bm{m}(\bm{x},t) &= -m_0 \; e^{b_0 x_1} \frac{b_0\sin{(\pi n_0x_1)} - n_0\pi\cos{(\pi n_0 x_1)}}{b_0^2 + n_0^2\pi^2}\phi(t) \;\bm{e}_1,
\end{align*}
and the ``boundary tractions'' $\bm{\mathsf{h}} = (\bm{h}, \bm{n})$ are
\begin{align*}
    \bm{h}(\bm{x},t) &= -h_0 \; e^{b_0 x_1} \frac{b_0\sin{(\pi n_0x_1)} - n_0\pi\cos{(\pi n_0 x_1)}}{b_0^2 + n_0^2\pi^2}\phi(t) \;(\bm{e}_1 \cdot \bm{r}) \bm{e}_3, \\
    \bm{n}(\bm{x},t)&=\bm{0},
\end{align*}
where $\bm{r}$ denotes the outward unit normal vector and $\phi(t)=\sin(\omega t)$. The parameter values are $f_0=0.1$ $\mathrm{N/m^2}$, $m_0=0.1$ $\mathrm{N/m}$, $h_0=0.1$ $\mathrm{N/m^2}$, $b_0=6$, $n_0=28$ and $\omega=\frac{1}{10L}\sqrt{\frac{E}{\rho}}$ with unit elastic modulus $E$ and density $\rho$ and Poisson ratio $\nu=0.25$. Finally, the initial conditions are
\begin{align*}
    &\bm{u}(\bm{x},0) = \bm{0}, \\
    &\dot{\bm{u}}(\bm{x},0) = v_0 \; e^{b_0 x_1}\sin{(\pi n_0 x_1)} \;\bm{e}_3,\\
    &\bm{\theta}(\bm{x},0) = \bm{0}, \\
    &\dot{\bm{\theta}}(\bm{x},0) = \bm{0}
\end{align*}
where $v_0=7.74 \times 10^{-5}$ $\mathrm{m/s}$. The background mesh discretizing the geometry aligns with the plate's boundaries, whose trimmed edges are only $\delta$ away from the nearest grid lines. The ratio $\epsilon= \delta/h$, where $h$ denotes the mesh size, measures the severity of trimming. In this experiment, we have set $\epsilon=10^{-9}$ and $L=1$ m. \figref{fig: rotated_plate_sol}{e$_1$-e$_2$} shows a reference solution computed over a fine mesh at two distinct times $t_2$ and $t_3$, as indicated in \figref{fig: rotated_plate_sol}{a}. The numerical solutions are computed with exactly the same strategy as in the previous example. Namely, the Central Difference method for the lumped mass solutions and an unconditionally stable Newmark method for the consistent mass with the same step size. As shown in \figref{fig: rotated_plate_sol}{c$_1$-c$_2$} and \figref{fig: rotated_plate_sol}{s$_1$-s$_2$}, the solutions for the consistent and stabilized lumped mass closely resemble the reference solution. On the contrary, the solution for the (non-stabilized) lumped mass once again features spurious oscillations that are particularly pronounced at the corner of the two trimmed edges. This example proves that these oscillations can occur even for relatively simple expressions of the prescribed data and, therefore, they could a priori arise in practical industrial applications. As shown in \figref{fig: rotated_plate_sol}{s$_1$-s$_2$}, stabilizing the discrete space prior to lumping the mass is a very convenient way of ruling out this possibility.

\begin{figure}[H]	
    \centering
    \includegraphics[width=0.5\textwidth]{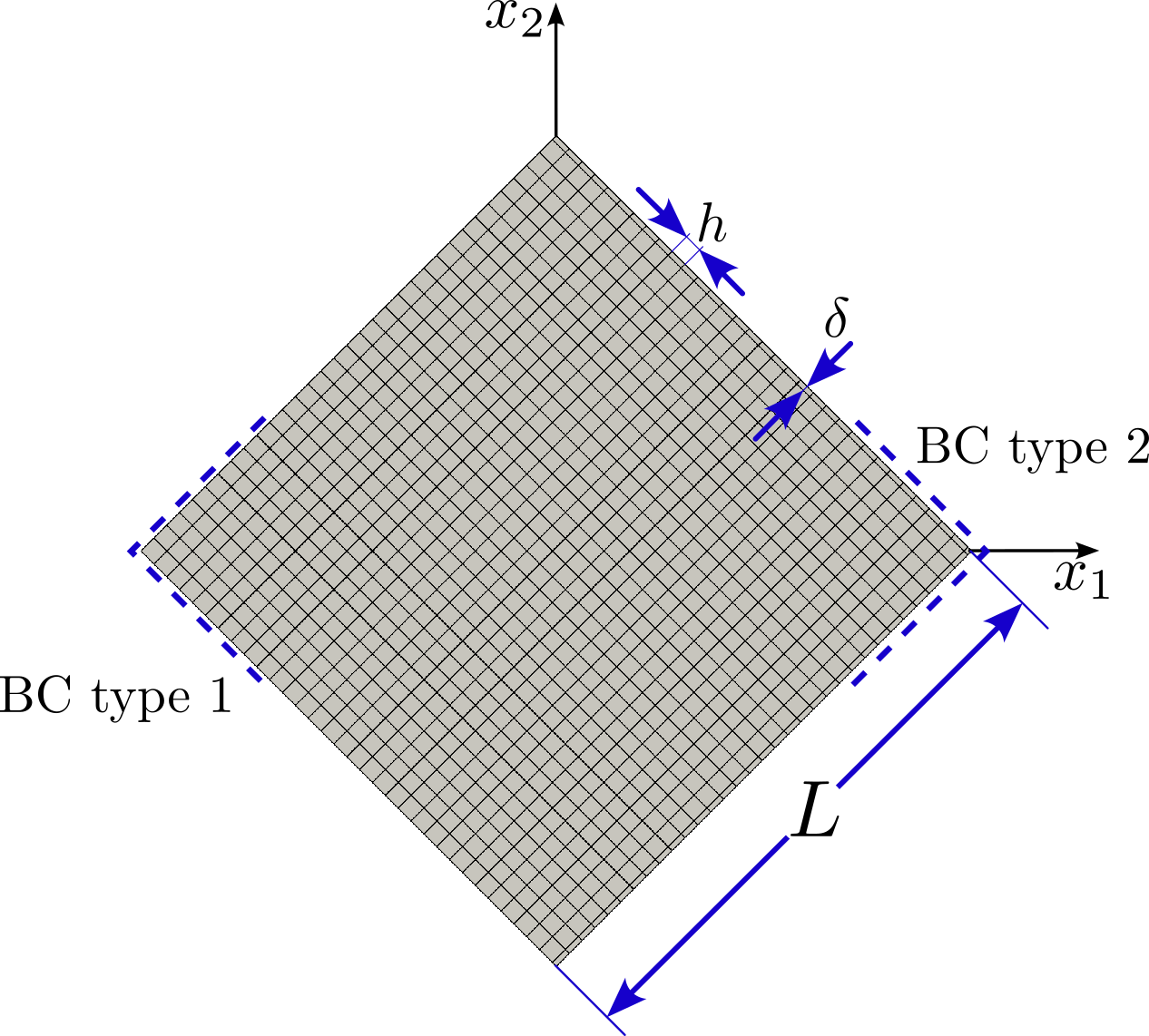}
    \caption{\Cref{ex: rotated_square}: Rotated plate.}
    \label{fig: rotated_plate}
\end{figure}

\begin{figure}[H]	
    \centering
    \includegraphics[width=0.7\textwidth]{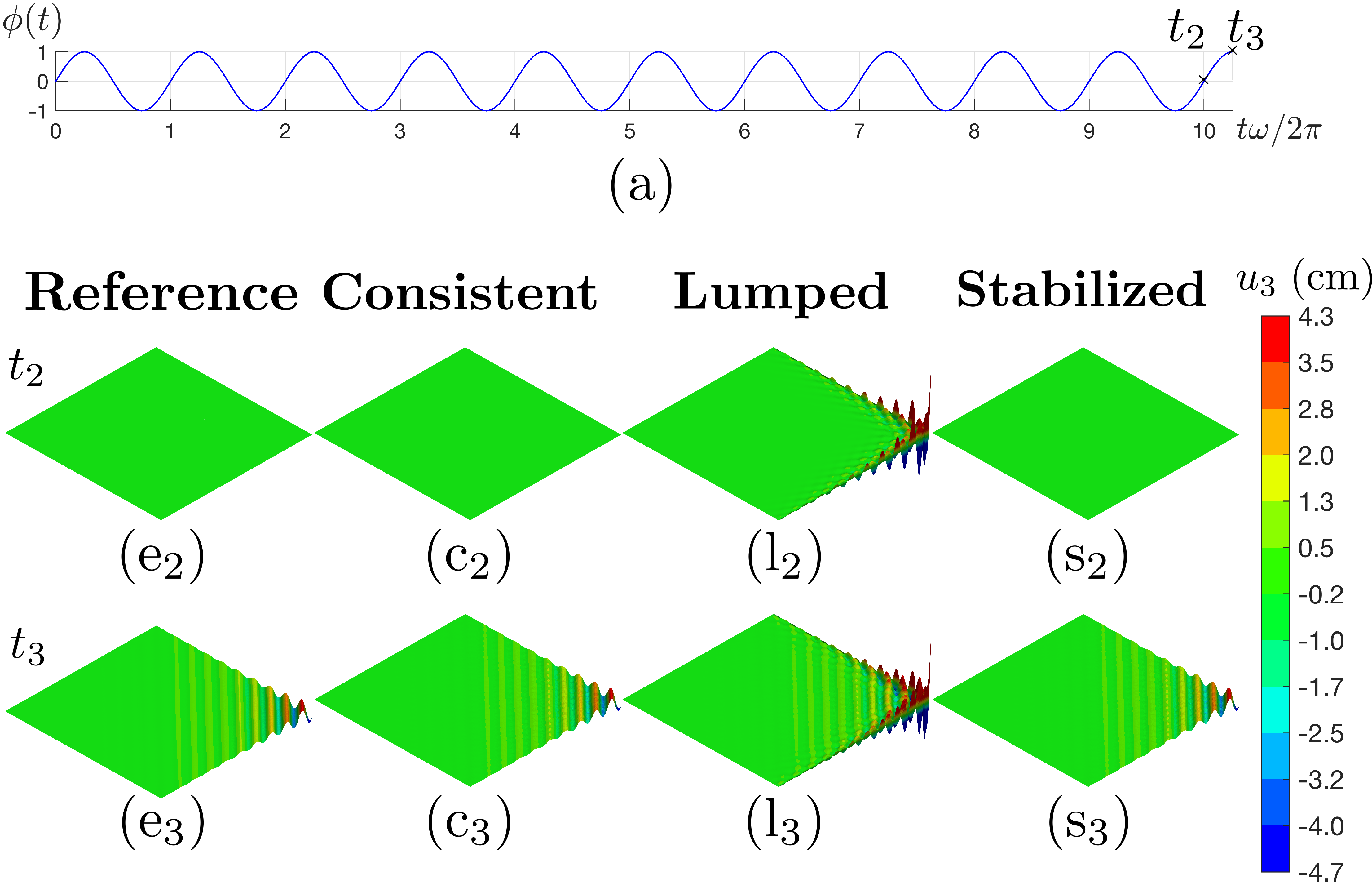}
    \caption{\Cref{ex: rotated_square}: Snapshots of the reference and numerical solutions.}
    \label{fig: rotated_plate_sol}
\end{figure}
\end{example}

% \begin{align*}
%     &\bm{f}(\bm{x},t) = -f_0\; e^{b_0 x_1}\sin{(\pi n_0 x_1)} \phi(t) \;\bm{e}_3\\
%     &\bm{m}(\bm{x},t) =  m_0 \; e^{b_0 x_1} \frac{b_0\sin{(\pi n_0x_1)} - n_0\pi\cos{(\pi n_0 x_1)}}{b_0^2 + n_0^2\pi^2}\phi(t) \;\bm{e}_1\\
%     &\bm{h}(\bm{x},t) = h_0 \; e^{b_0 x_1} \frac{b_0\sin{(\pi n_0x_1)} - n_0\pi\cos{(\pi n_0 x_1)}}{b_0^2 + n_0^2\pi^2}\phi(t) \;(\bm{e}_1 \cdot \bm{n}) \bm{e}_3
% \end{align*}

% $\bm{n}$ is the outer normal unit vector. I am not sure how to call it... I saw that $\bm{n}$ was reserved for the external moment applied on the boundary \yannis{Use $\bm{r}$?}

% \begin{align*}
%     &\bm{u}(\bm{x},0) = \bm{0} \\
%     &\dot{\bm{u}}(\bm{x},0) = v_0 \; e^{b_0 x_1}\sin{(\pi n_0 x_1)} \;\bm{e}_3
% \end{align*}

% $\phi(t)=\sin(\omega t)$

% $\tau=5$ cm

% $f_0=0.1$ $\mathrm{N/m^2}$

% $h_0=0.1$ $\mathrm{N/m^2}$

% $m_0=0.1$ $\mathrm{N/m}$

% $v_0=7.7400\cdot 10^{-5}$ $\mathrm{m/s}$ 

% $b_0=6$

% $n_0=28$

% $L=1$ m

% $\epsilon=10^{-9}$

% The reference solution is obtained with a mesh with a higher refinement level. 

\begin{example}[Plate with cut-out]
\label{ex: plate_with_cut-out}
This third example slightly modifies the first one by creating a square cut-out in the middle of the plate of side length $2a=0.4$ m (\Cref{fig: plate_with_cut_out}). Moreover, the inner boundary is trimmed instead of the outer one. The manufactured solution is modified accordingly as
\begin{equation}
\label{eq: exact_sol_cut_out}
    \bm{u}(\bm{\xi},t) = w(\bm{\xi})\phi(t) \bm{e}_3, \quad \text{and} \quad \bm{\theta}(\bm{\xi},t) = - \sum_{i=1}^2 w_{,i}(\bm{\xi})\phi(t)\bm{a}^i,
\end{equation}
where the spatial part $w \colon S \to \mathbb{R}$ and temporal part $\phi \colon [0,\; +\infty) \to \mathbb{R}$ are defined as
\begin{equation*}
    w(\bm{\xi}) = a_0 w_0 w_n \circ R(\bm{\xi}), \quad \text{and} \quad \phi(t) = \sin(\omega t),
\end{equation*}
with the functions
\begin{align*}
    w_0(y) &= e^{(1-y^2)}, \\
    w_n(y) &= \cos(2\pi n_0 y), \\
\end{align*}
and
\begin{equation*}
    R(\bm{\xi}) = \sqrt[n]{\frac{1}{2}\left[\left(\frac{\xi_1+\xi_2}{a}\right)^n +\left(\frac{\xi_1-\xi_2}{a}\right)^n \right]}
\end{equation*}
with parameter values $a_0 = 10$ cm, $n_0=4$, $n=6$, Poisson ratio $\nu=0.25$ and $\omega = \frac{1}{2L}\sqrt{\frac{E}{\rho}}$ with unit Young modulus $E$ and density $\rho$.

Dirichlet (or type 1) boundary conditions are prescribed along the outer boundary while Neumann (or type 2) boundary conditions are prescribed along the inner boundary. The exact solution \eqref{eq: exact_sol_cut_out} is depicted in \figref{fig: plate_with_cut_out_sol}{e$_1$-e$_3$} at three distinct times $t_1$-$t_3$. The plate is discretized with maximally smooth B-splines over a grid of mesh size $h$. The inner boundaries are $\delta$ away from the nearest grid lines and $\epsilon=\delta/h$ again denotes the relative trimming parameter. The solutions for the consistent mass and (stabilized) lumped mass, computed with the same methods as in the first example, are shown in \Cref{fig: plate_with_cut_out_sol} for times $t_1$-$t_3$. Shortly after the beginning of the simulation (at time t$_1$), all solutions look reasonably accurate. However, oscillations gradually develop in the solution for the lumped mass and amplify over time. Nevertheless, we do not expect these oscillations to indefinitely grow over time. Based on the results in \cite{voet2025stabilization}, the error should still behave periodically, albeit with an amplitude and period that depends on both $\epsilon$ and $p$. Indeed, the pattern shown in \Cref{fig: trimmed_plate_with_cut_out_err_cons_lumped} corroborates this finding. For relatively large values of $\epsilon$, the error behaves periodically and we expect it still holds over larger time scales as $\epsilon$ becomes smaller. Moreover, by comparing \figref{fig: trimmed_plate_with_cut_out_err_cons_lumped}{a$_1$-a$_3$} to \figref{fig: trimmed_plate_with_cut_out_err_cons_lumped}{b$_1$-b$_3$}, increasing the degree exacerbates this trend. As explained in \cite{voet2025stabilization}, this finding correlates with the behavior of the smallest discrete eigenvalues, which were analyzed in \cite{bioli2025theoretical}. As shown in \Cref{fig: trimmed_plate_with_cut_out_err_cons_stab_lumped}, stabilizing the discrete formulation (nearly) fixes the issue. However, contrary to the first example, the error for the stabilized lumped mass does not perfectly overlap with the consistent mass for thick plates. We believe this is due to the inherent inaccuracy of the row-sum technique (independently of trimming) and advanced mass lumping strategies \cite{voet2025mass} might eliminate or mitigate this discrepancy.

% \giuliano{I actually think the reason for that might be in the jump of shape functions in stabilized elements that is more pronounced in this test in the internal corners}

\begin{figure}[H]	
    \centering
    \includegraphics[width=0.5\textwidth]{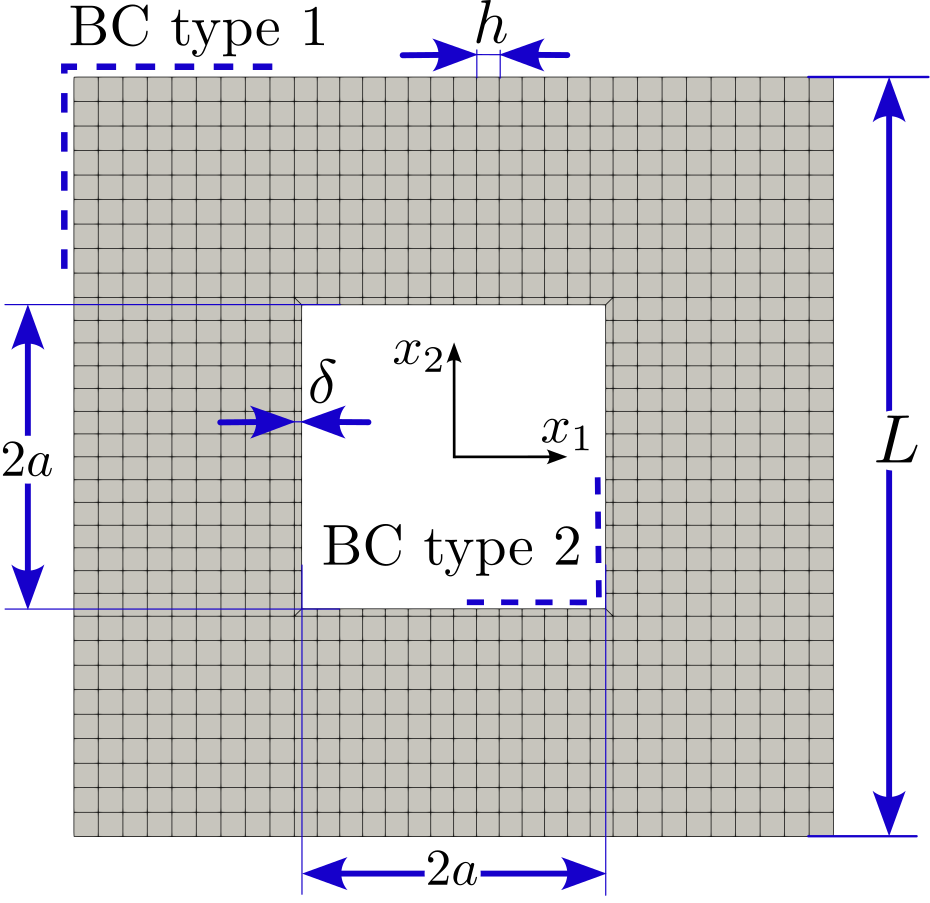}
    \caption{\Cref{ex: plate_with_cut-out}: Square plate with cut-out.}
    \label{fig: plate_with_cut_out}
\end{figure}

% \begin{itemize}
%     \item BC type 2: Neumann boundary conditions
%     \item BC type 1: Strong Dirichlet boundary conditions
%     \item $\epsilon=\delta/h$
%     \item $\omega=0.5\sqrt(\frac{E}{\rho L^2})$
%     \item $L=1$ m
%     \item $a=0.2$ m
%     \item manufactured displacement field $\bm{u} = \varphi \bm{e}_3$
%     \item manufactured rotation field $\bm{\theta} = -\varphi_{,\alpha}\bm{a}^\alpha$
% \end{itemize}

% The function $\varphi$ is constructed as 

% $R = \sqrt[n]{\frac{1}{2}\left[\left(\frac{{x}_1}{a}+\frac{{x}_2}{b}\right)^n +\left(\frac{{x}_1}{a}-\frac{{x}_2}{b}\right)^n \right]}$

% $\varphi = a_0 \varphi_0 \varphi_n$

% $\varphi_0 = e^{(1-R^2)}$

% $\varphi_n =  \cos\left(2\pi n_0 R \right)$

% The values of the parameters for this test are the following:

% \begin{tabular}{c|c}
% \hline
% \hline
% Parameter & Value \\  
% \hline
% $a_0$ & 10 cm \\ 
% $n_0$ & 4 \\
% $n$ & 6
% \end{tabular}

\begin{figure}[H]	
    \centering
    \includegraphics[width=0.7\textwidth]{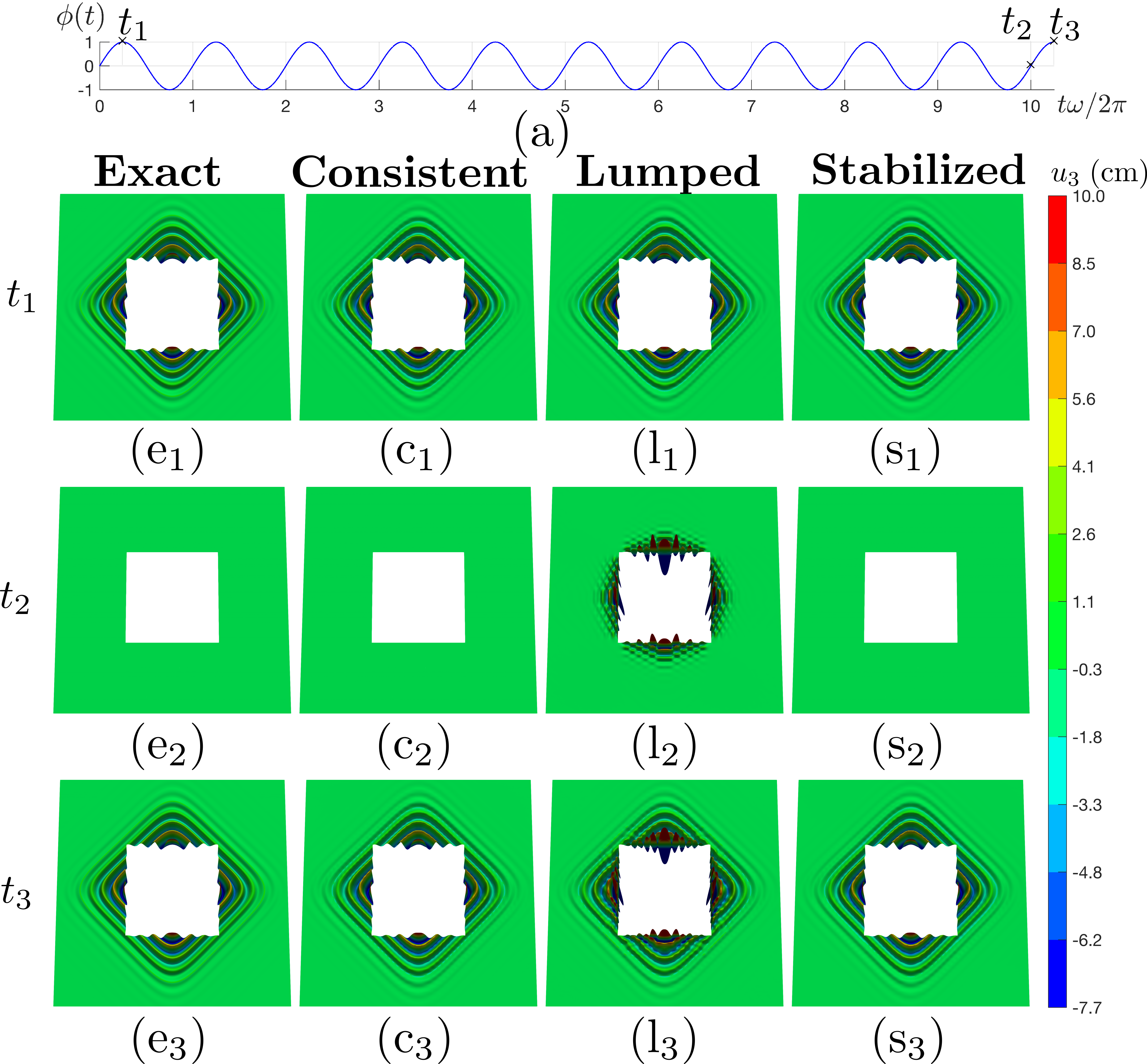}
    \caption{\Cref{ex: plate_with_cut-out}: Snapshots of the exact and numerical solutions.}
    \label{fig: plate_with_cut_out_sol}
\end{figure}

\begin{figure}[H]	
    \centering
    \begin{subfigure}[b]{0.27\textwidth}\includegraphics[height=4.0cm]{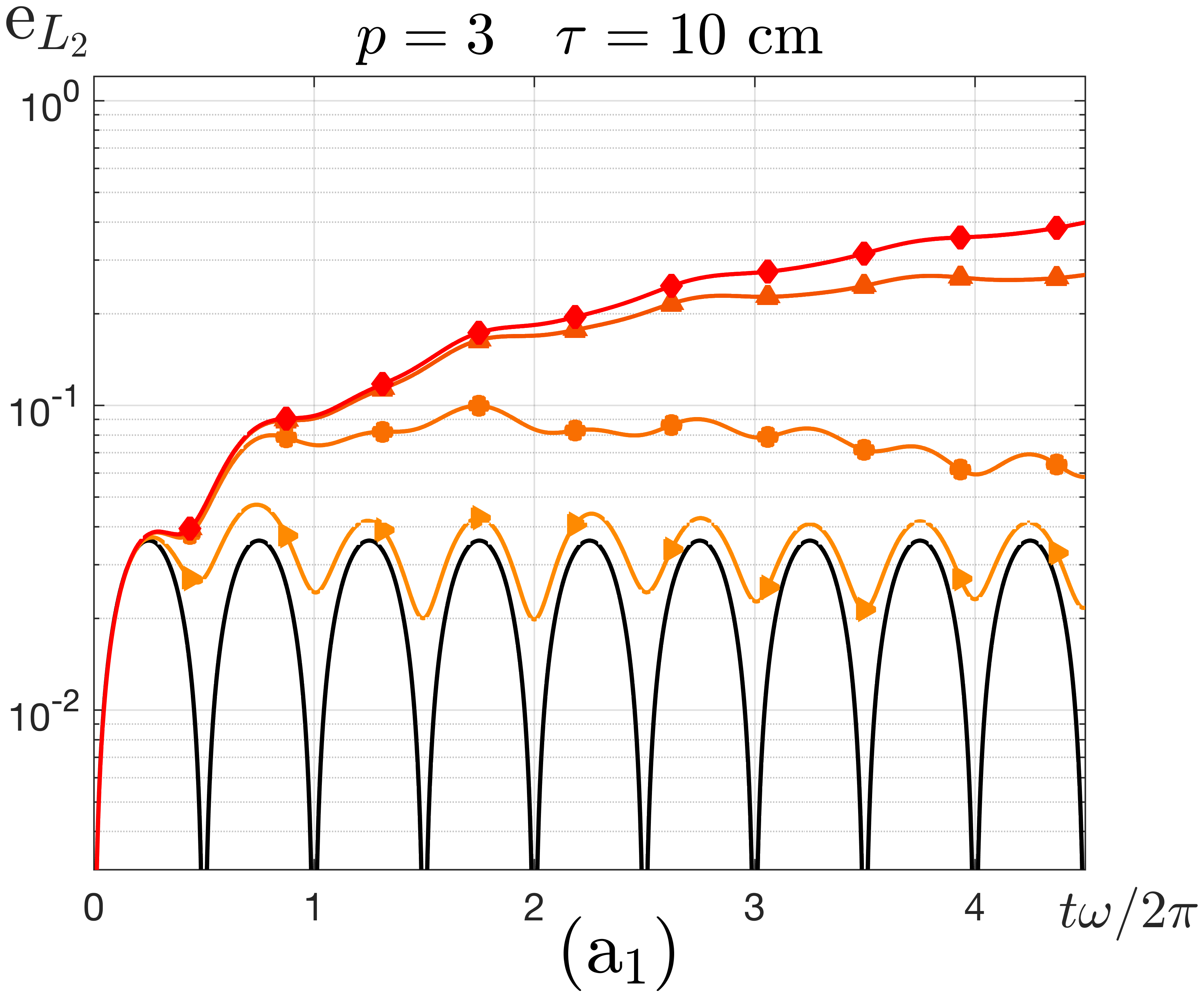}\end{subfigure} 
    \begin{subfigure}[b]{0.27\textwidth}\includegraphics[height=4.0cm]{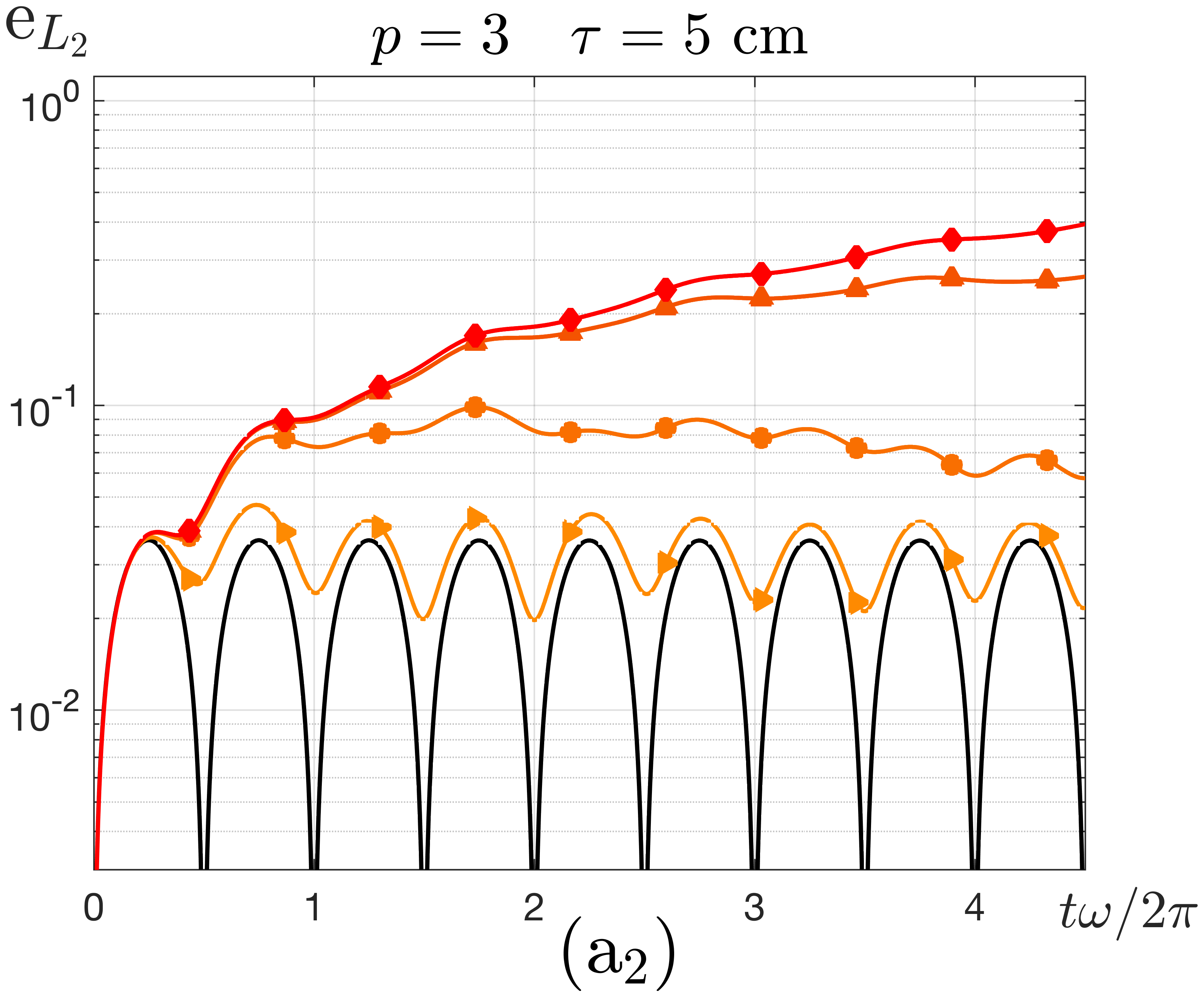}\end{subfigure} 
    \begin{subfigure}[b]{0.44\textwidth}\includegraphics[height=4.0cm]{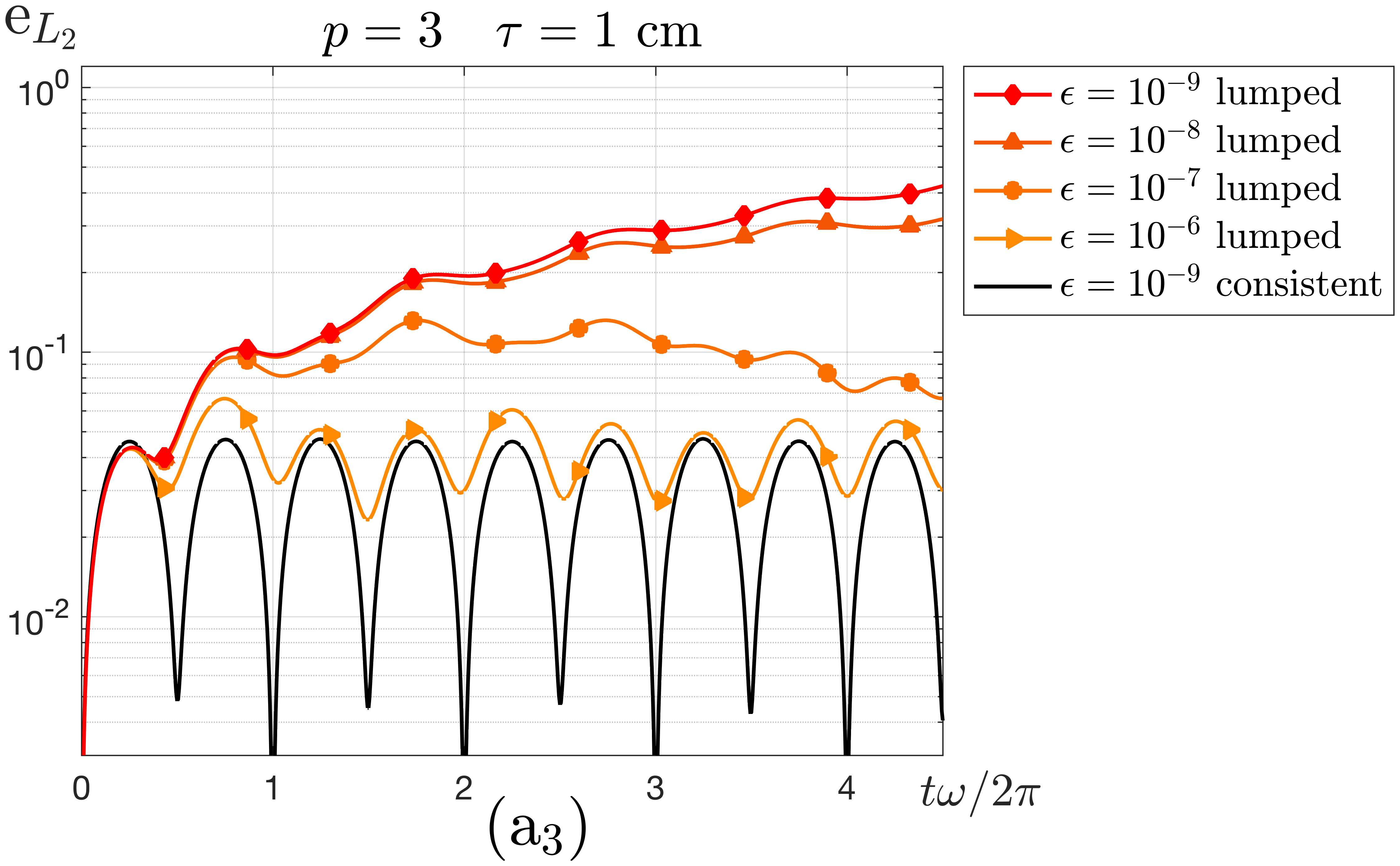}\end{subfigure}\\
    \begin{subfigure}[b]{0.27\textwidth}\includegraphics[height=4.0cm]{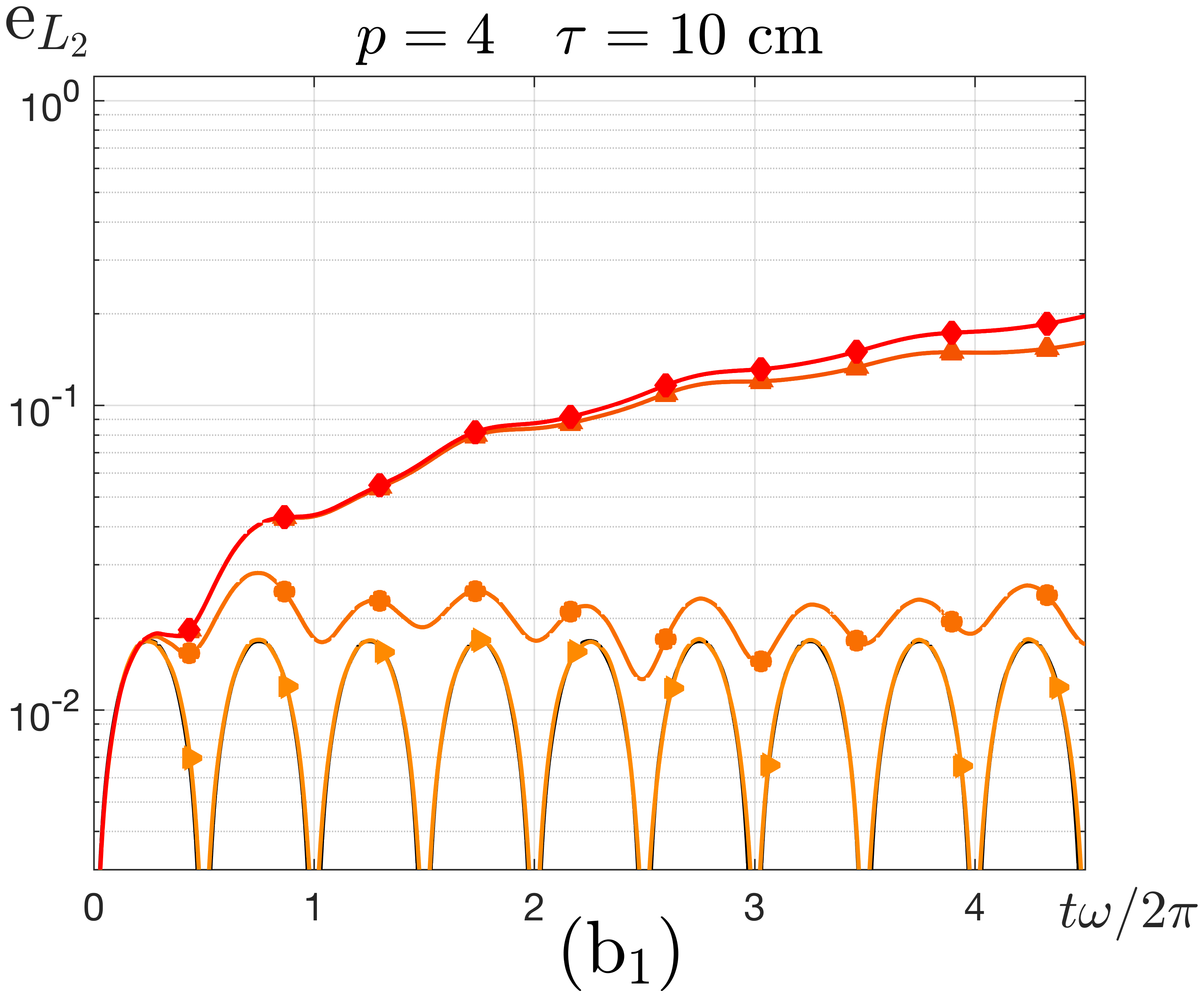}\end{subfigure}
    \begin{subfigure}[b]{0.27\textwidth}\includegraphics[height=4.0cm]{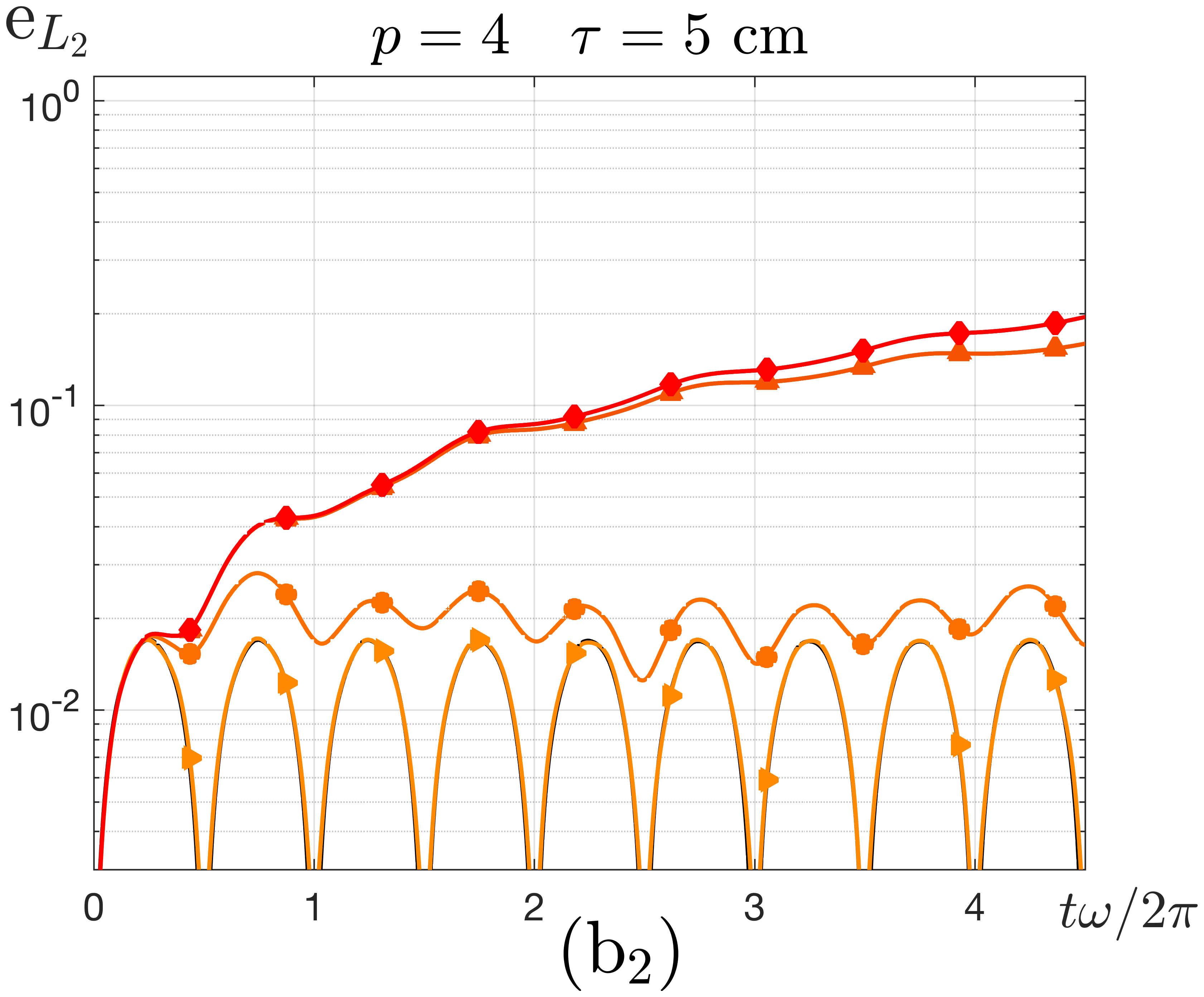}\end{subfigure}
    \begin{subfigure}[b]{0.44\textwidth}\includegraphics[height=4.0cm]{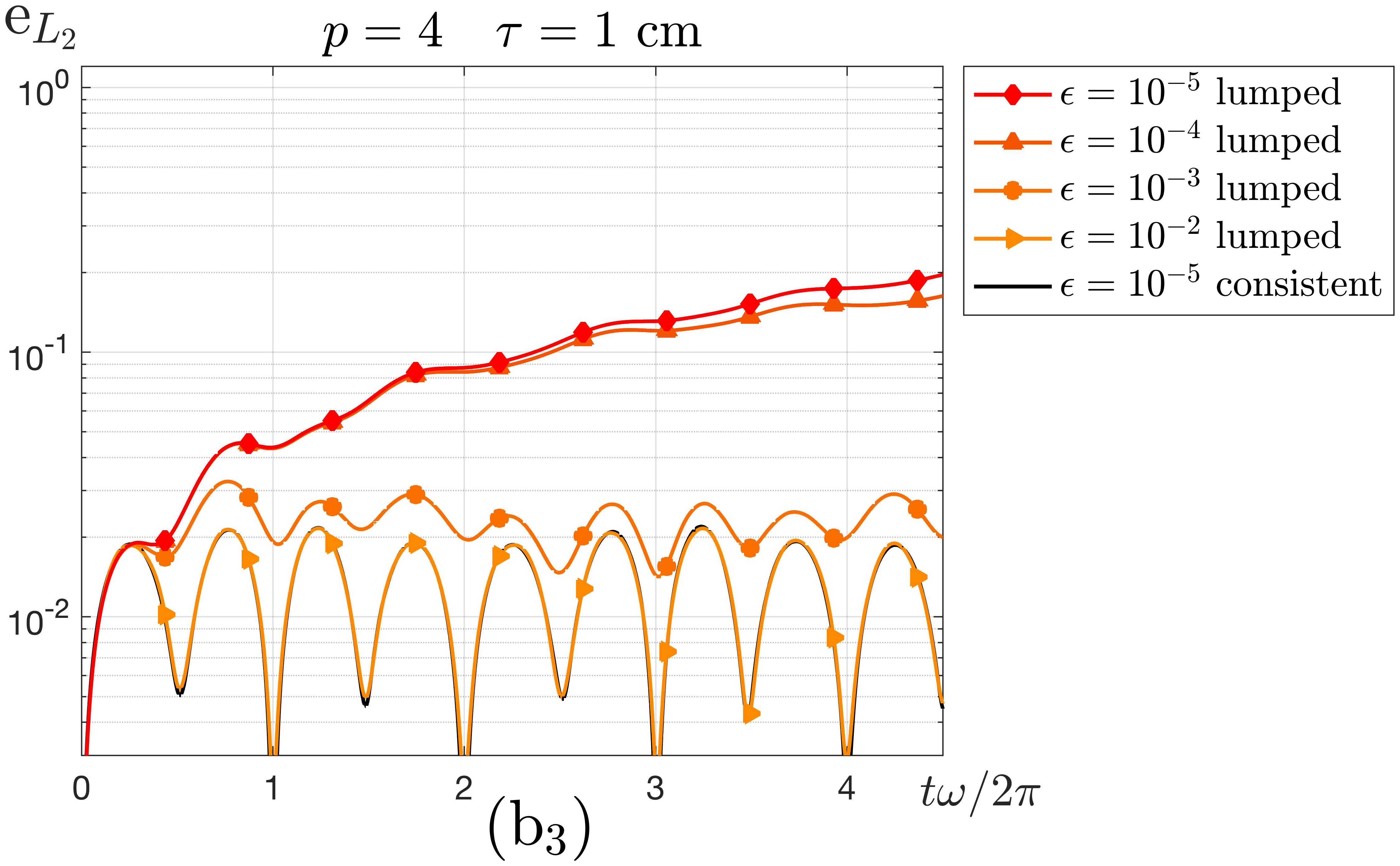}\end{subfigure}
    \caption{\Cref{ex: plate_with_cut-out}: $L^2$ error for the consistent and lumped mass solutions for different spline orders, thicknesses and trimming parameters.} 
    \label{fig: trimmed_plate_with_cut_out_err_cons_lumped}
\end{figure}
\begin{figure}[H]	
    \centering
    \begin{subfigure}[b]{0.27\textwidth}\includegraphics[height=4.0cm]{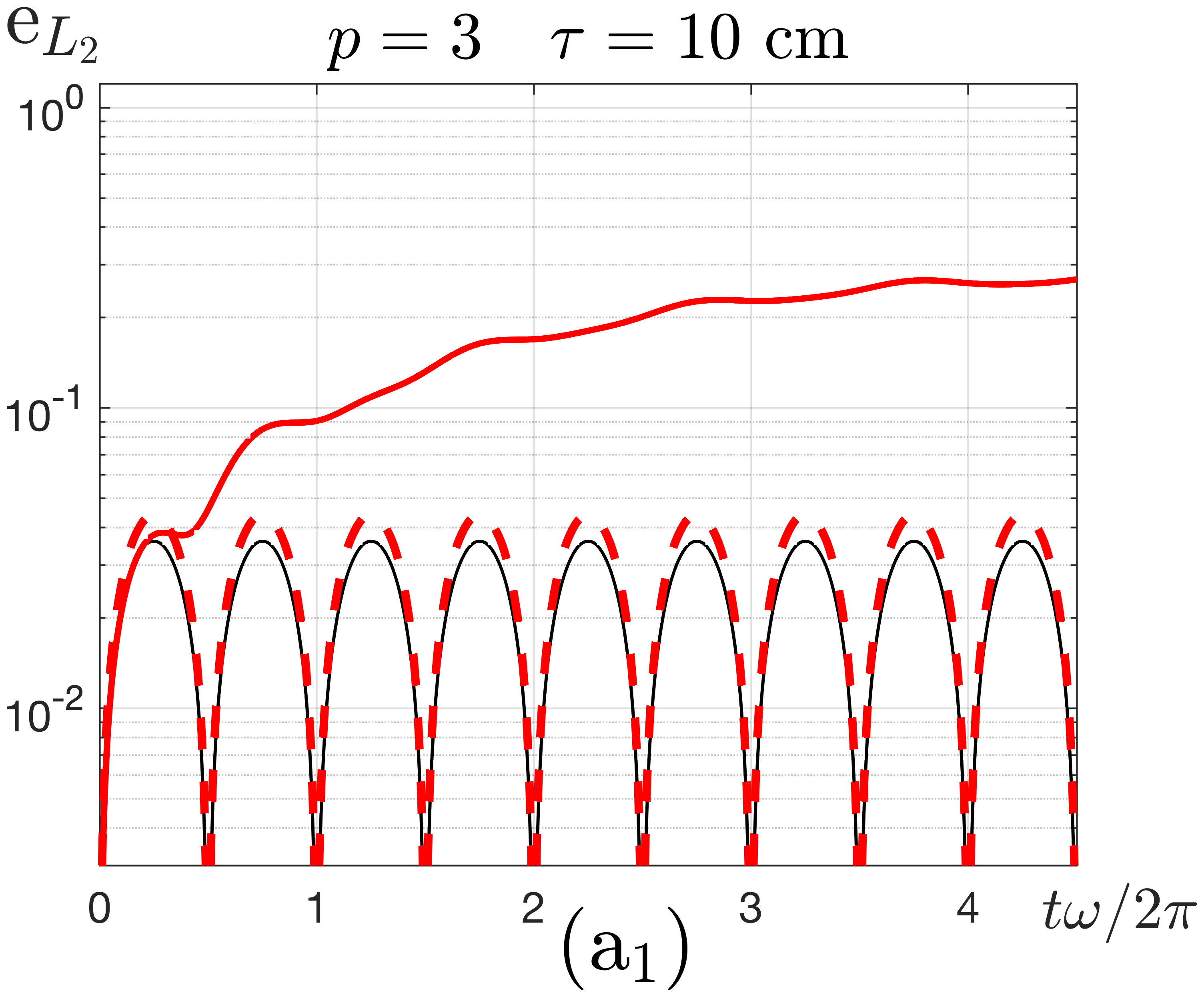}
    \end{subfigure}
    \begin{subfigure}[b]{0.27\textwidth}\includegraphics[height=4.0cm]{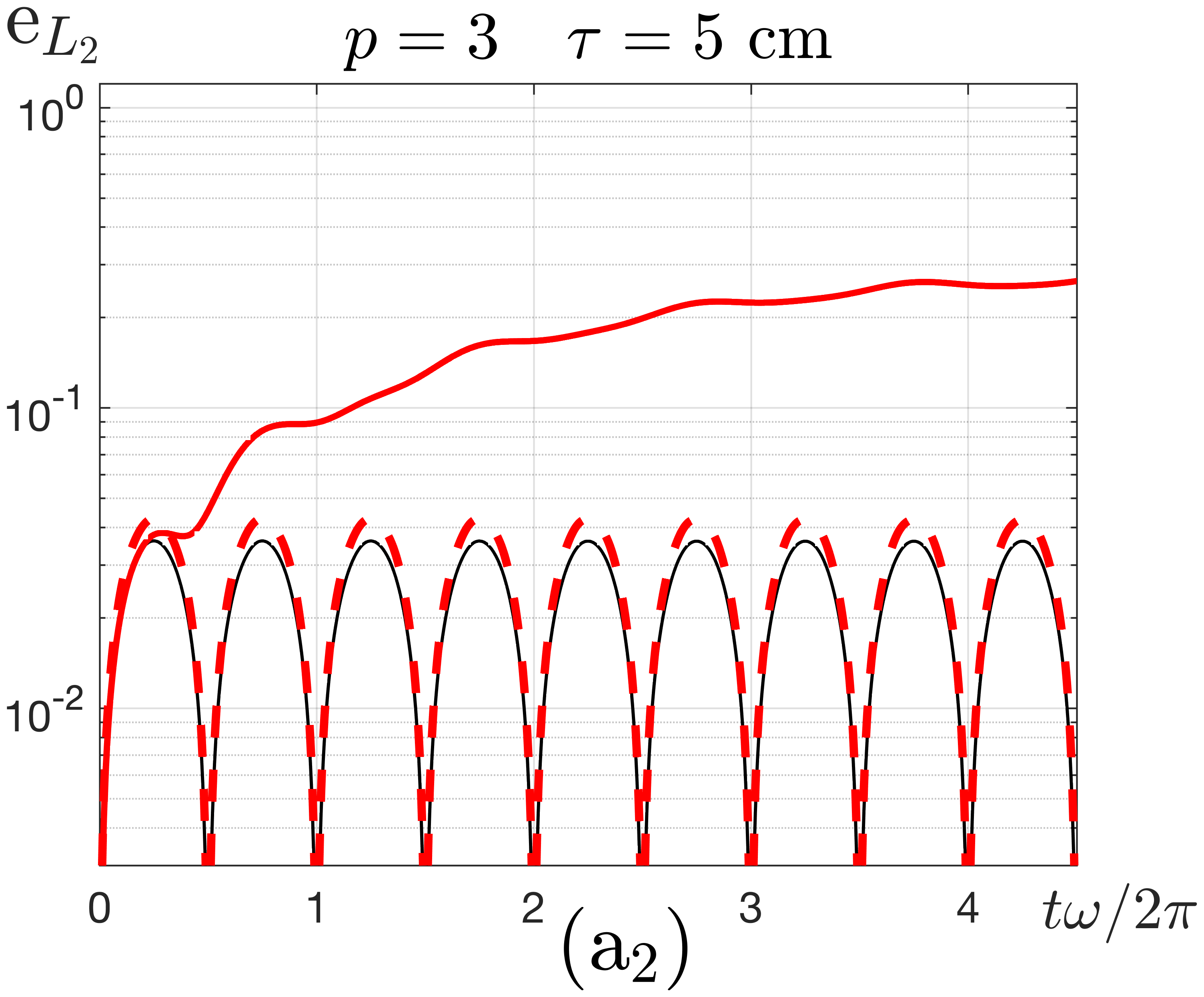}
    \end{subfigure}
    \begin{subfigure}[b]{0.44\textwidth}\includegraphics[height=4.0cm]{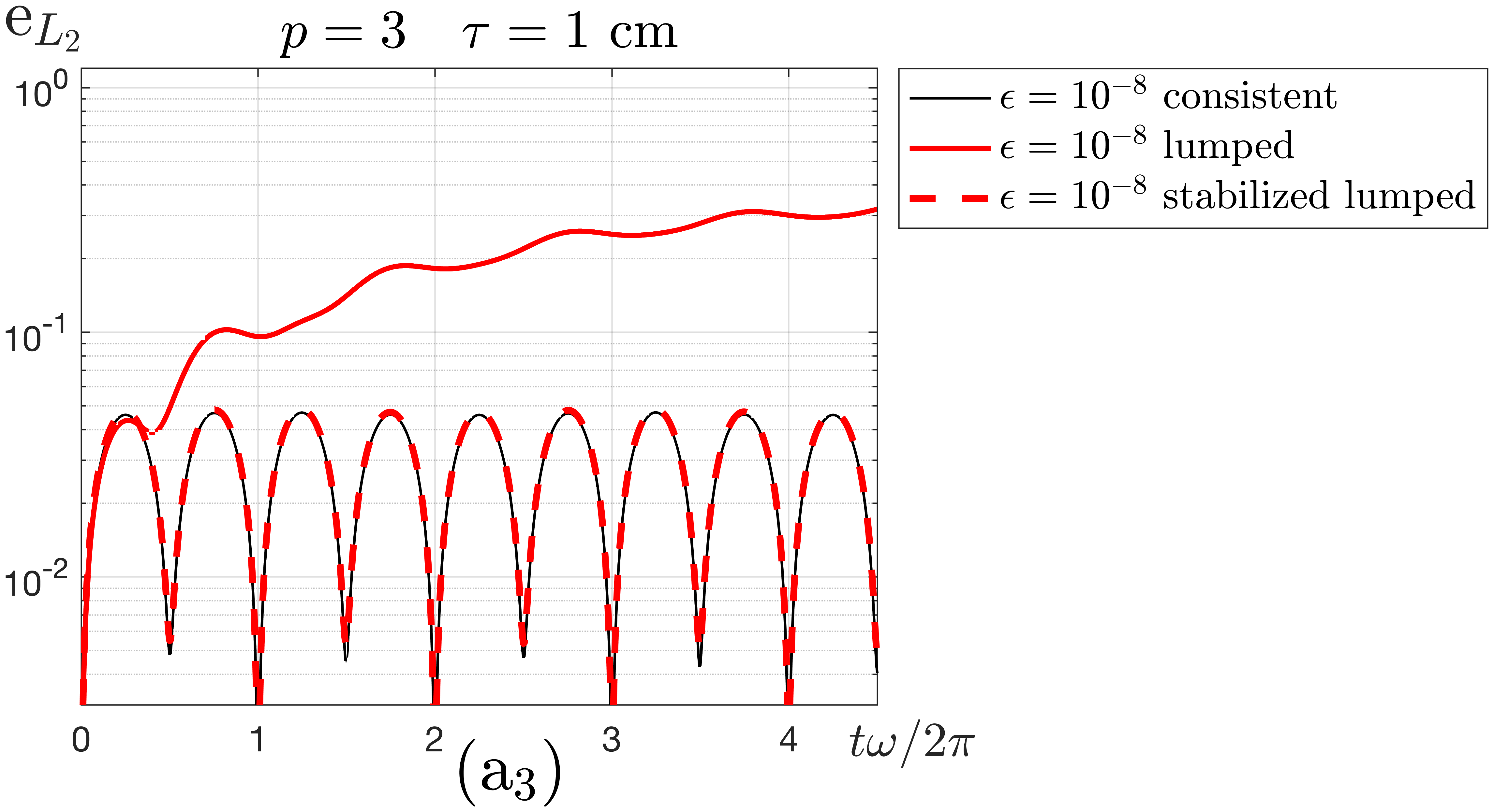}
    \end{subfigure}\\
    \begin{subfigure}[b]{0.27\textwidth}\includegraphics[height=4.0cm]{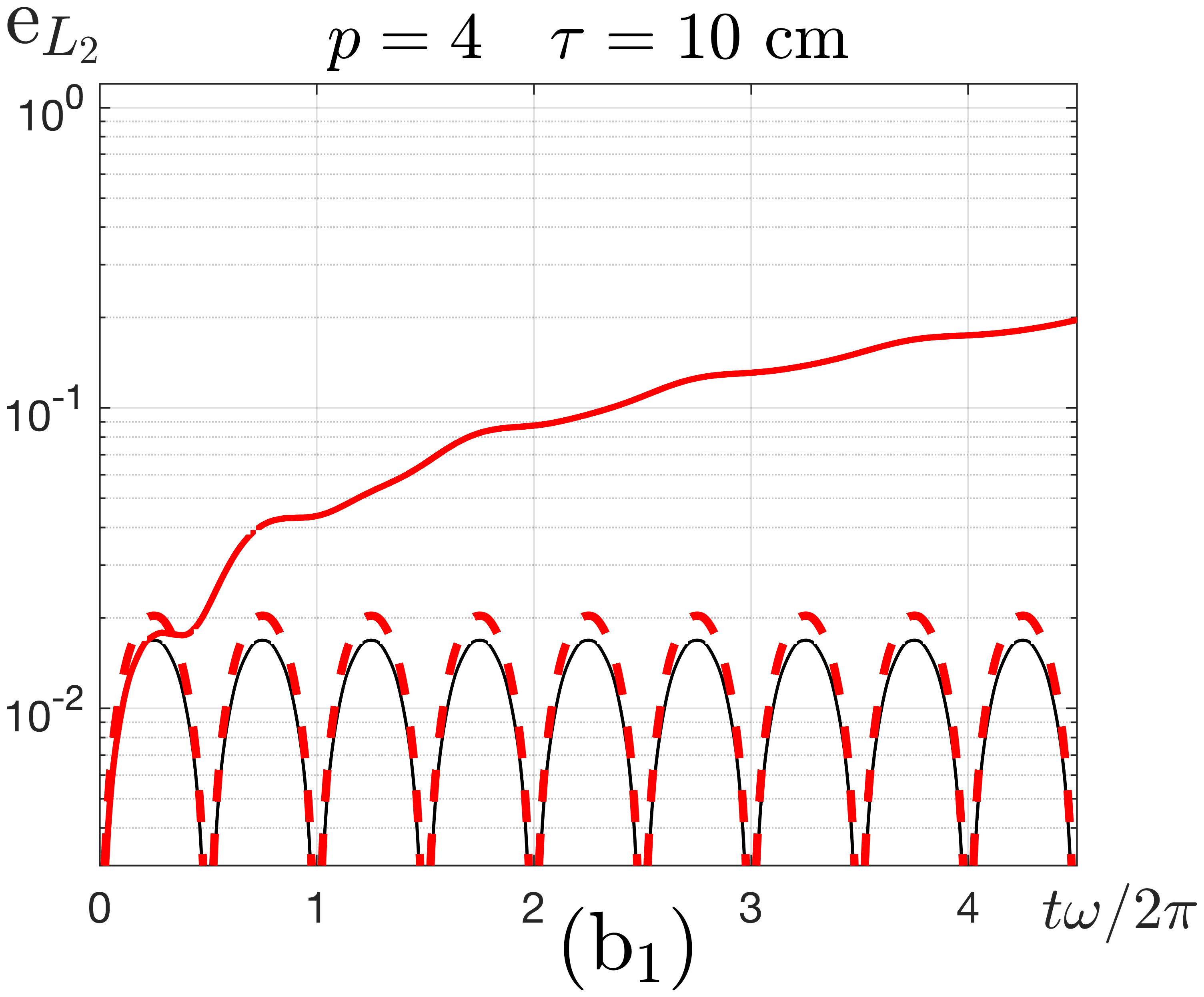}
    \end{subfigure}
    \begin{subfigure}[b]{0.27\textwidth}\includegraphics[height=4.0cm]{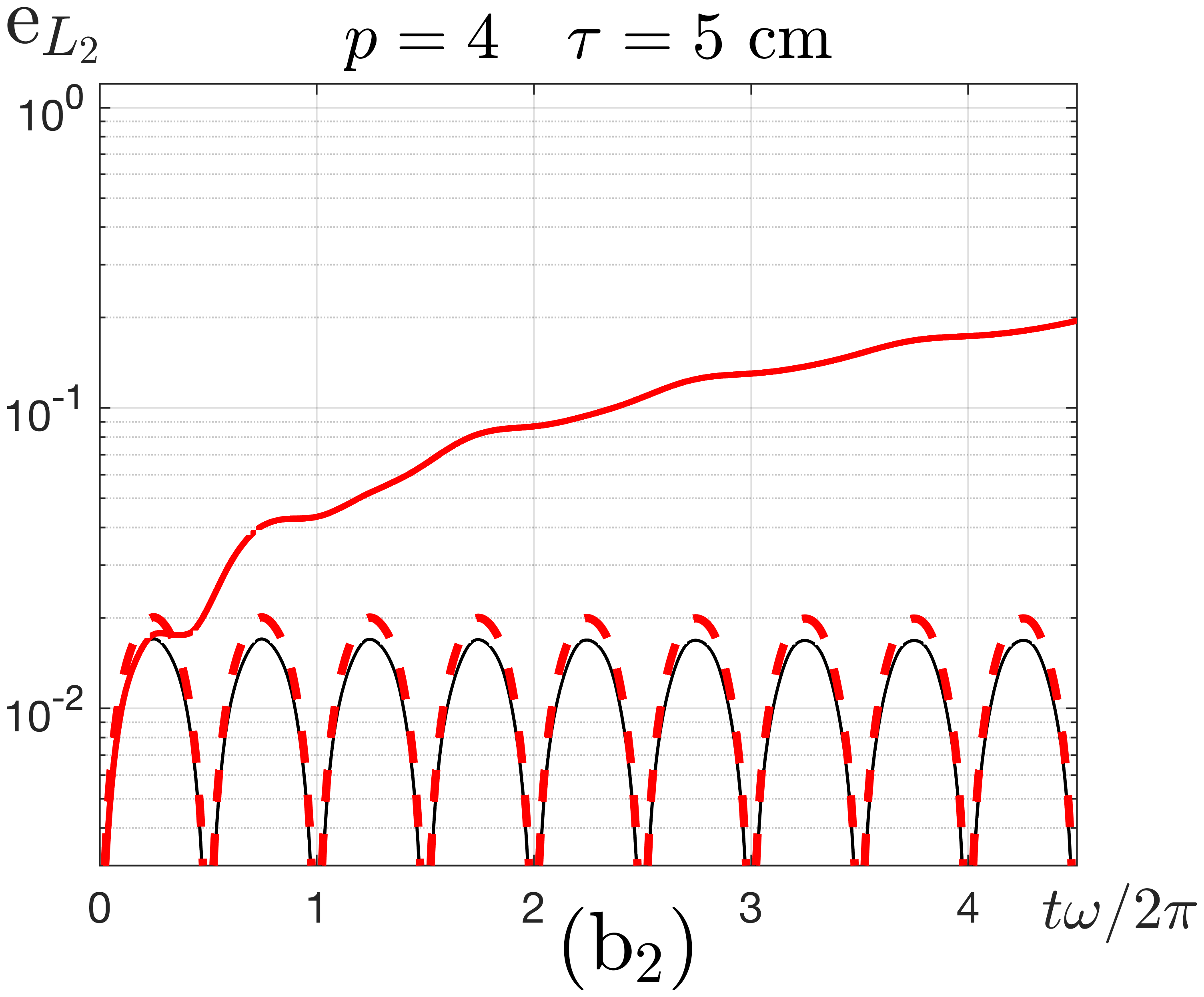}
    \end{subfigure}
    \begin{subfigure}[b]{0.44\textwidth}\includegraphics[height=4.0cm]{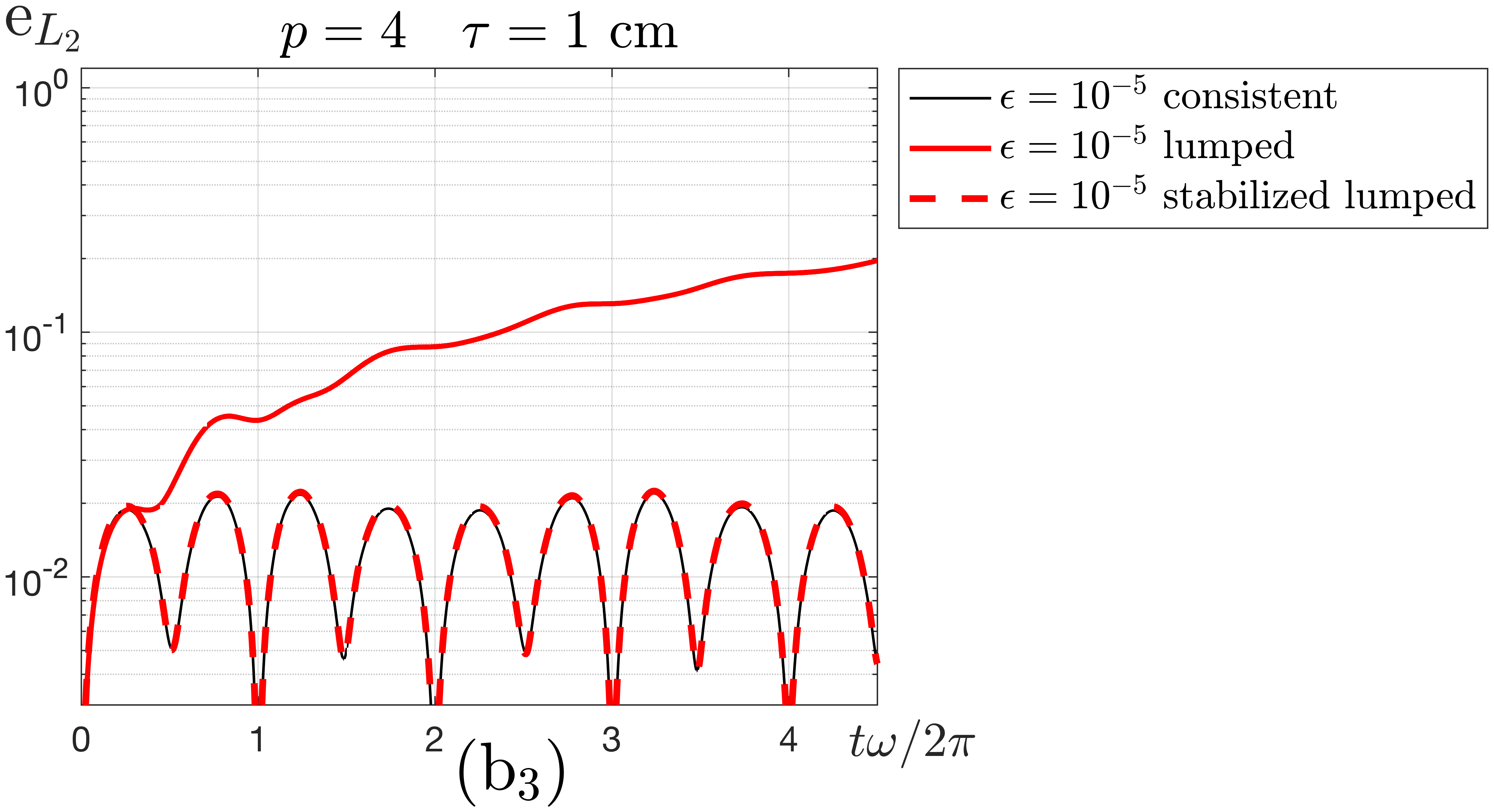}
    \end{subfigure}
    \caption{\Cref{ex: plate_with_cut-out}: $L^2$ error for the consistent and (stabilized) lumped mass solutions for different spline orders, thicknesses and trimming parameters.}
    \label{fig: trimmed_plate_with_cut_out_err_cons_stab_lumped}
\end{figure}
\end{example}

\begin{example}[Fuselage window]
\label{ex: fuselage_window}
Our fourth example is one step closer to industrial type examples and models an aircraft fuselage window. This is our first encounter with an actual shell structure, whose mid-surface is defined by the parameterization
\begin{equation}
    \bm{x} = \bm{\mathcal{F}}(\xi_1,\xi_2) = 
    \begin{pmatrix} 
    \xi_1-\frac{1}{2} \\ 
    R \sin{\left(\xi_2-\frac{1}{2}\right)} \\ 
    R \cos{\left(\xi_2-\frac{1}{2}\right)} 
    \end{pmatrix},
\end{equation}
where $R=1$ m is the radius of curvature. A window, consisting in a rectangular opening with rounded corners, is cut out from a fuselage panel of side length $L=1$ m (\figref{fig: fuselage_window}{b}). The width and height of the opening are $2a=0.3$ m and $2b=0.4$ m, respectively, and its corners form circular arcs of radius $r=0.08$ m (\figref{fig: fuselage_window}{a}). As in the previous examples, we consider a manufactured solution
\begin{equation}
\label{eq: manufactured_sol_fuselage}
    \bm{u}(\bm{\xi},t) = w(\bm{\xi})\phi(t) \bm{e}_3, \quad \text{and} \quad \bm{\theta}(\bm{\xi},t) = - \sum_{i=1}^2 (\bm{a}_3 \cdot \bm{u}_{,i}) \bm{a}^i,
\end{equation}
where the spatial part $w \colon S \to \mathbb{R}$ and temporal part $\phi \colon [0,\; +\infty) \to \mathbb{R}$ are defined as
\begin{equation*}
    w(\bm{\xi}) = a_0 w_1 w_2 w_n \circ \bm{l}(\bm{\xi}), \quad \text{and} \quad \phi(t) = \sin(\omega t),
\end{equation*}
with the functions
\begin{align*}
    w_1(\bm{y}) &= e^{1-\left(\frac{y_1}{a}\right)^2}, \\
    w_2(\bm{y}) &= \left(1+e^{\beta\left(\frac{y_2}{a}-1\right)}\right)^{-1} \left(1+e^{\beta\left(-\frac{y_2}{a}-1\right)}\right)^{-1}, \\
    w_n(\bm{y}) &= \cos\left(2\pi n_0 \sqrt[n]{\frac{1}{2}\left[\left(\frac{y_1}{a}+\frac{y_2}{b}\right)^n +\left(\frac{y_1}{a}-\frac{y_2}{b}\right)^n \right]}\right), \\
\end{align*}
and
\begin{equation*}
    \bm{l}(\bm{\xi}) =
    \begin{pmatrix}
        \xi_1 - \frac{1}{2} \\
        \xi_2 - \frac{1}{2}
    \end{pmatrix}
\end{equation*}
with parameter values $a_0 = 10$ cm, $\beta = 10$, $n_0 = 3$, $n = 6$, Poisson ratio $\nu=0.25$ and $\omega = \frac{1}{10L}\sqrt{\frac{E}{\rho}}$ with unit material parameters.

Dirichlet (or type 1) boundary conditions are prescribed along the outer sides of the panel while Neumann (or type 2) boundary conditions are prescribed all along the window's boundaries. The boundary data (as well as the source term and initial conditions) are computed from the expressions in \eqref{eq: manufactured_sol_fuselage}. The exact solution is shown in \figref{fig: fuselage_window_err_cons_lumped}{e$_1$-e$_3$} for times t$_1$-t$_3$. The problem is discretized with maximally smooth B-splines of order $p$ over a grid of mesh size $h$. The vertical sides of the window are at a distance $\delta$ from the nearest grid line and as usual $\epsilon=\delta/h$ denotes the relative trimming parameter. The numerical solutions for the consistent, lumped and stabilized lumped solutions are computed with the same schemes as before (i.e., the Central Difference method is used for both solutions with a lumped mass and an unconditionally stable Newmark method is used for the solution with the consistent mass). Snapshots of the solutions are shown in \Cref{fig: fuselage_window_sol} for the consistent (c$_1$-c$_3$), lumped (l$_1$-l$_3$) and stabilized lumped (s$_1$-s$_3$) solutions. Once again, spurious oscillations are visible along the trimmed boundary for the lumped mass solution, while the other two solutions are visually indistinguishable from the exact one. Those oscillations are not confined to the displacement field and may also affect the stress field, as shown in \figref{fig: fuselage_window_ben}{c} for the first component of the bending stress ($M^{11}$) at the final time t$_3$. Even without an exact or reference solution, the solution for the lumped mass would raise suspicions as it contradicts physical intuition. However, the error is not in the implementation but in the lumping itself. The error for the consistent and lumped mass solutions is reported in \Cref{fig: fuselage_window_err_cons_lumped} and follows a trend in all aspects similar to previous examples. While the accuracy of the solution for the consistent mass is not affected by small trimmed elements, the solution for the lumped mass rapidly deteriorates as $\epsilon$ becomes smaller or $p$ becomes larger (note the different range of values for $\epsilon$ in \figref{fig: fuselage_window_err_cons_lumped}{a$_1$-a$_3$} and \figref{fig: fuselage_window_err_cons_lumped}{b$_1$-b$_3$}). Increasing the slenderness also negatively affects the solution, although to a milder extent. As expected, stabilizing the solution completely resolves most of these issues (\Cref{fig: fuselage_window_err_cons_stab_lumped}). However, the quality of the solution still depends on the slenderness, as it does for boundary-fitted discretizations. Hence, the stabilization only resolves trimming-related issues and we may envision combining it with different techniques for solving common shell-related issues (e.g., locking).

\begin{figure}[H]	
    \centering
    \includegraphics[width=0.8\textwidth]{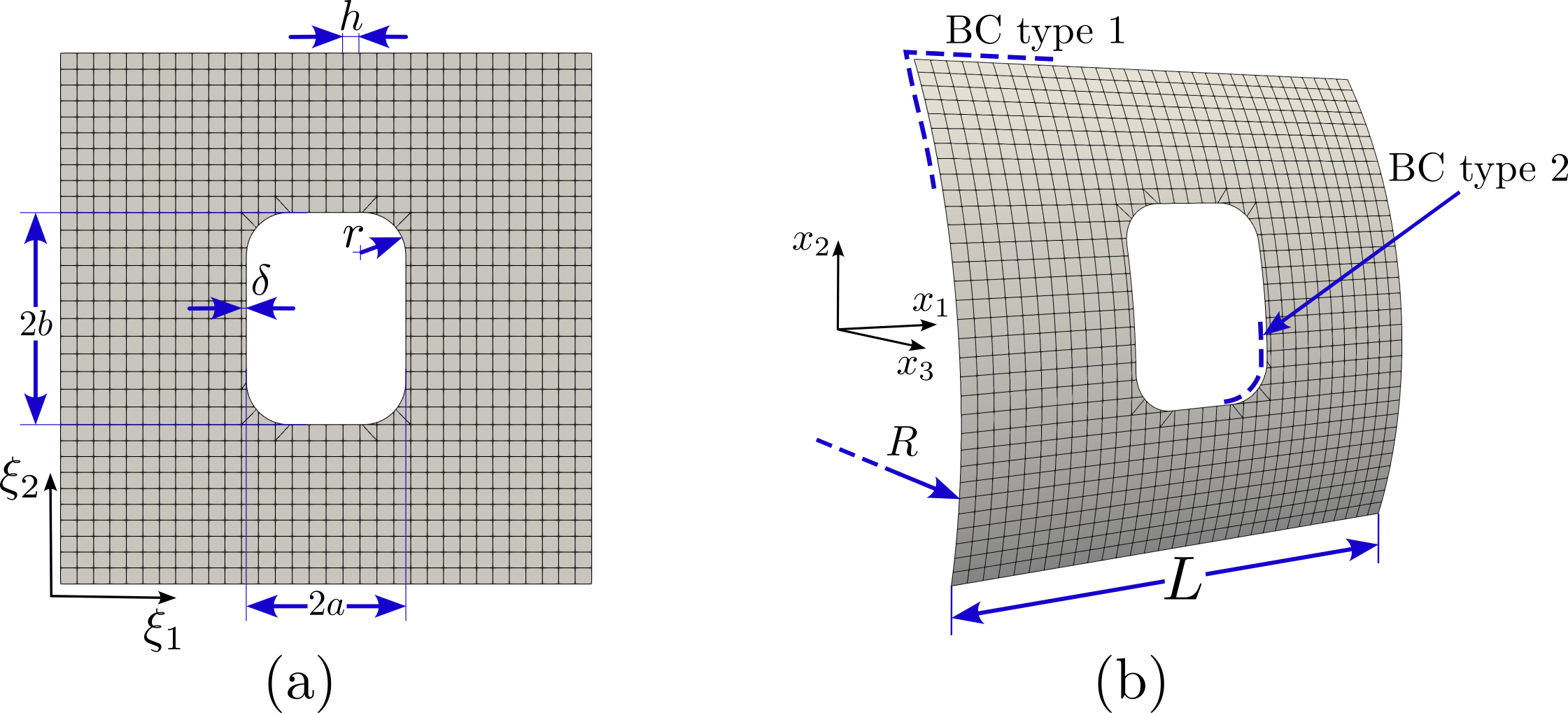}
    \caption{\Cref{ex: fuselage_window}: Aircraft fuselage window.}
    \label{fig: fuselage_window}
\end{figure}

% \begin{itemize}
%     \item BC type 2: Neumann boundary conditions
%     \item BC type 1: Strong Dirichlet boundary conditions
%     \item $\epsilon=\delta/h$
%     \item $\omega=0.5\sqrt(\frac{E}{\rho L^2})$
%     \item $L=1$ m
%     \item $a=0.15$ m
%     \item $b=0.20$ m 
%     \item $r=0.08$ m
%     \item manufactured displacement field $\bm{u} = \varphi \bm{a}_3$
%     \item manufactured rotation field $\bm{\theta} = - (\bm{a}_3\cdot\bm{u}_{,\alpha}) \bm{a}^\alpha$ \yannis{Please clarify}. The function for the rotation is obtained in a way that makes $\gamma_{\alpha}=0$ (I noticed it helps activating the spurious modes compared to $\theta_\alpha=0$) So, by equation (2.12d) : $\bm{\theta} + \bm{a}_3\cdot\bm{u}_{,\alpha}=0$. 
% \end{itemize}

% The function $\varphi$ is constructed as 

% $\varphi = a_0 \varphi_1 \varphi_2 \varphi_n$

% $\varphi_1 = e^{1-\left(\frac{\tilde{\xi}_1}{a}\right)^2}$

% $\varphi_2 = \left(1+e^{\beta\left(\frac{\tilde{\xi}_2}{a}-1\right)}\right)^{-1} \left(1+e^{\beta\left(-\frac{\tilde{\xi}_2}{a}-1\right)}\right)^{-1}$

% $\varphi_n =  \cos\left(2\pi n_0 \sqrt[n]{\frac{1}{2}\left[\left(\frac{\tilde{\xi}_1}{a}+\frac{\tilde{\xi}_2}{b}\right)^n +\left(\frac{\tilde{\xi}_1}{a}-\frac{\tilde{\xi}_2}{b}\right)^n \right]}\right)$

% $\tilde{\xi}_\alpha= \xi_\alpha- \frac{1}{2}$

% The values of the parameters for this test are the following:

% \begin{tabular}{c|c}
% \hline
% \hline
% Parameter & Value \\  
% \hline
% $a_0$ & 10 cm \\  
% $\beta$ & 10 \\  
% $n_0$ & 3 \\
% $n$ & 6 
% \end{tabular}

\begin{figure}[H]	
    \centering
    \includegraphics[width=0.7\textwidth]{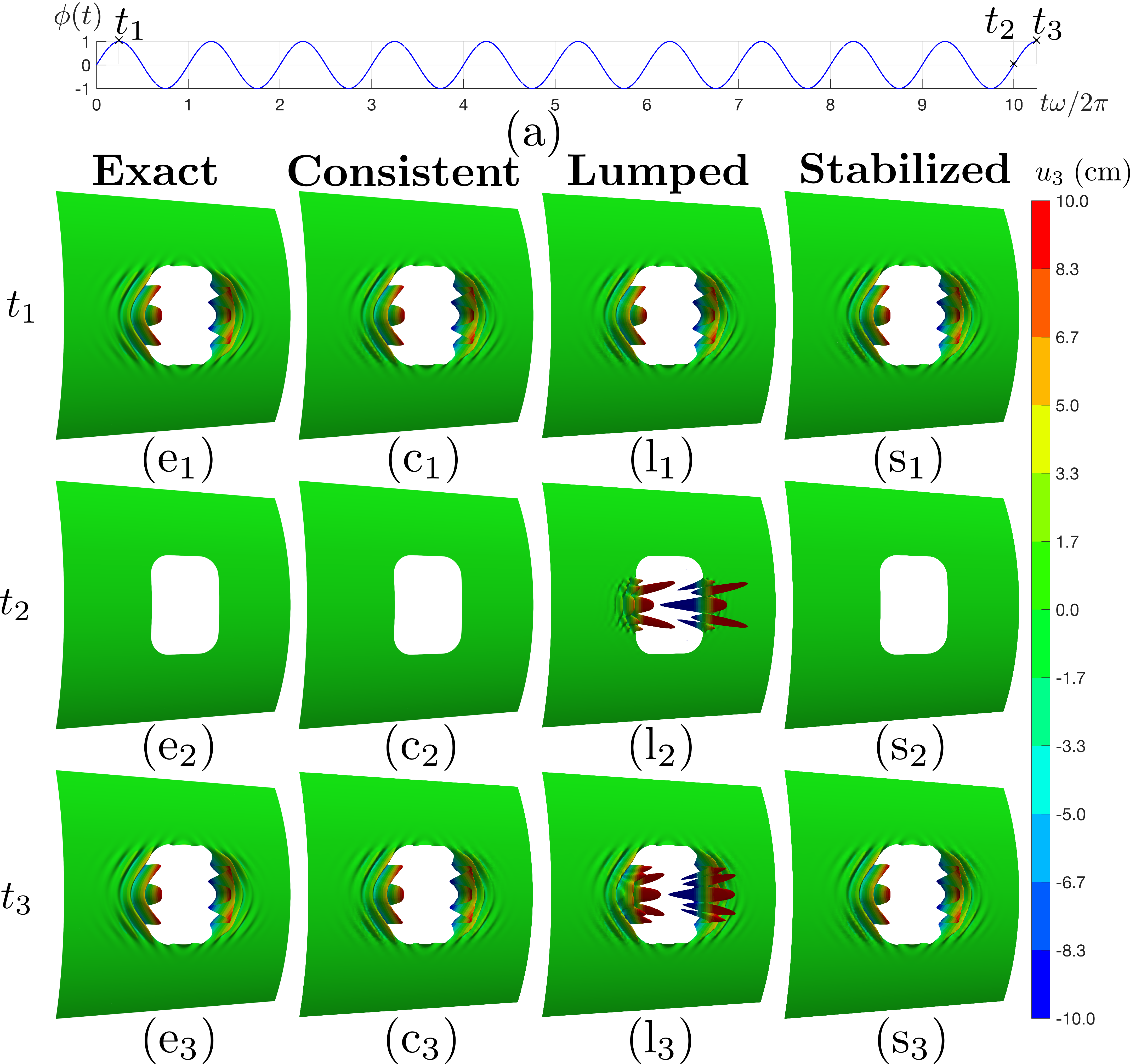}
    \caption{\Cref{ex: fuselage_window}: Snapshots of the exact and numerical solutions.}
    \label{fig: fuselage_window_sol}
\end{figure}

\begin{figure}[H]	
    \centering
    \includegraphics[width=0.7\textwidth]{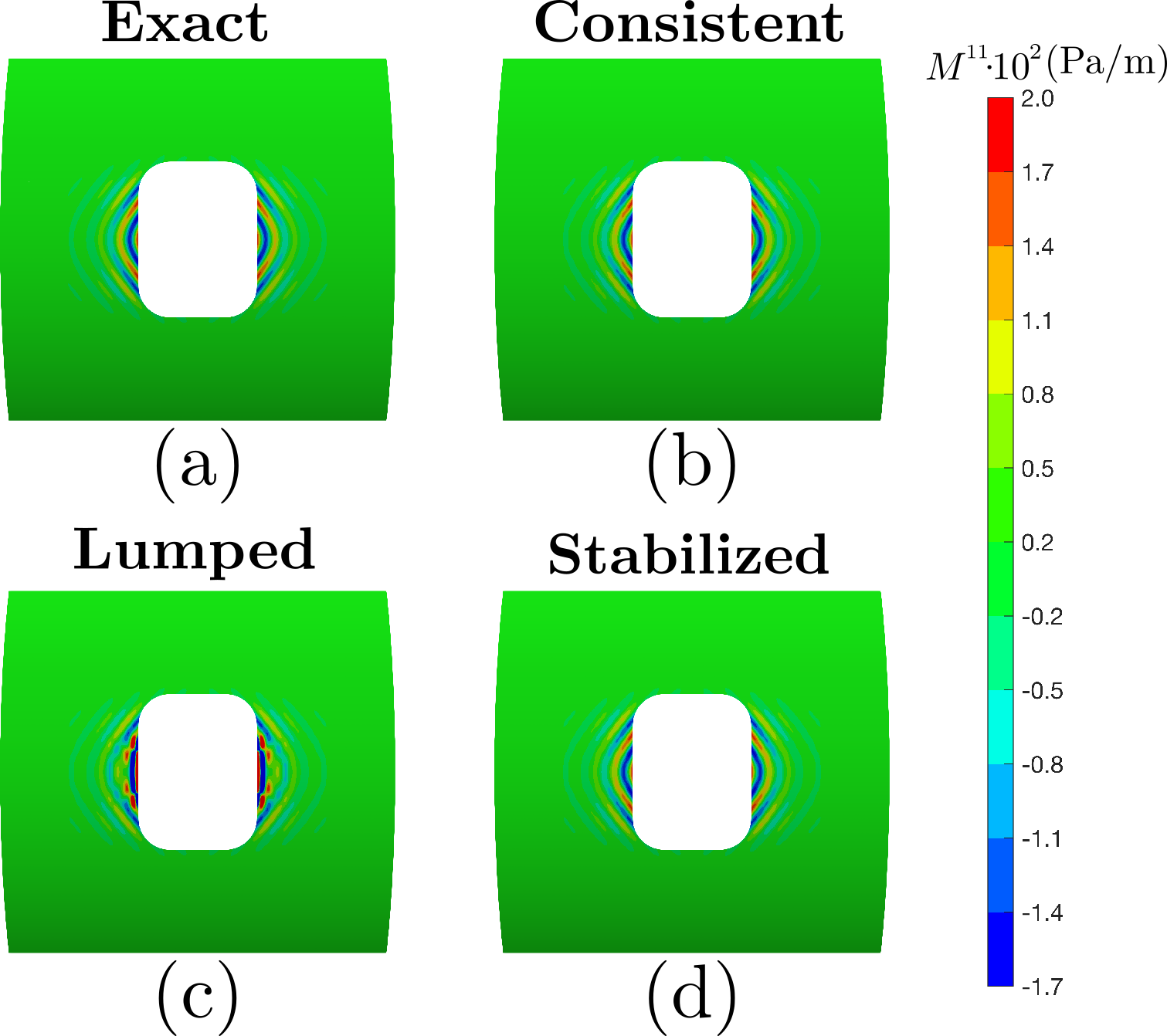}
    \caption{\Cref{ex: fuselage_window}: Contour plots of the first component of the bending stress for the exact and numerical solutions.}
    \label{fig: fuselage_window_ben}
\end{figure}

\begin{figure}[H]	
    \centering
    \begin{subfigure}[b]{0.27\textwidth}\includegraphics[height=4.0cm]{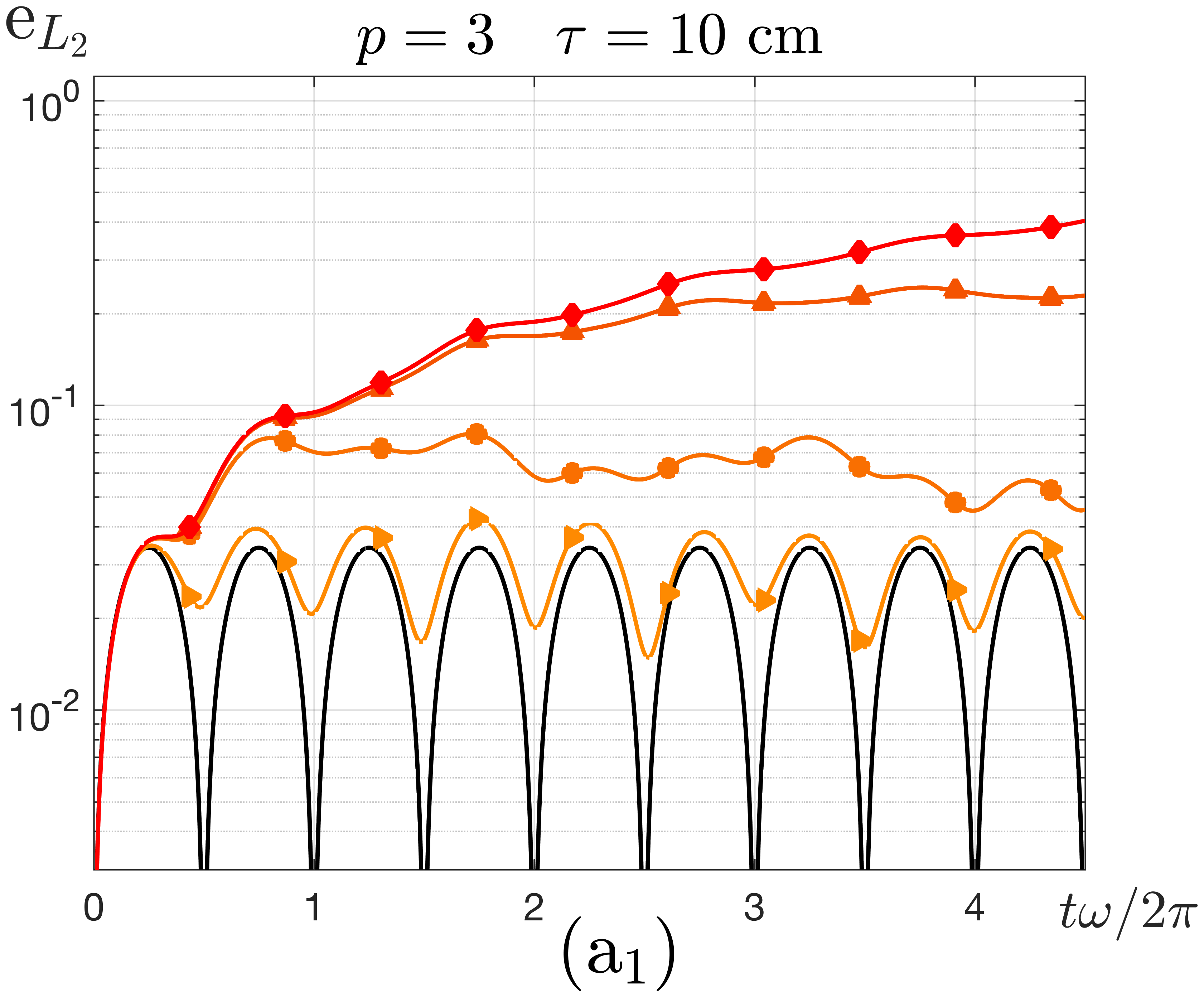}\end{subfigure} 
    \begin{subfigure}[b]{0.27\textwidth}\includegraphics[height=4.0cm]{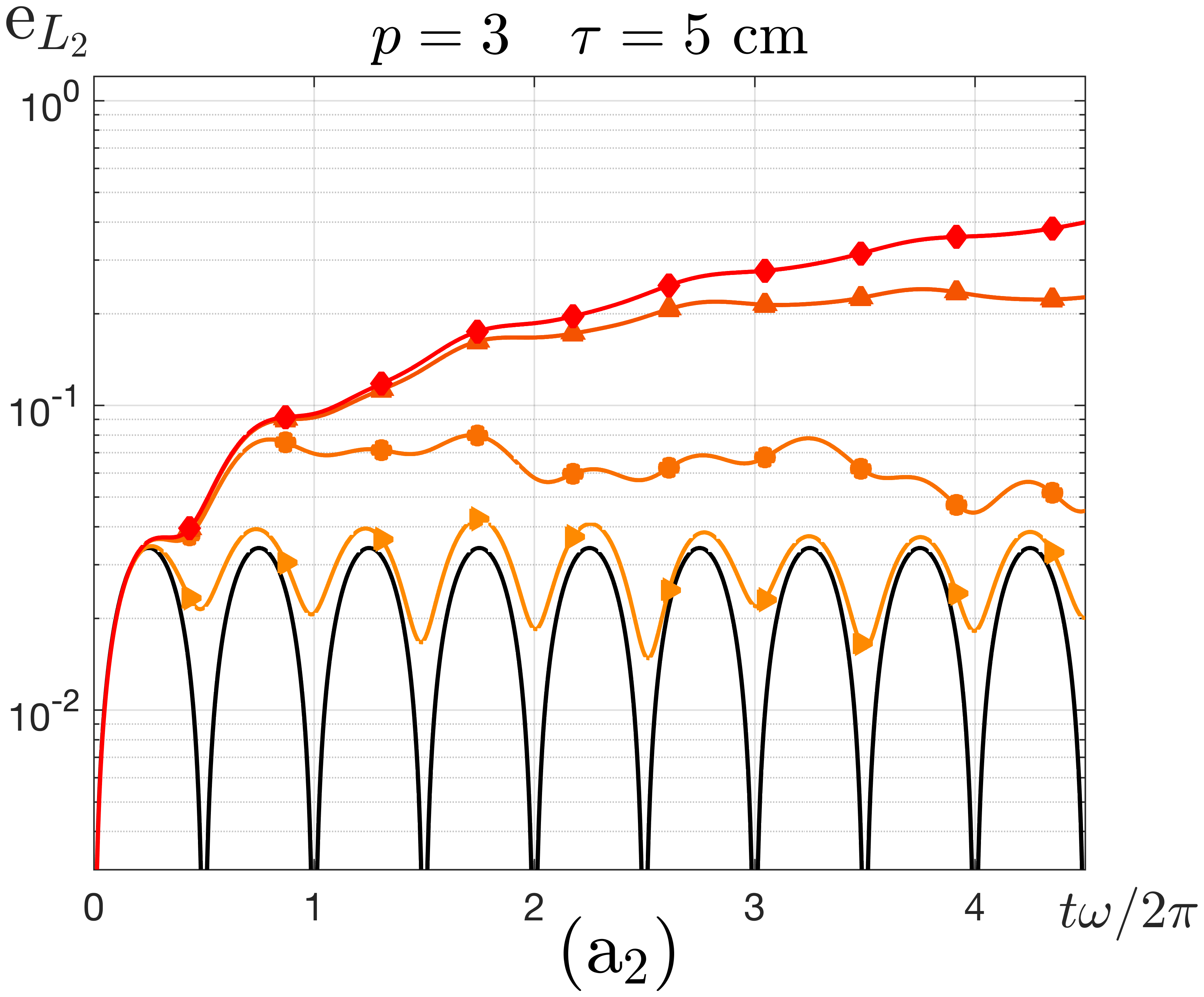}\end{subfigure} 
    \begin{subfigure}[b]{0.44\textwidth}\includegraphics[height=4.0cm]{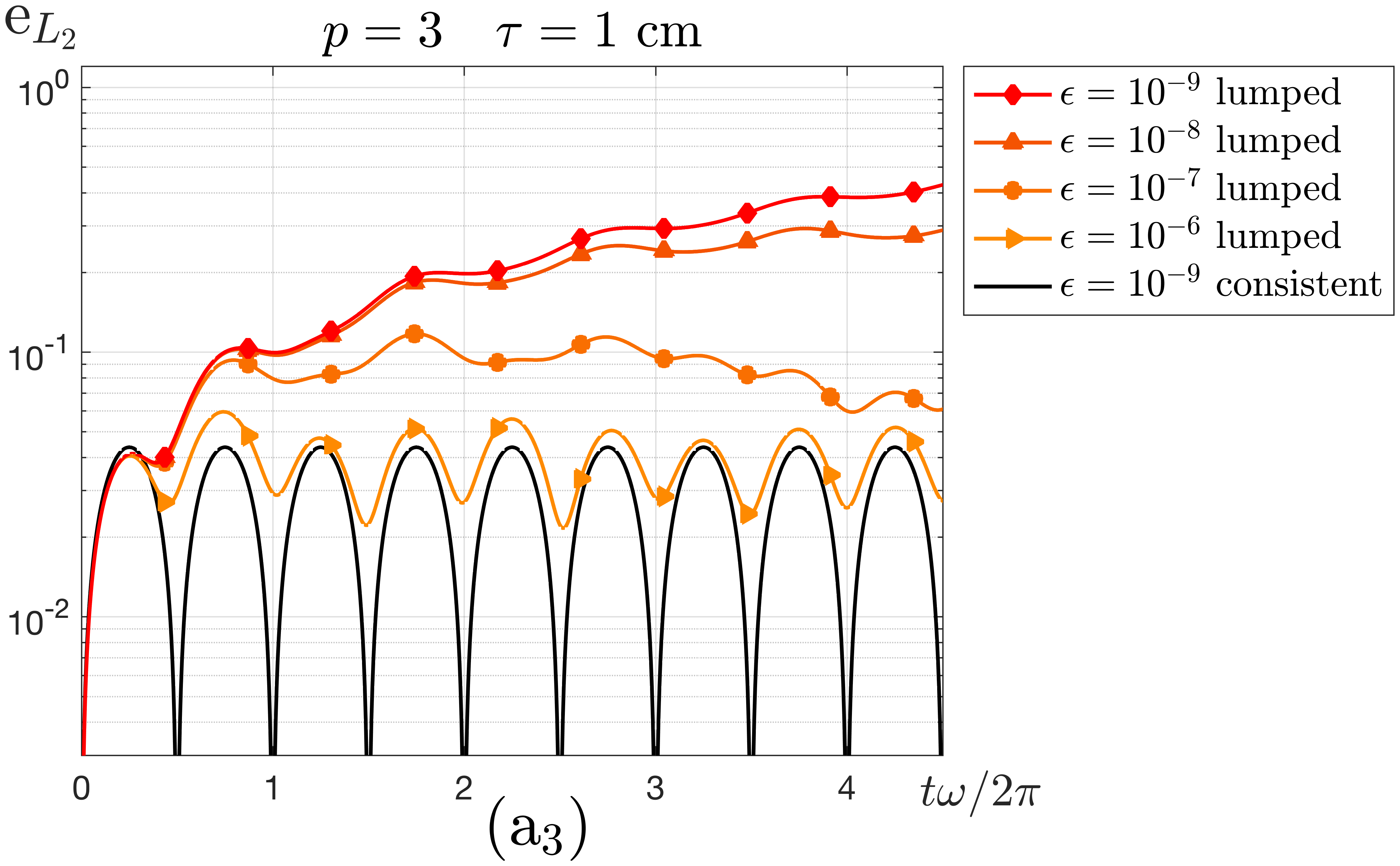}\end{subfigure}\\
    \begin{subfigure}[b]{0.27\textwidth}\includegraphics[height=4.0cm]{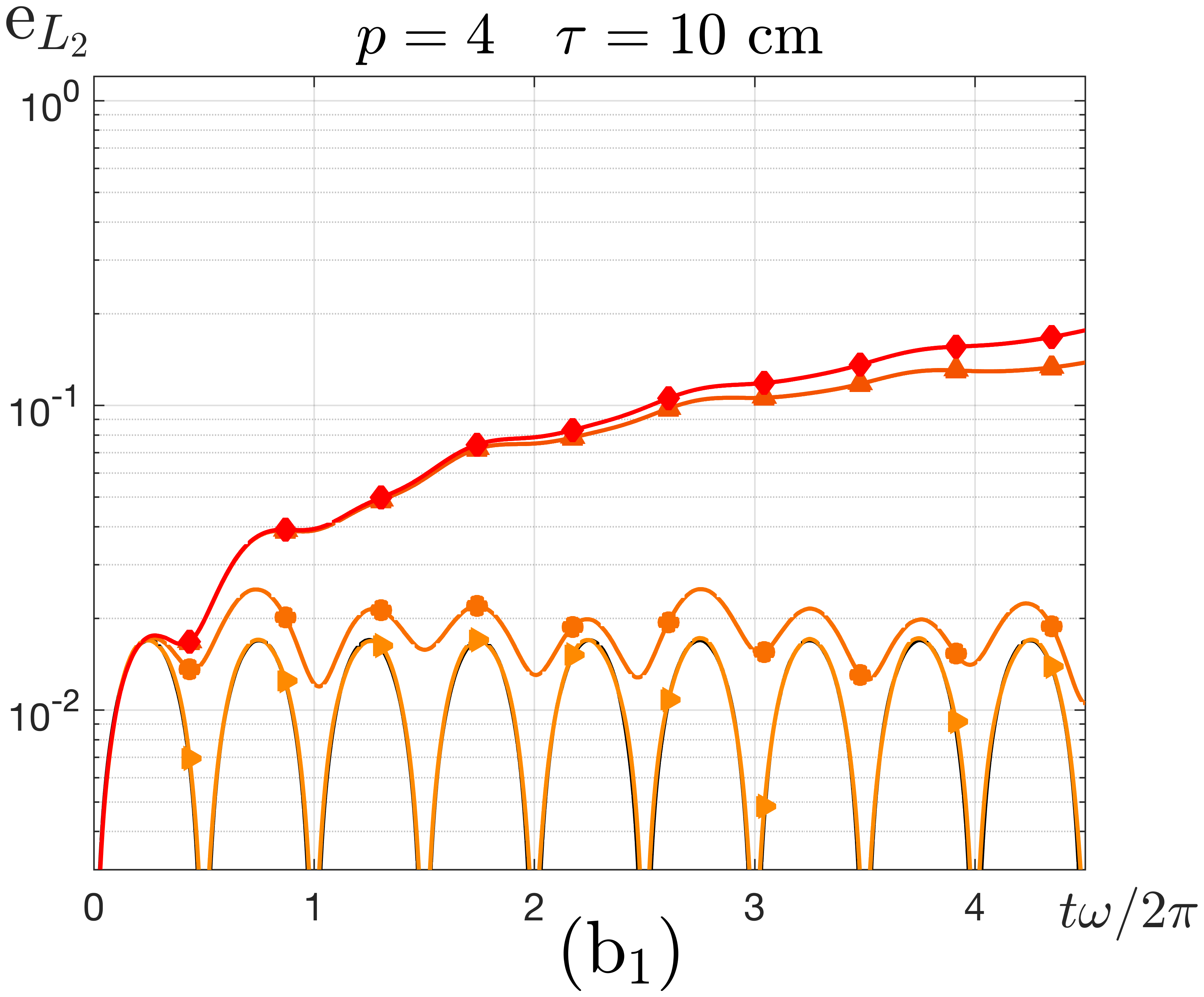}\end{subfigure}
    \begin{subfigure}[b]{0.27\textwidth}\includegraphics[height=4.0cm]{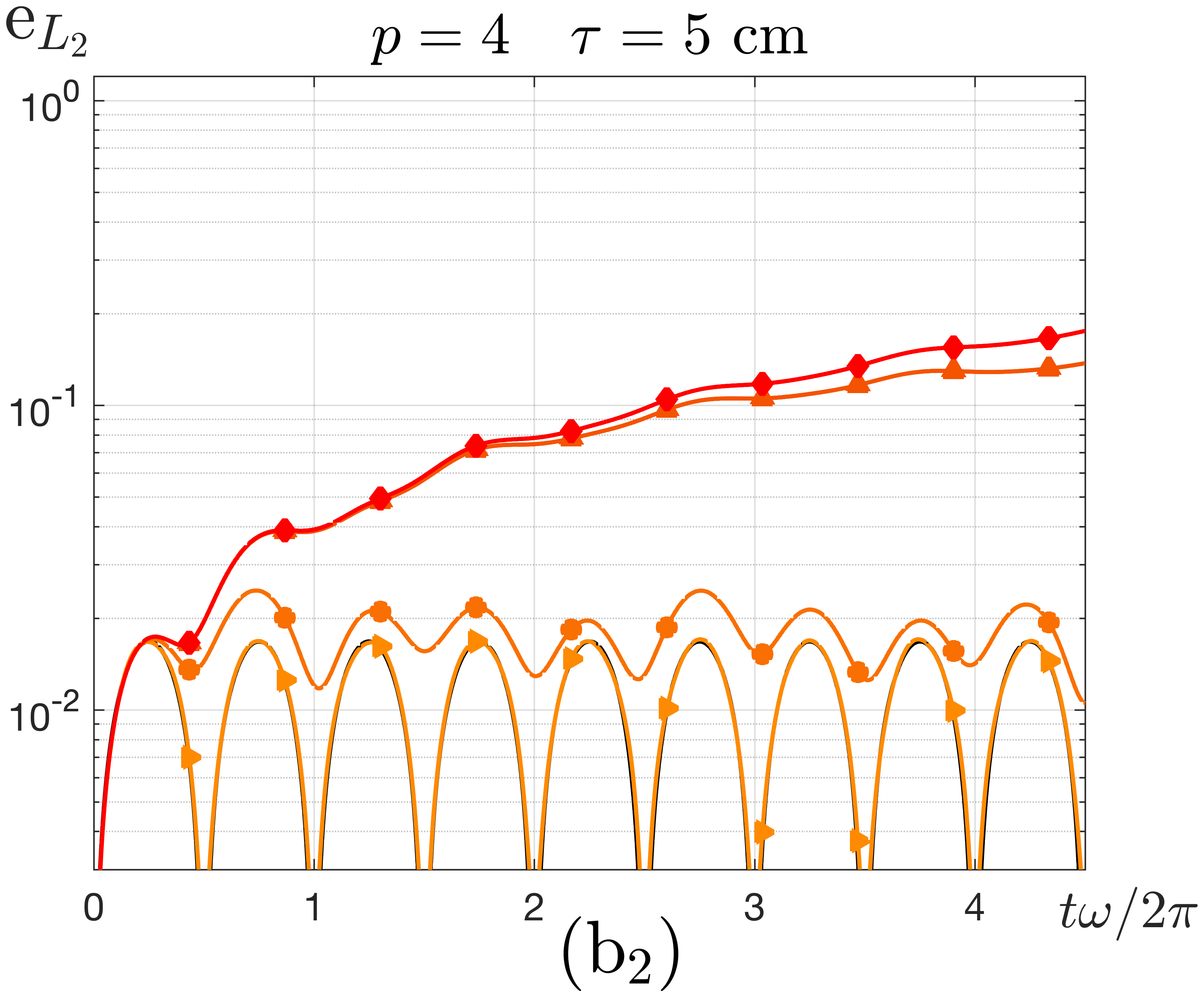}\end{subfigure}
    \begin{subfigure}[b]{0.44\textwidth}\includegraphics[height=4.0cm]{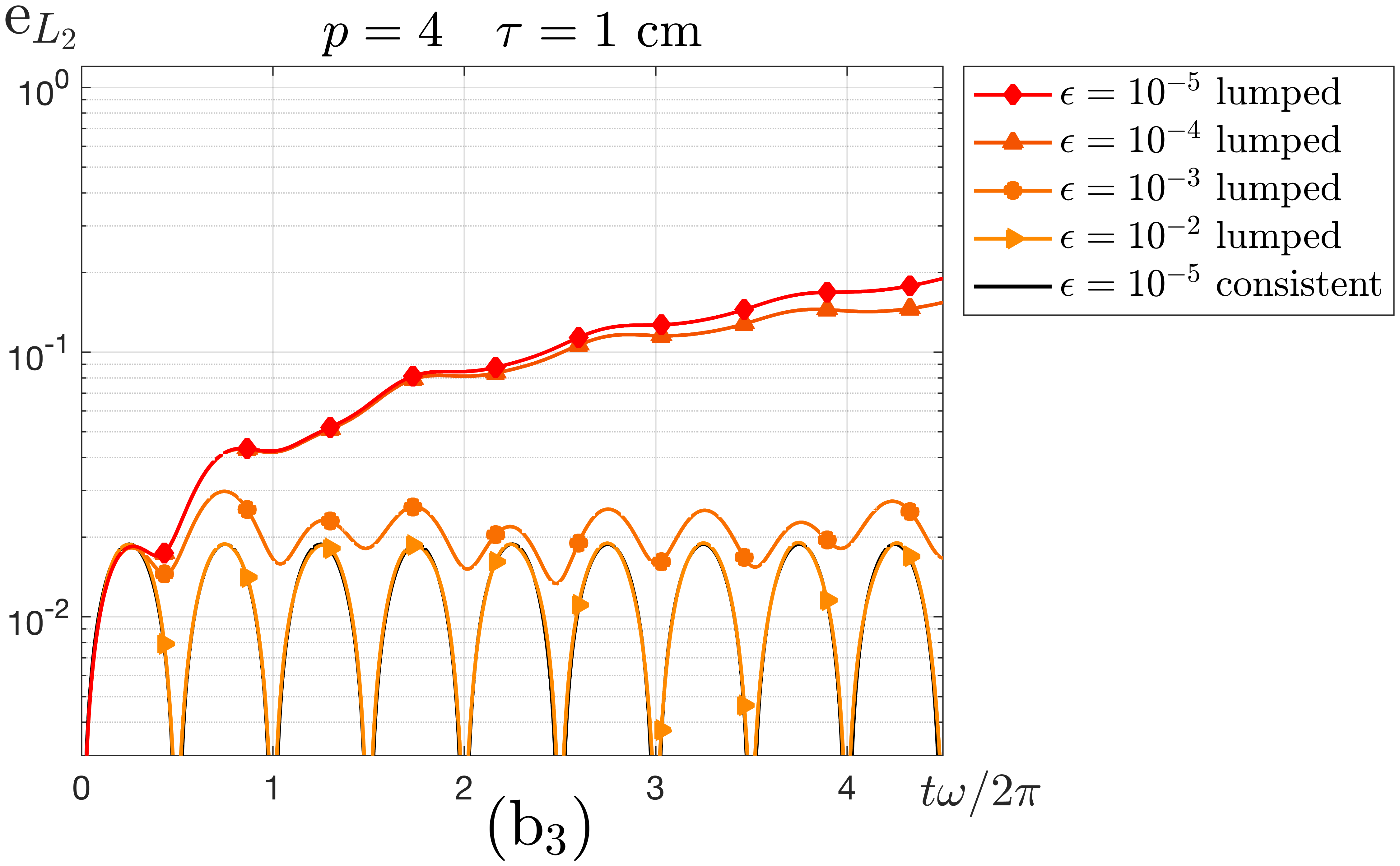}\end{subfigure}
    \caption{\Cref{ex: fuselage_window}: $L^2$ error for the consistent and lumped mass solutions for different spline orders, thicknesses and trimming parameters.} 
    \label{fig: fuselage_window_err_cons_lumped}
\end{figure}
\begin{figure}[H]	\centering
    \begin{subfigure}[b]{0.27\textwidth}\includegraphics[height=4.0cm]{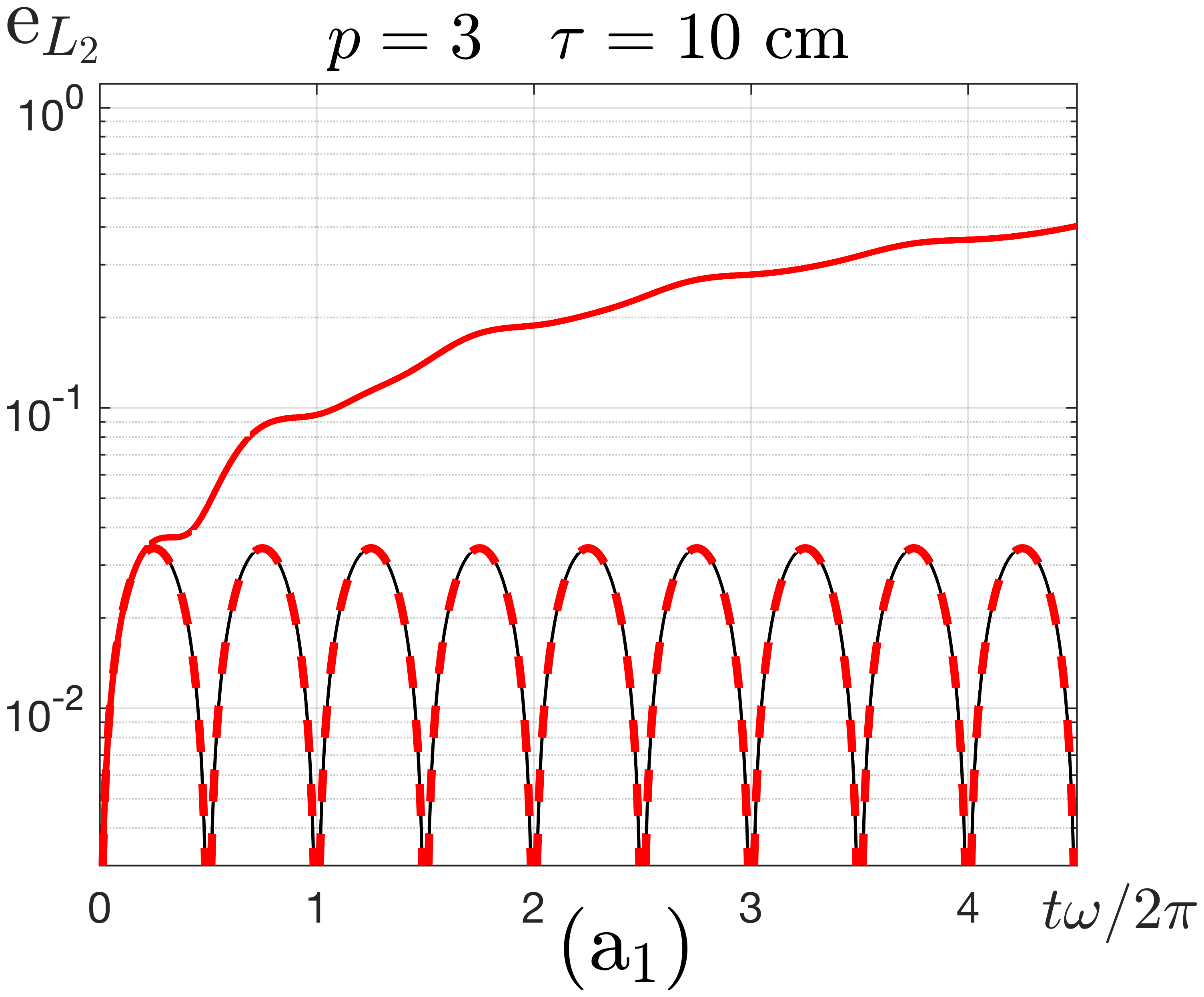}
    \end{subfigure}
    \begin{subfigure}[b]{0.27\textwidth}\includegraphics[height=4.0cm]{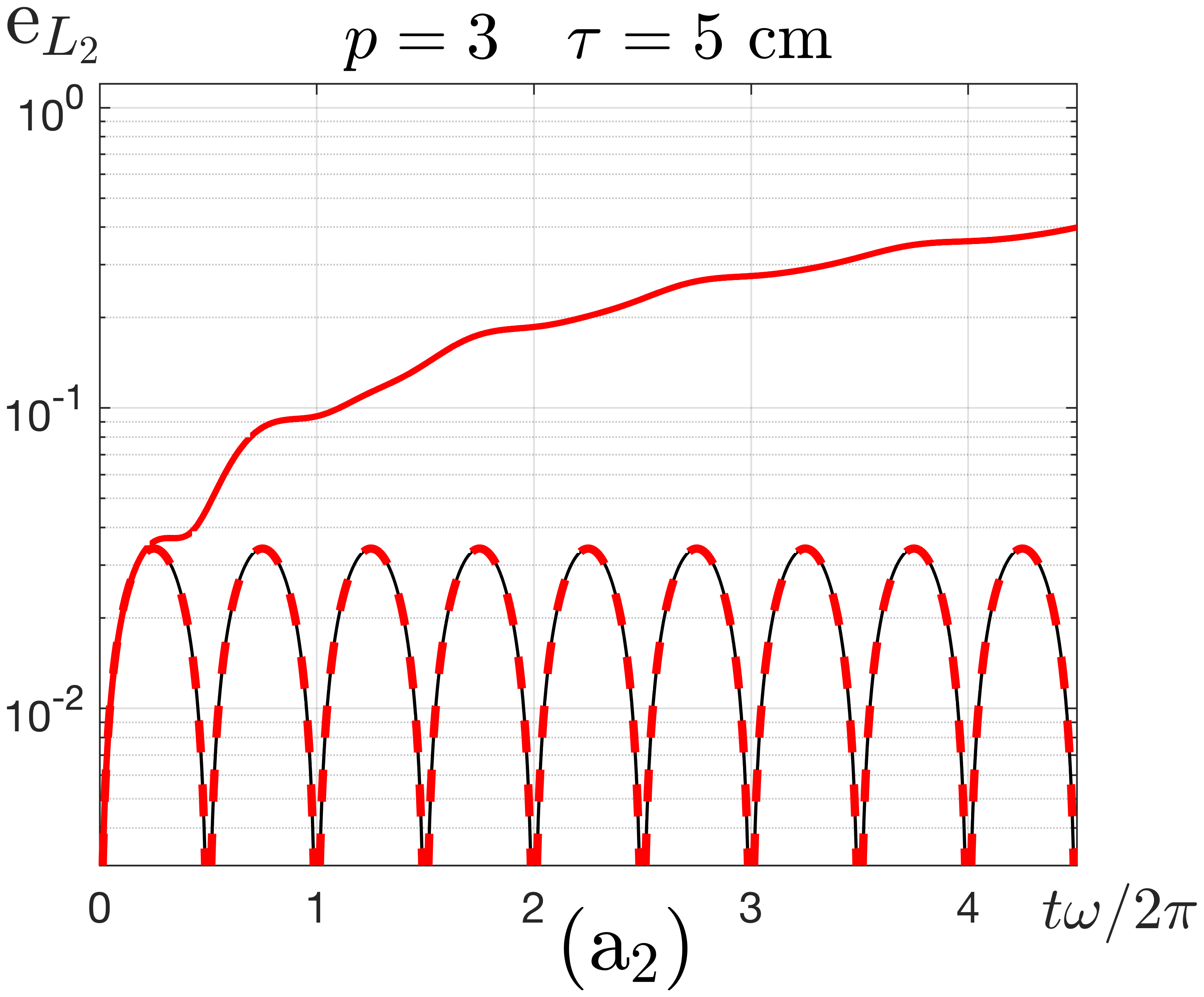}
    \end{subfigure}
    \begin{subfigure}[b]{0.44\textwidth}\includegraphics[height=4.0cm]{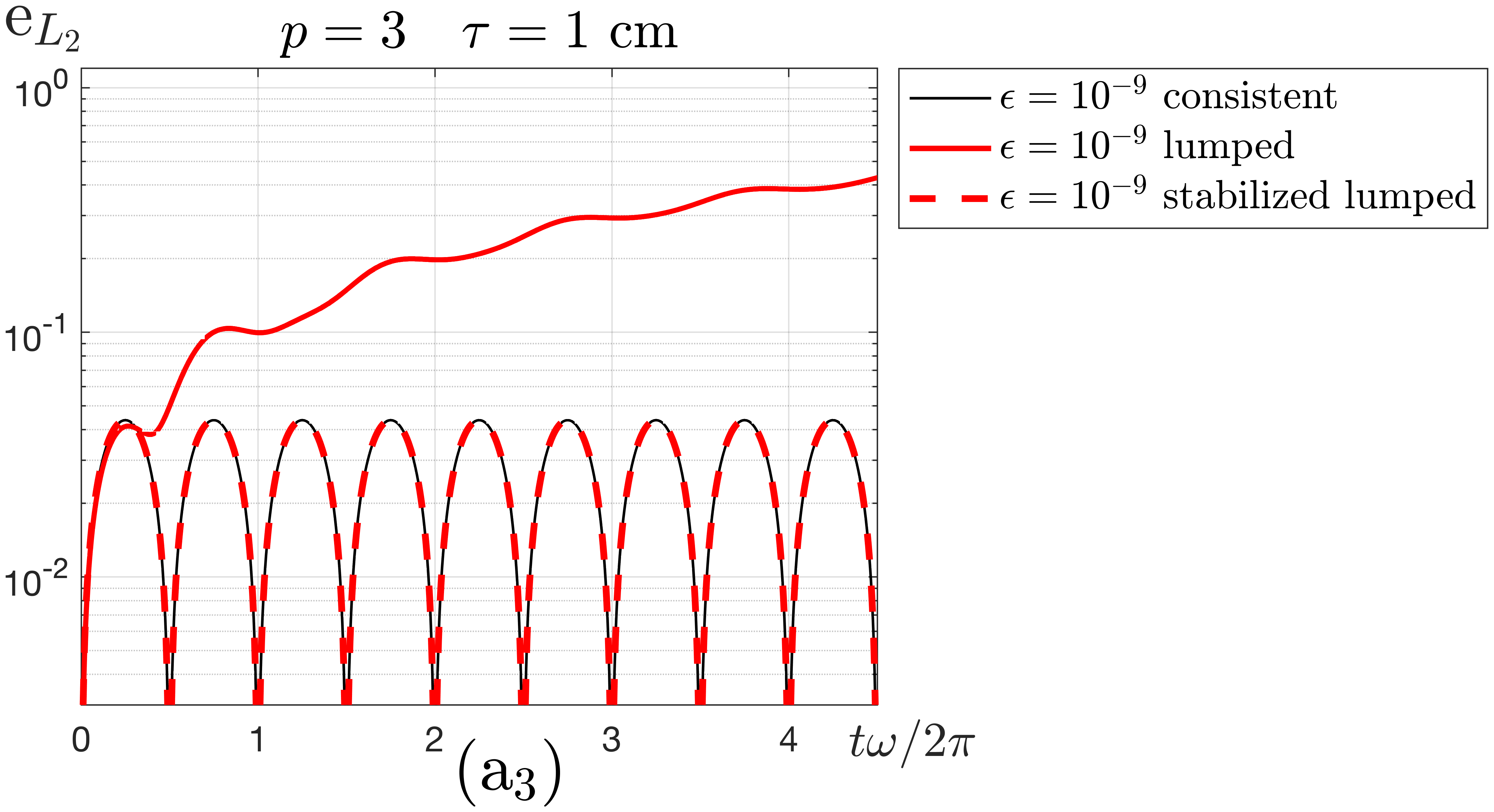}
    \end{subfigure}\\
    \begin{subfigure}[b]{0.27\textwidth}\includegraphics[height=4.0cm]{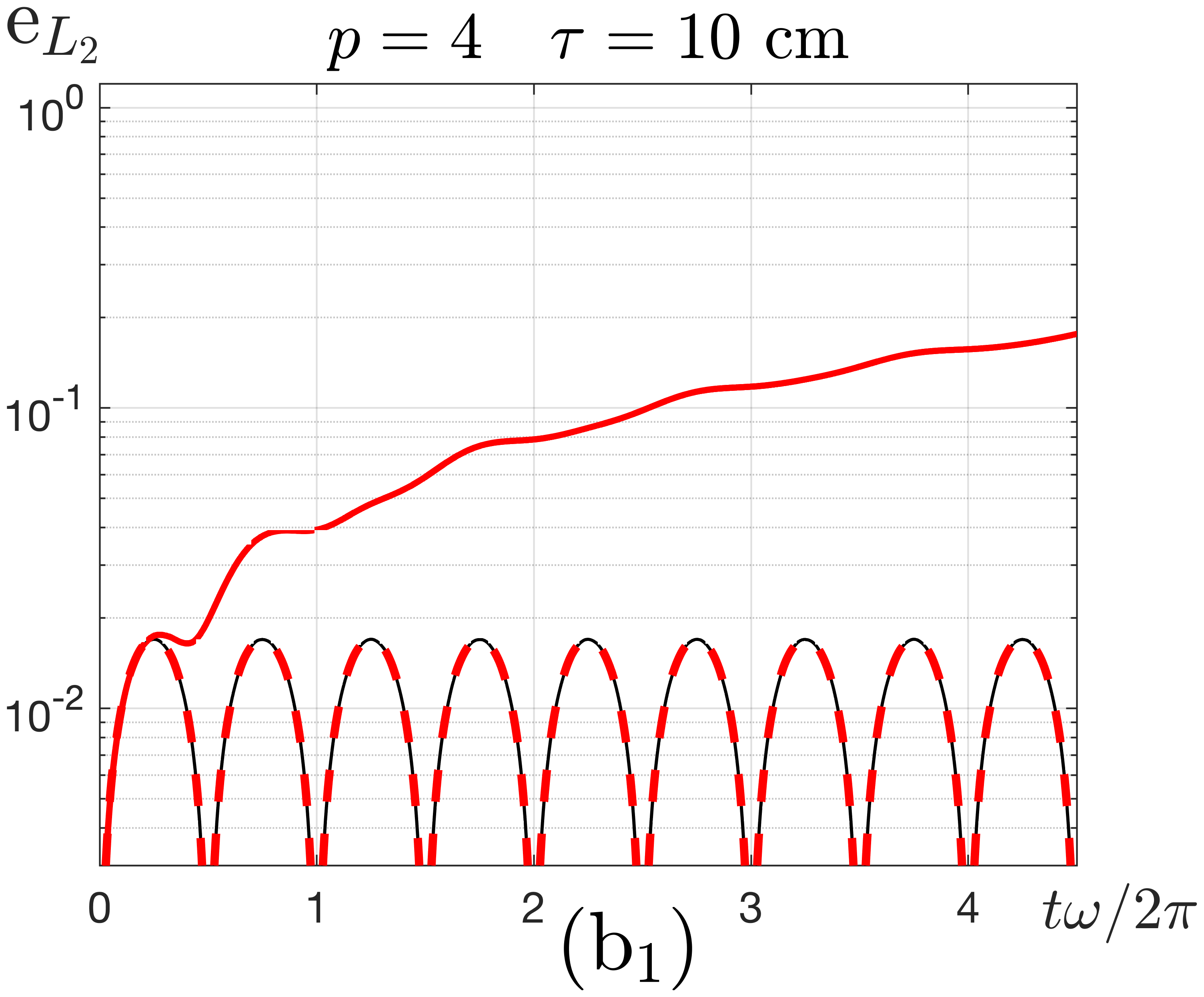}
    \end{subfigure}
    \begin{subfigure}[b]{0.27\textwidth}\includegraphics[height=4.0cm]{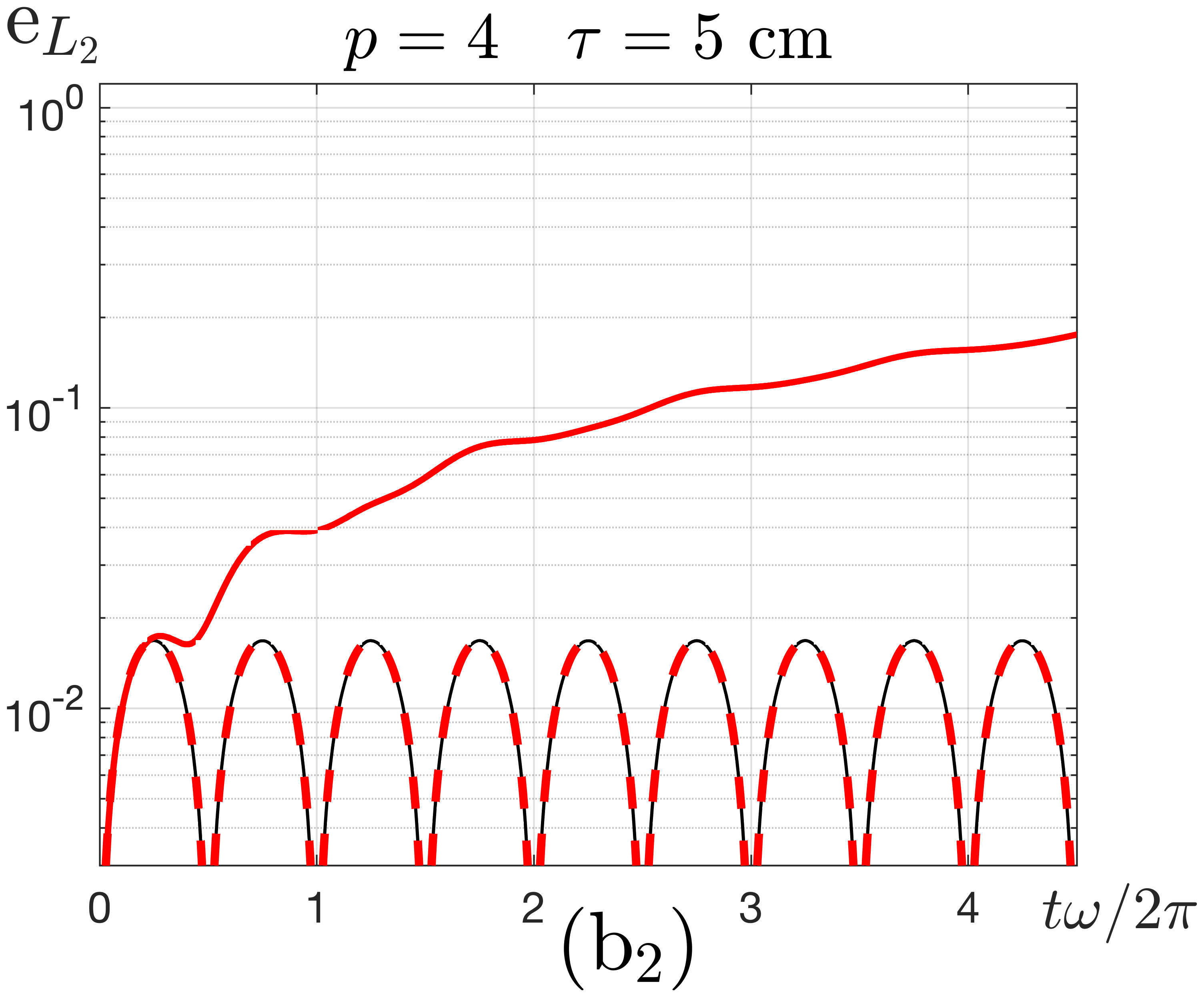}
    \end{subfigure}
    \begin{subfigure}[b]{0.44\textwidth}\includegraphics[height=4.0cm]{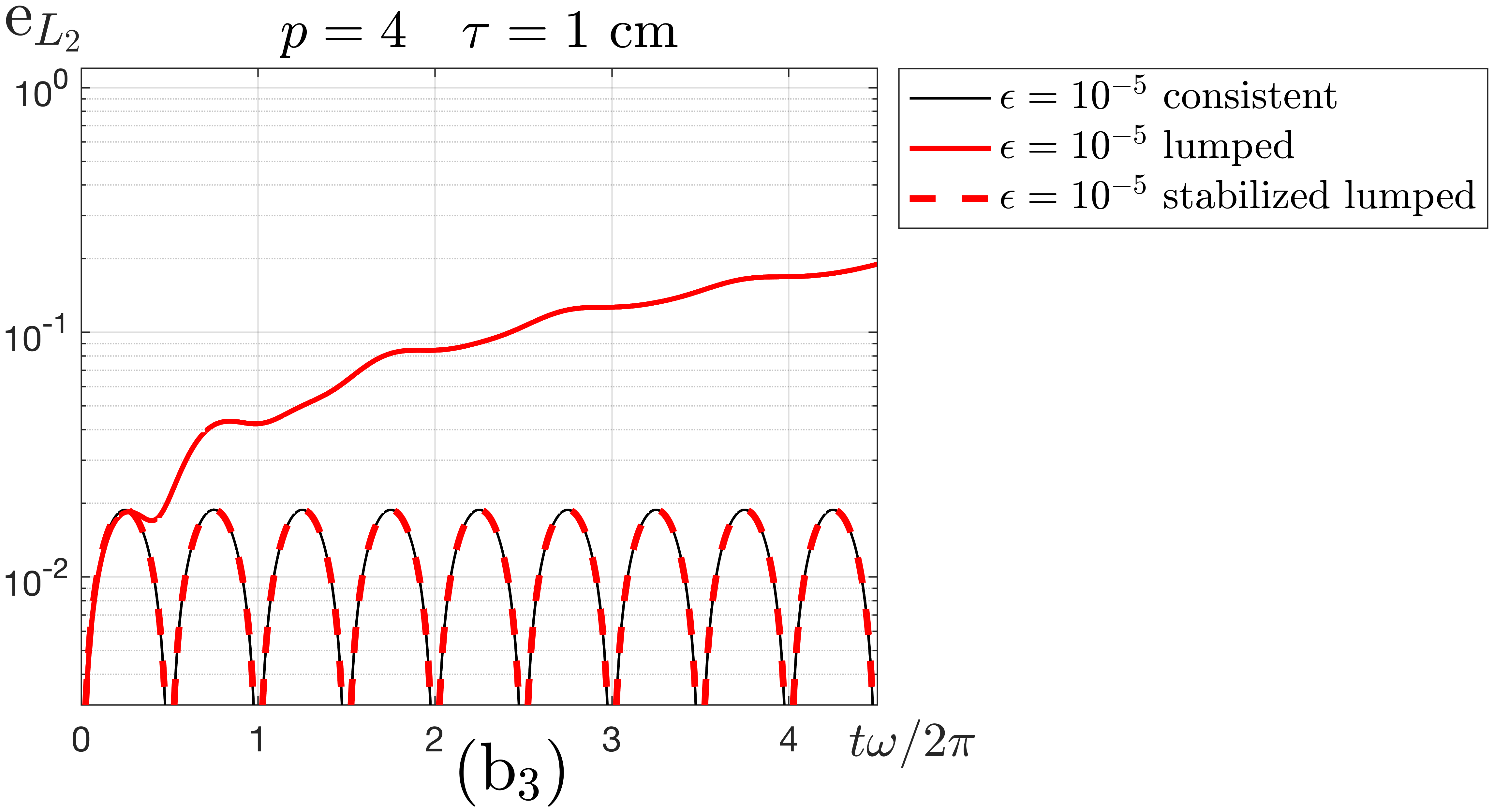}
    \end{subfigure}
    \caption{\Cref{ex: fuselage_window}: $L^2$ error for the consistent and (stabilized) lumped mass solutions for different spline orders, thicknesses and trimming parameters.}
    \label{fig: fuselage_window_err_cons_stab_lumped}
\end{figure}
\end{example}

\begin{example}[B-Pillar]
\label{ex: B-Pillar}
Our final example, inspired from \cite{coradello2020adaptive}, treats the B-pillar of a car. This vertical part, located between the front and rear doors, is an essential component to the vehicle's safety. The irregularly shaped object of length $L=1.4$ m, Young modulus $E=210\cdot10^9$ Pa and Poisson ratio $\nu=0.3$ is clamped on one side in all directions (type 1 boundary conditions) and is subjected to a displacement field $u_3=-0.025$ m on its opposite side. The other components of the displacement field on this same side are left unconstrained by prescribing Neumann boundary conditions instead. We refer to such hybrid cases as type 3 boundary conditions. Finally, to add additional complexity to the model, an L-shaped cut-out is created on the base of the B-pillar and Neumann (or type 2) boundary conditions are prescribed along its inner boundaries. All geometric details and boundary conditions are depicted in \figref{fig: bpillar}{a} and a close-up of the cut-out is shown in \figref{fig: bpillar}{b} together with the underlying mesh. Once again, some of the grid lines are exceedingly close to the inner boundaries, only at a distance $\delta =5\cdot 10^{-8}$. External forces and moments are prescribed on the triangular portion of the domain between the L-shaped cut-out the dashed blue segment $P_B-P_A = (0.09,\;0.09)$ m (see \figref{fig: bpillar}{b}). We denote $\bm{\eta}$ the downward pointing normal to this segment and $\bm{r}$ the usual outward pointing normal to the boundary. The shell's ``external forces'' $\bm{\mathsf{f}} = (\bm{f}, \bm{m})$ are 
\begin{align*}
    \bm{f}(\eta,t) &= -f_0\; e^{-\frac{b_0}{\eta}}\frac{ (b_0^2-2 b_0 \eta-n_0^2\pi^2\eta^4)\sin{(n_0\pi\eta)} + 2 b_0 n_0\pi\eta^2\cos{(n_0\pi\eta)}}{\eta^4}\phi(t) \;\bm{e}_3, \\
    \bm{m}(\eta,t) &=  m_0 \; e^{-\frac{b_0}{\eta}}\frac{b_0\sin{(n_0\pi \eta)}+\eta^2n_0\pi\cos{(n_0\pi\eta)}}{\eta^2} \phi(t) \;\bm{\eta},
\end{align*}
and the ``boundary tractions'' $\bm{\mathsf{h}} = (\bm{h}, \bm{n})$ are
\begin{align*}
    \bm{h}(\eta,t) &= h_0 \; e^{-\frac{b_0}{\eta}}\frac{b_0\sin{(n_0\pi \eta)}+\eta^2n_0\pi\cos{(n_0\pi\eta)}}{\eta^2} \phi(t) \;(\bm{\eta}\cdot \bm{r}) \bm{e}_3, \\
    \bm{n}(\eta,t)&=\bm{0},
\end{align*}
where $\eta$ denotes the coordinate along the normal vector $\bm{\eta}$ and $\phi(t) = \sin(\omega t)$. The parameter values are $f_0 = 1.7\cdot 10^7$ Pa/m$^2$, $m_0 = 1.7\cdot 10^7$ Pa/m, $h_0 = 1.7\cdot 10^7$ Pa/m, $b_0 = 0.1$, $n_0 = 90$ and $\omega = 10^4$ rad/s. The initial velocity conditions in the local triangular portion are
\begin{align*}
    &\dot{\bm{u}}(\eta,0) = v_0 \; e^{-\frac{b_0}{\eta}}\sin{(n_0\pi \eta)}\;\bm{e}_3,\\
    &\dot{\bm{\theta}}(\eta,0) = \bm{0},
\end{align*}
where $v_0 = 5 \cdot 10^2$ m/s. Finally, $\bm{u}(\bm{x},0)$ and $\bm{\theta}(\bm{x},0)$ are given by the solution of the static problem that only accounts for the boundary conditions of type 1 and 3. With such settings, this example aims at simulating a local perturbation to a structural component otherwise in equilibrium. The exact solution to this problem in unknown and we must content ourselves with a reference solution computed over a fine mesh. The numerical solutions for the consistent and (stabilized) lumped mass are computed with the same schemes as before. At a first glance, all solutions reported in \Cref{fig: bpillar_sol} at final time appear satisfactory. The magnitude of the displacement field is maximal at the top of B-pillar subjected to the prescribed displacement and gradually decreases to zero, matching the homogeneous Dirichlet boundary conditions at its base. However, a closer inspection reveals oscillations near one of the corners of the cut-out, which are confirmed by the close-up figures in the bottom row of \Cref{fig: bpillar_sol}. Although the results are not as clear-cut as in the previous examples, the oscillations for the (non-stabilized) lumped mass are somewhat sporadic and unexpected (\figref{fig: bpillar_sol}{l$_2$}). Such patterns do not appear in the solutions for the consistent and stabilized lumped mass (\figref{fig: bpillar_sol}{c$_2$} and \figref{fig: bpillar_sol}{s$_2$}, respectively), suggesting they are a pure artifact of mass lumping and trimming. This finding provides yet another example of the deleterious effects mass lumping may have for trimmed geometries, even in complex industrial-like settings.

\begin{figure}[H]	
    \centering
    \includegraphics[width=0.8\textwidth]{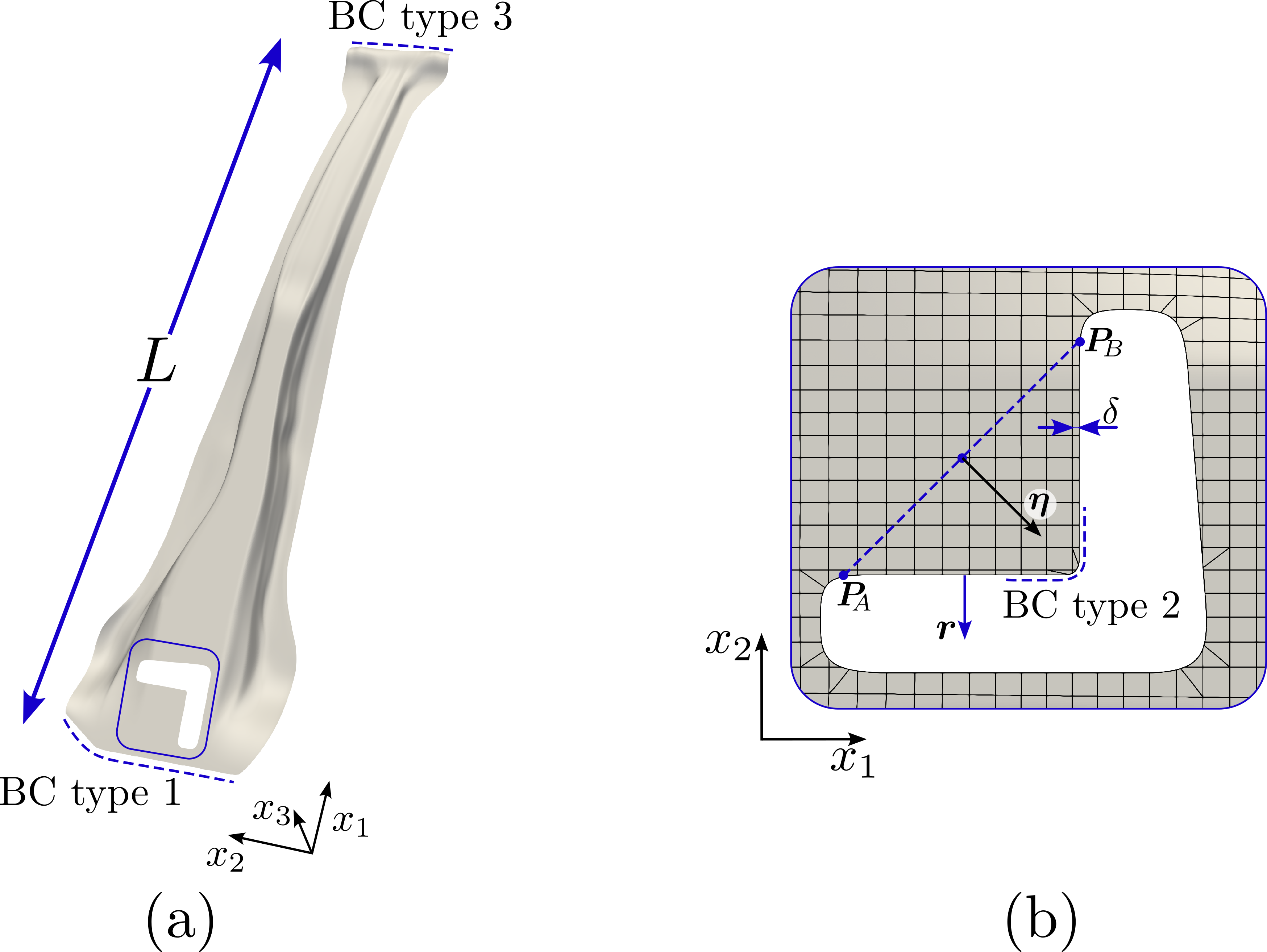}
    \caption{\Cref{ex: B-Pillar}: B-Pillar of a car.}
    \label{fig: bpillar}
\end{figure}

% \begin{align*}
%     \bm{f}(\eta,t) &= -f_0\; e^{-\frac{b_0}{\eta}}\frac{ (b_0^2-2 b_0 \eta-n_0^2\pi^2\eta^4)\sin{(n_0\pi\eta)} + 2 b_0 n_0\pi\eta^2\cos{(n_0\pi\eta)}}{\eta^4}\phi(t) \;\bm{e}_3, \\
%     \bm{m}(\eta,t) &=  m_0 \; e^{-\frac{b_0}{\eta}}\frac{b_0\sin{(n_0\pi \eta)}+\eta^2n_0\pi\cos{(n_0\pi\eta)}}{\eta^2} \phi(t) \;\bm{e}_\eta,
% \end{align*}

% \begin{align*}
%     \bm{h}(\eta,t) &= h_0 \; e^{-\frac{b_0}{\eta}}\frac{b_0\sin{(n_0\pi \eta)}+\eta^2n_0\pi\cos{(n_0\pi\eta)}}{\eta^2} \phi(t) \;(\bm{e}_\eta\cdot \bm{r}) \bm{e}_3, \\
%     \bm{n}(\eta,t)&=\bm{0},
% \end{align*}

% \begin{align*}
%     &\bm{u}(\bm{x},0) = \bm{0}, \\
%     &\dot{\bm{u}}(\bm{x},0) = v_0 \; e^{-\frac{b_0}{\eta}}\sin{(n_0\pi \eta)}\;\bm{e}_3.
% \end{align*}

% BC type 1 : clamped (Direchlet homogeneus for all the components)

% BC type 2 : Neumann boundary conditions with $\bm{h}$ and $\bm{n}$

% BC type 3 : Third component of the displacement vector $u_3=-u_0$ applied strongly. Neumann boundary conditions on all the other components of $\bm{u}$ and $\bm{\theta}$

% Young modulus: $E=210\cdot10^9$ Pa

% Poisson ration: $\nu=0.3$

% $b_0 = 0.1$

% $n_0 = 90$

% $f_0 = 1.7\cdot 10^7$ Pa/m$^2$

% $m_0 = 1.7\cdot 10^7$ Pa/m

% $h_0 = 1.7\cdot 10^7$ Pa/m

% $u_0 = 0.025$ m

% $v_0 = 5 \cdot 10^2$ m/s

% $P_B-P_A = (0.09,\;0.09)$ m

% $\phi(t) = \sin{\omega t}$

% $\omega = 10^4$ rad/s

% the domain terms are applied only on the triangle south-east-oriented with respect to the dashed line from $P_A$ to $P_B$ 

\begin{figure}[H]	
    \centering
    \includegraphics[width=0.9\textwidth]{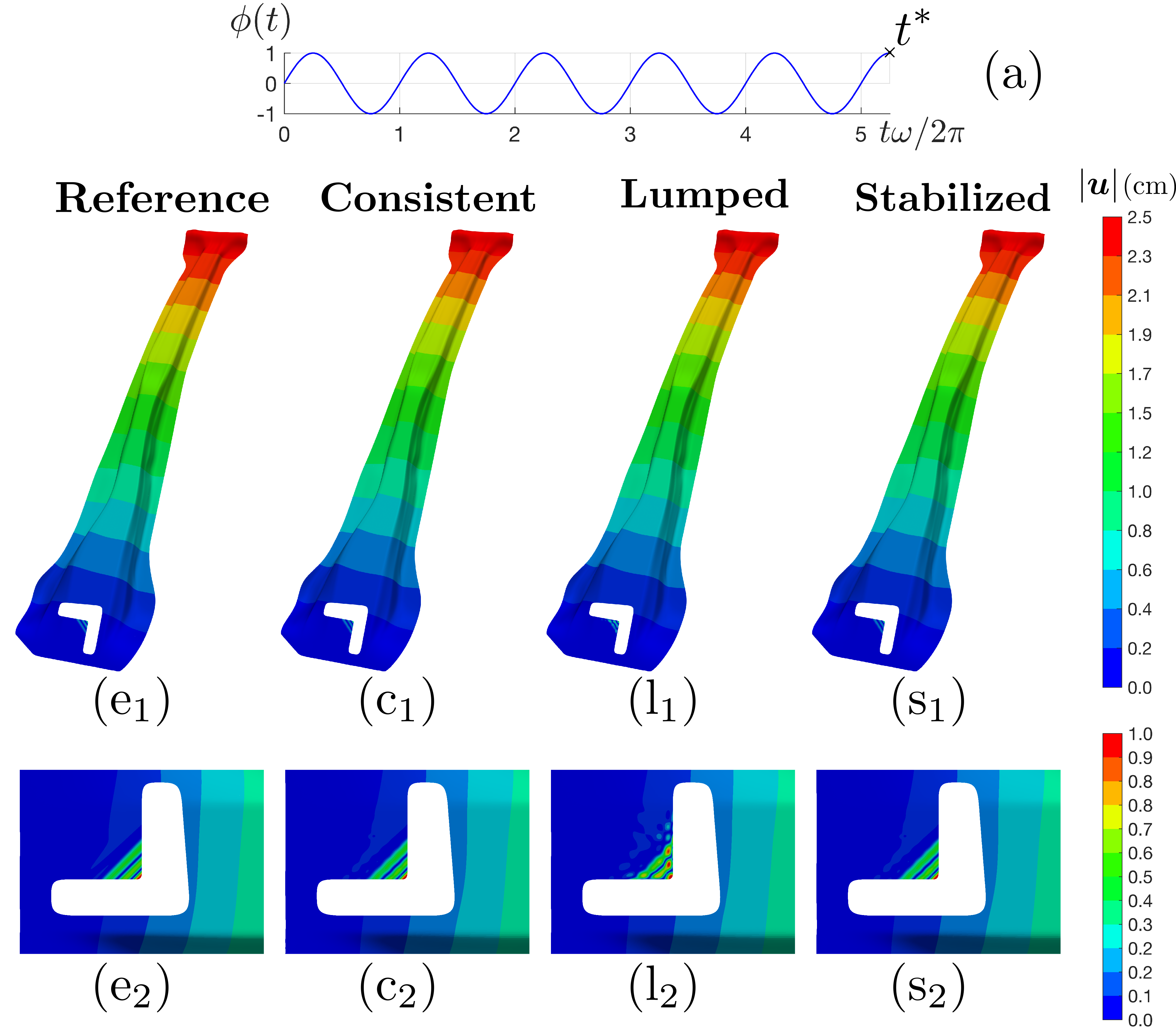}
    \caption{\Cref{ex: B-Pillar}: Snapshots of the reference and numerical solutions.}
    \label{fig: bpillar_sol}
\end{figure}
\end{example}

\section{Conclusion}
\label{se: conclusion}
Finite element shell models are ubiquitous in structural engineering and, as for other structural models, they commonly rely on mass lumping techniques for fast explicit solutions. However, for trimmed isogeometric discretizations, mass lumping may introduce spurious frequencies and modes in the low-frequency spectrum. These modes are not an issue as long as the structural response is orthogonal to them. However, if they are unluckily activated, they may trigger spurious oscillations in the solution thereby compromising simulation results. This fact was recently uncovered for standard wave equations in \cite{voet2025stabilization} and the extensive numerical experiments within this article indicate that the same issue arises for plate and shell models. Since our counter-examples were obtained not only for manufactured solutions but also for prescribed data, we cannot a priori exclude real engineering applications from suffering the same fate. Unfortunately, once lodged in the low-frequency spectrum, filtering out spurious eigenvalues and modes is difficult. Hence, we suggest modifying the discrete formulation prior to lumping the mass matrix. The modification consists in replacing polynomial segments of basis functions on small trimmed elements by extending those from large neighboring elements. This stabilization technique prevents trimming-related spurious modes from surfacing in the low-frequency spectrum when the mass matrix is lumped. As shown in our experiments, the stabilization restores in most cases a level of accuracy comparable to boundary-fitted discretizations. Nevertheless, the stabilization only resolves trimming-related issues and combining it with advanced mass lumping techniques or locking-free strategies might be necessary for further improving the accuracy of the solution.

\section*{Acknowledgments}
The authors kindly acknowledge the support of the Swiss National Science Foundation through the project “PDE tools for analysis-aware geometry processing in simulation science” n. 200021 215099, and through the project “Fast simulation tools for lattice structures” n. 200021 214987.

\appendix
\section{Covariant components of the linearized strain tensor}
\label{app: covariant_comp_strain}
The covariant components $\epsilon_{ij}$ of the strain tensor are its components in the contravariant basis; i.e.
\begin{equation*}
    \bm{\epsilon} = \epsilon_{ij} \bm{g}^i \otimes \bm{g}^j,
\end{equation*}
where $\{\bm{g}^1,\; \bm{g}^2,\; \bm{g}^3\}$ is the contravariant basis pertaining to the volume mapping. Thanks to the Kronecker delta property relating the covariant and contravariant bases, the $(i,j)$th component of the strain tensor is given by
\begin{equation*}
    \epsilon_{ij}= \bm{g}_i \cdot \bm{\epsilon} \bm{g}_j.
\end{equation*}
All we need for computing it is an expression for the strain tensor in curvilinear coordinates, which we compute from its expression in Cartesian coordinates. Indeed, the displacement field as a function of the Cartesian and curvilinear coordinates, denoted $\hat{\mathbf{u}}$ and $\mathbf{u}$, respectively, are related through the volume mapping $\bm{\mathcal{G}}$
\begin{equation*}
    \hat{\mathbf{u}}(\bm{x})=\hat{\mathbf{u}}(\bm{\mathcal{G}}(\bm{\xi}))=\mathbf{u}(\bm{\xi}).
\end{equation*}
Thus, given the familiar expression for the strain tensor in Cartesian coordinates
\begin{equation*}
    \hat{\bm{\epsilon}}(\bm{x}) = \frac{1}{2}\left(\nabla_{\bm{x}} \hat{\mathbf{u}}(\bm{x}) + \nabla_{\bm{x}} \hat{\mathbf{u}}(\bm{x})^T\right)
\end{equation*}
we obtain its expression in curvilinear coordinates as
\begin{equation*}
    \bm{\epsilon}(\bm{\xi}) = \hat{\bm{\epsilon}}(\bm{\mathcal{G}}(\bm{\xi})) = \frac{1}{2}\left(\nabla_{\bm{x}} \hat{\mathbf{u}}(\bm{\mathcal{G}}(\bm{\xi})) + \nabla_{\bm{x}} \hat{\mathbf{u}}(\bm{\mathcal{G}}(\bm{\xi}))^T\right) = \frac{1}{2}\left(\nabla_{\bm{\xi}} \mathbf{u}(\bm{\xi})\nabla_{\bm{\xi}} \bm{\mathcal{G}}(\bm{\xi})^{-1} + \nabla_{\bm{\xi}} \bm{\mathcal{G}}(\bm{\xi})^{-T} \nabla_{\bm{\xi}} \mathbf{u}(\bm{\xi})^T\right)
\end{equation*}
where the subscript for the gradient indicates the differentiation variable. Finally, recalling that the columns of $\nabla_{\bm{\xi}} \bm{\mathcal{G}}=[\bm{g}_1,\; \bm{g}_2,\; \bm{g}_3]$ are the covariant basis vectors, we obtain
\begin{align*}
    \epsilon_{ij}=\bm{g}_i \cdot \bm{\epsilon}(\bm{\xi})\bm{g}_j &= \frac{1}{2}\left(\bm{g}_i \cdot \nabla_{\bm{\xi}} \mathbf{u}(\bm{\xi})\nabla_{\bm{\xi}} \bm{\mathcal{G}}(\bm{\xi})^{-1}\bm{g}_j + \nabla_{\bm{\xi}} \mathbf{u}(\bm{\xi}) \nabla_{\bm{\xi}} \bm{\mathcal{G}}(\bm{\xi})^{-1}\bm{g}_i \cdot \bm{g}_j \right) \\
    &= \frac{1}{2}\left(\bm{g}_i \cdot \nabla_{\bm{\xi}} \mathbf{u}(\bm{\xi}) \bm{e}_j + \nabla_{\bm{\xi}} \mathbf{u}(\bm{\xi}) \bm{e}_i \cdot \bm{g}_j \right) \\
    &=\frac{1}{2}\left(\bm{g}_i \cdot \mathbf{u}_{,j} + \mathbf{u}_{,i} \cdot \bm{g}_j \right),
\end{align*}
which is the expression in \eqref{eq: covariant_comp}.

\section{From volume to surface integration}
\label{app: volume_to_surface}

\begin{figure}[H]	 
    \centering
    \includegraphics[width=0.7\textwidth]{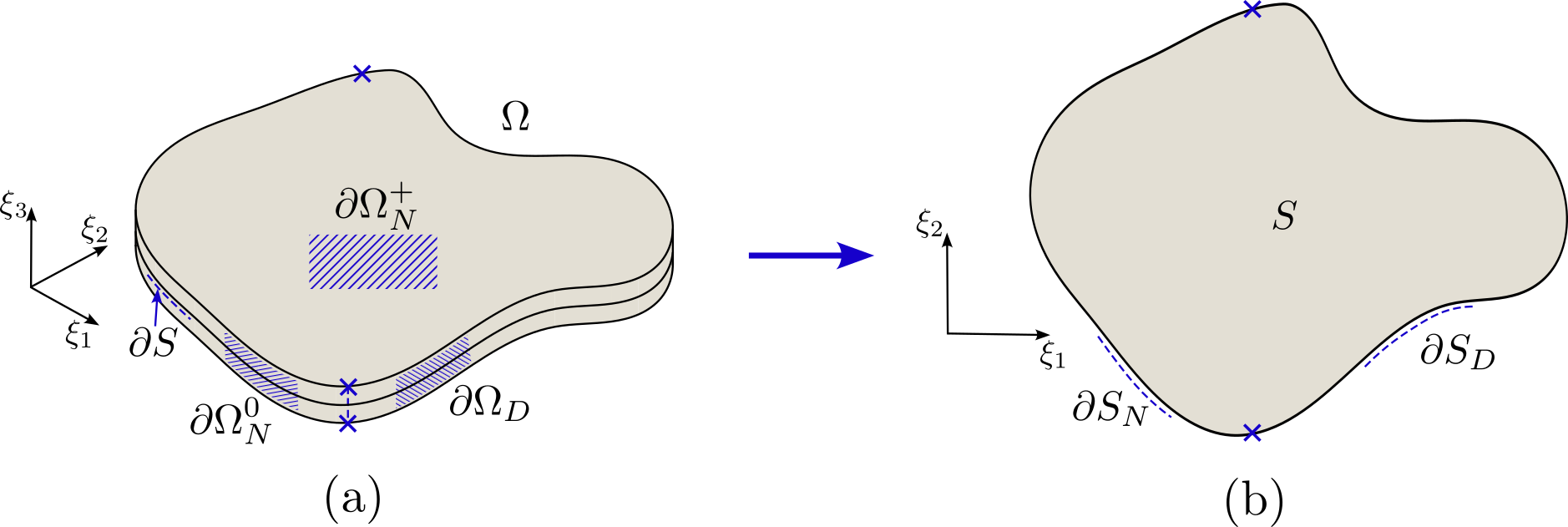}
    \caption{From 3D to 2D.}
    \label{fig: boundaries_def}
\end{figure}
In order to ensure that the article is self-contained, this appendix outlines how the general three-dimensional variational problem integrated over the volume may be reduced to a simpler two-dimensional one integrated over the mid-surface, as is customary for shell formulations \cite{guarino2021equivalent,guarino2023discontinuous}. For the Reissner-Mindlin model, this process is not so obvious and requires several ad hoc assumptions and approximations, which are somewhat scattered in the literature \cite{ciarlet2005introduction,ciarlet2022mathematical,chapelle2010finite,bischoff2004models}. Hence, we summarize the main steps in this appendix.

We recall that the transverse coordinate $\xi_3$ varies in the interval $T = \left(-\frac{\tau}{2}, \ \frac{\tau}{2}\right)$ and note that the Neumann boundary is composed of three parts: $\partial \Omega_N = \partial \Omega_N^- \cup \partial \Omega_N^+ \cup \partial \Omega_N^0$ (see \Cref{fig: boundaries_def}) where

\begin{align*}
    \partial \Omega_N^- = S \times \left\{-\frac{\tau}{2}\right\}, \quad \partial \Omega_N^+ = S \times \left\{\frac{\tau}{2}\right\} \quad \text{and} \quad \partial \Omega_N^0 = \partial S_N \times T.
\end{align*}

The assumption on the displacement field \eqref{eq: disp_field} allows to considerably simplify the general three-dimensional variational problem: find $\mathbf{u}(t) \in V = \{\mathbf{v} \in [H^1]^3 \colon \mathbf{v}|_{\partial \Omega_D} = 0\}$ such that
\begin{equation}
\label{eq: 3D_elasticity_curvilinear_app}
    \int_{\Omega} \rho \ddot{\mathbf{u}} \cdot \mathbf{v} \sqrt{g} \ \dd \Omega + \int_{\Omega} \bm{\epsilon}(\mathbf{u}) \colon \bm{C} \colon \bm{\epsilon}(\mathbf{v}) \sqrt{g} \ \dd \Omega = \int_{\Omega} \mathbf{f} \cdot \mathbf{v} \sqrt{g} \ \dd \Omega + \int_{\partial \Omega_N} \mathbf{h} \cdot \mathbf{v} \sqrt{g} \ \dd \partial \Omega \quad \forall \mathbf{v} \in V.
\end{equation}
The steps for reducing \eqref{eq: 3D_elasticity_curvilinear_app} to an analogous (but approximate) two-dimensional problem are rather nontrivial. The main difficulty lies in isolating the dependency on $\xi_3$ from the other two curvilinear coordinates in order to integrate along the thickness (assumed constant). Our derivation is quite informal and follows an approach similar to \cite[][Chapter 3]{ciarlet2022mathematical} whereby high order terms in $\xi_3$ are dropped out. We will first investigate how $g$ (the determinant of the volume metric tensor) is approximated before analyzing each term of \eqref{eq: 3D_elasticity_curvilinear_app} individually from left to right.

\subsection{Volume metric}
Let us begin with the determinant of the volume metric tensor $g=\det(\nabla \bm{\mathcal{G}}^T \nabla \bm{\mathcal{G}}) = \det(\nabla \bm{\mathcal{G}})^2$. From the multilinearity of the determinant and the expression for the covariant basis vectors \eqref{eq: vol_covariant_basis}, we obtain
\begin{align*}
    \det(\nabla \bm{\mathcal{G}}) &= \det([\bm{a}_1 + \xi_3 \bm{a}_{3,1},\; \bm{a}_2 + \xi_3 \bm{a}_{3,2},\; \bm{a}_3]) \\
    &= \det([\bm{a}_1,\; \bm{a}_2,\; \bm{a}_3]) + \xi_3\left(\det([\bm{a}_{3,1},\; \bm{a}_2,\; \bm{a}_3])+\det([\bm{a}_1,\; \bm{a}_{3,2},\; \bm{a}_3])\right) + \xi_3^2 \det([\bm{a}_{3,1},\; \bm{a}_{3,2},\; \bm{a}_3]).
\end{align*}
Moreover, since $\bm{a}_i \cdot \bm{a}_3 = \delta_{i3}$,
\begin{equation}
\label{eq: det}
    \det([\bm{a}_1,\; \bm{a}_2,\; \bm{a}_3])^2 = 
    \begin{vmatrix}
        \bm{a}_1 \cdot \bm{a}_1 & \bm{a}_1 \cdot \bm{a}_2 & 0 \\
        \bm{a}_2 \cdot \bm{a}_1 & \bm{a}_2 \cdot \bm{a}_2 & 0 \\
        0 & 0 & 1
    \end{vmatrix}
    = \det(\nabla \bm{\mathcal{F}}^T \nabla \bm{\mathcal{F}}) = a.
\end{equation}
Hence, a first order approximation is given by
\begin{equation*}
    g=\det(\nabla \bm{\mathcal{G}}^T \nabla \bm{\mathcal{G}}) = \det(\nabla \bm{\mathcal{G}})^2 \approx \det([\bm{a}_1,\; \bm{a}_2,\; \bm{a}_3])^2 = \det(\nabla \bm{\mathcal{F}}^T \nabla \bm{\mathcal{F}}) = a.
\end{equation*}
This approximation justifies substituting $\sqrt{g}$ with $\sqrt{a}$ in \eqref{eq: 3D_elasticity_curvilinear_app} for thin-walled structures.

\subsection{Inertia forces}
Now we simplify each term from left to right in \eqref{eq: 3D_elasticity_curvilinear_app} after replacing the trial and test functions by their expression according to the Reissner-Mindlin assumption
\begin{align*}
    \mathbf{u}(\xi_1,\xi_2,\xi_3,t) &= \bm{u}(\xi_1,\xi_2,t)+\xi_3 \bm{\theta}(\xi_1,\xi_2,t), \\
    \mathbf{v}(\xi_1,\xi_2,\xi_3) &= \bm{v}(\xi_1,\xi_2)+\xi_3 \bm{\phi}(\xi_1,\xi_2).
\end{align*}
The first term then becomes
\begin{align*}
    \int_{\Omega} \rho \ddot{\mathbf{u}} \cdot \mathbf{v} \sqrt{g} \ \dd \Omega &\approx
    \int_{S} \int_T \rho (\ddot{\bm{u}} + \xi_3 \ddot{\bm{\theta}}) \cdot (\bm{v}+\xi_3 \bm{\phi}) \sqrt{a} \ \dd \xi_3 \dd S \\
    &= \int_{S} \left(\ddot{\bm{u}} \cdot \bm{v} \int_T \rho \dd \xi_3  + \left(\ddot{\bm{u}} \cdot \bm{\phi} + \bm{v} \cdot \ddot{\bm{\theta}} \right) \int_T \rho \xi_3 \dd \xi_3 + \ddot{\bm{\theta}} \cdot \bm{\phi} \int_T \rho \xi_3^2 \dd \xi_3 \right) \sqrt{a} \dd S.
\end{align*}
If $\rho$ is assumed independent of $\xi_3$, the above expression reduces to
\begin{equation}
\label{eq: inertia_forces}
    \int_{S} \rho_u \ddot{\bm{u}} \cdot \bm{v} + \rho_\theta \ddot{\bm{\theta}} \cdot \bm{\phi} \sqrt{a} \dd S = \int_{S}(\ddot{\bm{\mathsf{u}}} \cdot W \bm{\mathsf{v}}) \sqrt{a} \dd S
\end{equation}
with $\bm{\mathsf{u}}=(\bm{u},\bm{\theta})$, $\bm{\mathsf{v}}=(\bm{v},\bm{\phi})$ and $W=\diag(\rho_u I_3,\; \rho_\theta I_3)$ where
\begin{equation*}
    \rho_u = \int_T \rho \dd \xi_3 = \tau \rho \quad \text{and} \quad \rho_{\theta} = \int_T \rho \xi_3^2 \dd \xi_3 = \frac{\tau^3}{12}\rho.
\end{equation*}

\begin{remark}
If $\rho$ depends on $\xi_3$, $W$ is non-diagonal (but symmetric).
\end{remark}

\subsection{Internal forces}
The derivations for the second term (related to the internal forces) is quite lengthy. We first express it componentwise as
\begin{equation*}
    \int_{\Omega} \bm{\epsilon}(\mathbf{u}) \colon \bm{C} \colon \bm{\epsilon}(\mathbf{v}) \sqrt{g} \ \dd \Omega = \int_{\Omega} C^{ijkl} \epsilon_{kl}(\mathbf{u})\epsilon_{ij}(\mathbf{v}) \sqrt{g} \ \dd \Omega
\end{equation*}
where $C^{ijkl}$ are the contravariant components of $\bm{C}$, which, for a linear elastic homogeneous and isotropic material, are given by
\begin{equation*}
    C^{ijkl} = \lambda g^{ij}g^{kl} + \mu(g^{ik}g^{jl} + g^{il}g^{jk})
\end{equation*}
and $\epsilon_{ij}$ are the covariant components of the strain tensor, listed in \eqref{eq: linearized_strains} and recalled below:
\begin{subequations}
\label{eq: linearized_strains_app}
\begin{align}
    \epsilon_{\alpha\beta} &= \varepsilon_{\alpha \beta} + \xi_3 \kappa_{\alpha \beta} + \xi_3^2 \chi_{\alpha\beta},  \label{eq: eps_alpha_beta} \\
    \epsilon_{3\alpha} &= \frac{\gamma_{\alpha}}{2},\\
    \epsilon_{33} &= 0. 
\end{align}
\end{subequations}
The expressions for the membrane, bending and shear strains ($\bm{\varepsilon}$, $\bm{\kappa}$ and $\bm{\chi}$, respectively) are given in \eqref{eq: strains_def} but are not needed here; one must only bear in mind that they do not depend on $\xi_3$. Since the components of the elasticity tensor $\bm{C}$ do depend on $\xi_3$, we will first substitute them with a first order approximation. For this purpose, recall that the contravariant components of the volume metric tensor $g^{ij}$ are the components of $(\nabla \bm{\mathcal{G}}^T \nabla \bm{\mathcal{G}})^{-1}$. Let us decompose $\nabla \bm{\mathcal{G}}$ as
\begin{equation*}
    \nabla \bm{\mathcal{G}} = \bm{A} + \xi_3 \bm{X}
\end{equation*}
where
\begin{equation*}
    \bm{A} = [\bm{a}_1,\; \bm{a}_2,\; \bm{a}_3], \quad \text{and} \quad \bm{X} = [\bm{a}_{3,1},\; \bm{a}_{3,2},\; \bm{0}].
\end{equation*}
Since $\bm{A}$ is invertible, $\nabla \bm{\mathcal{G}}^{-1} = \bm{A}^{-1}(\bm{I} + \xi_3 \bm{X}\bm{A}^{-1})^{-1}$ and if $\xi_3$ is small enough (such that $\|\xi_3 \bm{X}\bm{A}^{-1}\| < 1$ in any suitable matrix norm), a Neumann series expansion yields (see \cite[][Corollary 5.6.16]{horn2012matrix} and the subsequent exercise therein)
\begin{equation*}
    \nabla \bm{\mathcal{G}}^{-1} = \bm{A}^{-1} \sum_{k=0}^{\infty} (-1)^k (\xi_3 \bm{X}\bm{A}^{-1})^k = \bm{A}^{-1} - \xi_3 \bm{A}^{-1} \bm{X} \bm{A}^{-1} + \mathcal{O}(\xi_3^2).
\end{equation*}
Hence, $\bm{A}^{-1}$ is a first order approximation of $\nabla \bm{\mathcal{G}}^{-1}$ and therefore
\begin{equation*}
    (\nabla \bm{\mathcal{G}}^T \nabla \bm{\mathcal{G}})^{-1} = (\bm{A}^T \bm{A})^{-1} + \mathcal{O}(\xi_3).
\end{equation*}
The foregoing expression justifies approximating the contravariant components of the volume metric tensor as $g^{ij} \approx a^{ij}$. Hence,
\begin{equation*}
    C^{ijkl} = \lambda g^{ij}g^{kl} + \mu(g^{ik}g^{jl} + g^{il}g^{jk}) \approx A^{ijkl} := \lambda a^{ij}a^{kl} + \mu(a^{ik}a^{jl} + a^{il}a^{jk}).
\end{equation*}
This approximation allows reducing the three-dimensional constitutive model to a two-dimensional one. For future reference, we note that $A$ inherits the major and minor symmetries of $C$ (i.e. $A^{ijkl}=A^{klij}$, $A^{ijkl}=A^{jikl}$ and $A^{ijkl}=A^{ijlk}$) and yields several useful identities:
\begin{subequations}
\label{eq: comp_A}
\begin{align}
    A^{33\alpha\beta} &= \lambda a^{\alpha\beta}, & A^{3\alpha\beta3} &= \mu a^{\alpha\beta},  & A^{3333} = \lambda + 2\mu, \\
    A^{\alpha\beta\gamma3} &= 0, & A^{\alpha333} &=0. & & 
\end{align}
\end{subequations}
Imposing the plane stress assumption ($\sigma_{33}=0$) allows to eliminate $\epsilon_{33}$ from the equations and the Lamé constants in this context are given by
\begin{equation*}
    \lambda=\frac{E\nu}{1-\nu^2} \quad \text{and} \quad \mu=\frac{E}{2(1+\nu)}.
\end{equation*}
Note that, within the Reissner-Mindlin theory, the plane stress assumption is incompatible with the condition 
$\epsilon_{33}=0$, which is a direct consequence of the kinematic hypothesis on the displacement field. However, the accuracy of the theory is not affected by this inconsistency. From the three-dimensional constitutive relation (with $C$ substituted for $A$) and the identities in \eqref{eq: comp_A},
\begin{equation}
\label{eq: eps_33}
    \sigma_{33} = A^{33kl} \epsilon_{kl} = A^{33\alpha\beta} \epsilon_{\alpha\beta} + 2A^{33\alpha3} \epsilon_{\alpha3} + A^{3333} \epsilon_{33} = \lambda a^{\alpha\beta}\epsilon_{\alpha\beta} + (\lambda+2\mu)\epsilon_{33} = 0 \implies \epsilon_{33} = -\frac{\lambda a^{\gamma \delta} \epsilon_{\gamma\delta}}{\lambda+2\mu}.
\end{equation}
Now we simplify the constitutive relation by exploiting \eqref{eq: comp_A} and substituting $\epsilon_{33}$ by its expression in \eqref{eq: eps_33}
\begin{align*}
    A^{ijkl}\epsilon_{ij}\epsilon_{kl} &= A^{\alpha\beta kl} \epsilon_{\alpha\beta}\epsilon_{kl} + 2A^{3\beta kl} \epsilon_{3\beta}\epsilon_{kl} + A^{33 kl} \epsilon_{33}\epsilon_{kl} \\
    &= \epsilon_{\alpha\beta}\left(A^{\alpha\beta\gamma\delta}\epsilon_{\gamma\delta} + 2A^{\alpha\beta3\delta}\epsilon_{3\delta} + A^{\alpha\beta33}\epsilon_{33}\right) + 2\epsilon_{3\beta}\left(A^{3\beta\gamma\delta}\epsilon_{\gamma\delta}+2A^{3\beta3\delta}\epsilon_{3\delta}+A^{3\beta33}\epsilon_{33}\right) \\
    &= \epsilon_{\alpha\beta}\left(A^{\alpha\beta\gamma\delta} -\frac{\lambda^2 a^{\alpha\beta} a^{\gamma\delta}}{\lambda+2\mu} \right)\epsilon_{\gamma\delta} + 4\epsilon_{3\beta}A^{3\beta3\delta}\epsilon_{3\delta} \\
    &= \epsilon_{\alpha\beta}\left(\frac{2 \lambda \mu}{\lambda+2\mu}a^{\alpha\beta} a^{\gamma\delta} + \mu(a^{\alpha\gamma} a^{\beta\delta}+a^{\alpha\delta} a^{\beta\gamma})\right)\epsilon_{\gamma\delta} + \gamma_\alpha (\mu a^{\alpha\beta}) \gamma_\beta.
\end{align*}
We denote
\begin{equation*}
    E^{\alpha\beta\gamma\delta}=\frac{2 \lambda \mu}{\lambda+2\mu}a^{\alpha\beta} a^{\gamma\delta} + \mu(a^{\alpha\gamma} a^{\beta\delta}+a^{\alpha\delta} a^{\beta\gamma})
\end{equation*}
the shell elasticity tensor. Substituting $\epsilon_{\alpha\beta}$ with its expression in \eqref{eq: eps_alpha_beta}, neglecting high order terms and integrating over the thickness yields
\begin{align*}
   \int_T A^{ijkl}\epsilon_{ij}\epsilon_{kl} \dd \xi_3 &\approx \int_T \left(\varepsilon_{\alpha\beta} E^{\alpha\beta\gamma\delta} \varepsilon_{\gamma\delta} + 2\xi_3 \varepsilon_{\alpha\beta} E^{\alpha\beta\gamma\delta} \kappa_{\gamma\delta} + \xi_3^2 \kappa_{\alpha\beta} E^{\alpha\beta\gamma\delta} \kappa_{\gamma\delta} \right) \dd \xi_3 + \int_T \gamma_\alpha (\mu a^{\alpha\beta}) \gamma_\beta \dd \xi_3 \\
   &= \bm{N} \colon \bm{\varepsilon} +  \bm{M} \colon \bm{\kappa}+ \bm{Q} \cdot \bm{\gamma}
\end{align*}
where 
\begin{align}
    N^{\alpha\beta} = E_u^{\alpha\beta\gamma\delta}\varepsilon_{\gamma\delta}, \quad M^{\alpha\beta} = E_\theta^{\alpha\beta\gamma\delta}\kappa_{\gamma\delta}, \quad \text{and} \quad Q^{\alpha} = \mu \tau a^{\alpha\beta }\gamma_\beta
\end{align}
are called the generalized membrane, bending and shear stresses, respectively, and $E_u^{\alpha\beta\gamma\delta} = \tau E^{\alpha\beta\gamma\delta}$ and $E_\theta^{\alpha\beta\gamma\delta}=\frac{\tau^3}{12}E^{\alpha\beta\gamma\delta}$ are simply rescaled versions of the shell elasticity tensor. In practice, the contribution from shear stresses is scaled with an ad hoc shear correction factor $\alpha_s$ such that $Q^{\alpha} = \mu\alpha_s\tau a^{\alpha\beta }\gamma_\beta$. This factor ensures consistency with classical bending theory \cite[][Chapter 6]{hughes2012finite}. Finally, putting all the pieces back together, we obtain
\begin{equation}
\label{eq: internal_forces}
    \int_{\Omega} \bm{\epsilon}(\mathbf{u}) \colon \bm{C} \colon \bm{\epsilon}(\mathbf{v}) \sqrt{g} \ \dd \Omega \approx \int_{S} (\bm{N}(\bm{\mathsf{u}}) \colon \bm{\varepsilon}(\bm{\mathsf{v}}) +  \bm{M}(\bm{\mathsf{u}}) \colon \bm{\kappa}(\bm{\mathsf{v}}) + \bm{Q}(\bm{\mathsf{u}}) \cdot \bm{\gamma}(\bm{\mathsf{v}})) \sqrt{a} \dd S.
\end{equation}

\subsection{Volume forces}
For the volume forces (i.e., the third term), we obtain
\begin{align}
    \int_{\Omega} \mathbf{f} \cdot \mathbf{v} \sqrt{g} \ \dd \Omega &\approx \int_{\Omega} \mathbf{f} \cdot (\bm{v} + \xi_3 \bm{\phi}) \sqrt{a} \ \dd \Omega \nonumber \\
    &= \int_S \left( \bm{v} \cdot \int_T \mathbf{f} \dd \xi_3 + \bm{\phi} \cdot \int_T \xi_3 \mathbf{f} \dd \xi_3 \right) \sqrt{a} \dd S \nonumber \\
    &= \int_S \left(\bm{v} \cdot \bm{f}_1 + \bm{\phi} \cdot \bm{m}_1 \right) \sqrt{a} \dd S \nonumber \\
    &= \int_{S} \bm{\mathsf{f}}_1 \cdot \bm{\mathsf{v}} \sqrt{a} \dd S \label{eq: body_forces}
\end{align}
where $\bm{\mathsf{f}}_1 = (\bm{f}_1, \bm{m}_1)$ with
\begin{equation*}
    \bm{f}_1=\int_T \mathbf{f} \dd \xi_3 \quad \text{and} \quad \bm{m}_1=\int_T \xi_3 \mathbf{f} \dd \xi_3
\end{equation*}
representing distributed forces and moments applied to the mid-surface $S$, respectively.

\subsection{Surface tractions}
Since $\partial \Omega_N = \partial \Omega_N^- \cup \partial \Omega_N^+ \cup \partial \Omega_N^0$, the surface tractions give rise to three distinct contributions: those from the ``bottom'' $\partial \Omega_N^- = S \times \left\{-\frac{\tau}{2}\right\}$ and ``top'' $\partial \Omega_N^+ = S \times \left\{\frac{\tau}{2}\right\}$ surfaces of the shell as well as from its ``sides'' $\partial \Omega_N^0 = \partial S_N \times T$. Denoting $\mathbf{h}^-$, $\mathbf{h}^+$ and $\mathbf{h}^0$ these respective contributions, the fourth and last term in \eqref{eq: 3D_elasticity_curvilinear_app} becomes
\begin{align}
    \int_{\partial \Omega_N} \mathbf{h} \cdot \mathbf{v} \sqrt{g} \ \dd \partial \Omega &\approx
    \int_S \left(\mathbf{h}^- \cdot (\bm{v} - \frac{\tau}{2}\bm{\phi}) + \mathbf{h}^+ \cdot (\bm{v} + \frac{\tau}{2}\bm{\phi})\right)  \sqrt{a} \ \dd S + \int_{\partial S_N} \left(\bm{v} \cdot \int_T \mathbf{h}^0 \dd \xi_3 + \bm{\phi} \cdot \int_T \xi_3 \mathbf{h}^0 \dd \xi_3\right) \sqrt{a} \dd \partial S \nonumber \\
    &= \int_S \left((\mathbf{h}^++\mathbf{h}^-) \cdot \bm{v} + \frac{\tau}{2}(\mathbf{h}^+-\mathbf{h}^-) \cdot \bm{\phi}\right) \sqrt{a} \ \dd S + \int_{\partial S_N} \left(\bm{h} \cdot \bm{v} + \bm{n} \cdot \bm{\phi}\right)\sqrt{a} \dd \partial S \nonumber \\
    &= \int_S \bm{\mathsf{f}}_2 \cdot \bm{\mathsf{v}} \sqrt{a} \dd S + \int_{\partial S_N} \bm{\mathsf{h}} \cdot \bm{\mathsf{v}} \sqrt{a} \dd \partial S \label{eq: surface_tractions}
\end{align}
where $\bm{\mathsf{f}}_2 = (\bm{f}_2, \bm{m}_2)$ and $\bm{\mathsf{h}}=(\bm{h}, \bm{n})$ whose components are
\begin{align*}
    \bm{f}_2 &= \mathbf{h}^++\mathbf{h}^- \quad \text{and} \quad \bm{m}_2 = \frac{\tau}{2}(\mathbf{h}^+-\mathbf{h}^-), \\
    \bm{h} &= \int_T \mathbf{h}^0 \dd \xi_3 \quad \text{and} \quad \bm{n}=\int_T \xi_3 \mathbf{h}^0 \dd \xi_3.
\end{align*}

Finally, combining \eqref{eq: inertia_forces}, \eqref{eq: internal_forces}, \eqref{eq: body_forces} and \eqref{eq: surface_tractions}, we obtain the two-dimensional shell variational formulation
\begin{equation*}
    \int_{S}(\ddot{\bm{\mathsf{u}}} \cdot W \bm{\mathsf{v}}) \sqrt{a} \dd S + \int_{S} (\bm{N}(\bm{\mathsf{u}}) \colon \bm{\varepsilon}(\bm{\mathsf{v}}) +  \bm{M}(\bm{\mathsf{u}}) \colon \bm{\kappa}(\bm{\mathsf{v}}) + \bm{Q}(\bm{\mathsf{u}}) \cdot \bm{\gamma}(\bm{\mathsf{v}})) \sqrt{a} \dd S = \int_{S} \bm{\mathsf{f}} \cdot \bm{\mathsf{v}} \sqrt{a} \dd S + \int_{\partial S_N} \bm{\mathsf{h}} \cdot \bm{\mathsf{v}} \sqrt{a} \dd \partial S
\end{equation*}
where $\bm{\mathsf{f}} = \bm{\mathsf{f}}_1 + \bm{\mathsf{f}}_2 = (\bm{f}, \bm{m})$.

%\bibliography{Bibliography}

\end{document}